\documentclass[11pt]{amsart}
\usepackage{enumitem}
\usepackage{times}
\usepackage{url}
\usepackage[leqno]{amsmath} 
\usepackage{cancel}
\usepackage{amssymb,mystyle,enumitem}
\usepackage{MnSymbol,wasysym}
\usepackage{tikzsymbols}
\usepackage{tikz-cd}

\usepackage[colorlinks=true,linkcolor=blue]{hyperref}
\usepackage{placeins}

  \usetikzlibrary{decorations.markings}
\tikzset{double line with arrow/.style args={#1,#2}{decorate,decoration={markings,%
mark=at position 0 with {\coordinate (ta-base-1) at (0,1pt);
\coordinate (ta-base-2) at (0,-1pt);},
mark=at position 1 with {\draw[#1] (ta-base-1) -- (0,1pt);
\draw[#2] (ta-base-2) -- (0,-1pt);
}}}}

\usepackage{calc}
\newsavebox\CBox
\newcommand\hcancel[2][0.5pt]{%
  \ifmmode\sbox\CBox{$#2$}\else\sbox\CBox{#2}\fi%
  \makebox[0pt][l]{\usebox\CBox}%
  \rule[0.5\ht\CBox-#1/2]{\wd\CBox}{#1}}

  \renewcommand{\gg}{{\mathfrak{g}}}
  \newcommand{\bsfT}

\author{Stephen DeBacker}
\address{University of Michigan\\
Ann Arbor, MI 48109-1043, USA}
\email{smdbackr@umich.edu}

 \makeatletter
\let\@wraptoccontribs\wraptoccontribs
\makeatother

\contrib[with appendices by]{Ram Ekstrom and Mitya Boyarchenko, Stephen DeBacker, Anna Spice,  Loren Spice, \&  Cheng-Chiang Tsai}


  \date{\today}

  \title{Totally Ramified Maximal Tori and Bruhat-Tits theory}

\begin{document}

\begin{abstract}
  Suppose $k$ is a nonarchimedean local field, $K$ is a maximally unramified extension of $k$, and $\bG$ is a connected reductive $k$-group.     If $\bT$ is a $K$-minisotropic maximal $k$-torus in $\bG$, then we use Bruhat-Tits theory to describe the stable classes in the $\bG$-orbit of $\bT$,  the rational classes in the $\bG$-orbit of $\bT$, and the  $k$-embeddings, up to rational conjugacy, into $\bG$ of $\bT$.  We also provide, via Bruhat-Tits theory, a complete and explicit description of: the rational conjugacy classes of $K$-minisotropic maximal tame $k$-tori in $\bG$; the stable classes of $K$-minisotropic maximal tame $k$-tori in $\bG$; and the $k$-embeddings, up to rational conjugacy, into $\bG$ of a $K$-minisotropic maximal tame $k$-torus of $\bG$.
\end{abstract}

\maketitle

\section*{Introduction}

Suppose $k$ is a nonarchimedean local field, $K$ is a maximally unramified extension of $k$, and $\bG$ is a connected reductive $k$-group.   If $\bT$ is a maximal $k$-torus in $\bG$, then $\bT^K$\!,  the maximal $K$-split torus in $\bT$, is defined over $k$ and $\bT$ is a maximal $K$-minisotropic $k$-torus in $\bL = C_\bG(\bT^K)$, the centralizer in $\bG$ of $\bT^K$.   The group $\bL$ is an unramified twisted Levi subgroup of $\bG$; that is, $\bL$ is a $k$-group that occurs as the Levi component of a parabolic $K$-subgroup of $\bG$.  Consequently, an approach to parameterizing the rational conjugacy classes of  maximal tori in $\bG$ is to
\begin{itemize}
    \item parameterize the rational conjugacy classes of unramified twisted Levi subgroups of $\bG$ and
    \item for each  unramified twisted Levi subgroup $\bL$ of $\bG$ parameterize the $\bL(k)$-conjugacy classes of $K$-minisotropic maximal $k$-tori in $\bL$.
\end{itemize}
This paper takes up the latter problem.  The former problem is the subject of~\cite{debacker:unramified}.
Future work will take up the problem of parameterizing, via Bruhat-Tits theory, the rational classes of  tame tori in a connected reductive $k$-group.

In an attempt to improve the exposition,  we will assume in this introduction  that $\bG$ is semisimple and $k$-split.

Recall that a $K$-torus $\bT$ in $\bG$ is called $K$-anisotropic provided that $\X^*(\bT)$ does not contain non-trivial characters that are defined over $K$.   Such tori are sometimes also referred to as totally ramified $K$-tori.    Let $G$ denote the group of $K$-rational points of $\bG$.

We begin by studying the set of $G$-conjugacy classes of $K$-anisotropic maximal $K$-tori in $\bG$.  Whereas unramified tori interact very nicely with Bruhat-Tits theory (for example, the apartment of a maximal unramified torus in $\bG$ is always an affine subspace of the building of $G$ which decomposes as a union of facets for $G$), the same is very much not true for totally ramified tori.   However, as shown in Section~\ref{sec:pointswild}, we can always associate to $\bT$, a totally ramified maximal $K$-torus in $\bG$, a unique point $x_T$ in the reduced Bruhat-Tits building of $G$.   Moreover, this point $x_T$ has the remarkable property that its $G$-orbit intersects the closure of each alcove in the reduced building  of $G$ exactly once.
We then combine results of Haines and Rapoport~\cite{haines-rapoport:parahoric}  and Kottwitz~\cite[Section 7]{kottwitz:isocrystalII}  to establish the following fundamental  fact: the stabilizer in $G$ of $x_T$ can be written as the product of $T$ and the parahoric subgroup attached to $x_T$.

Without some assumptions on $k$ and $\bG$, we cannot say much about $x_T$ outside of groups of type $A_n$.   However, under the assumption that the characteristic of $\ff$ does not divide the order of the Weyl group, in Section~\ref{sec:tamepoints}  we explicitly describe the point attached to our totally ramified maximal $K$-torus $\bT$ of $\bG$.  This can be done because, under these tameness assumptions, the $G$-conjugacy classes of totally ramified maximal $K$-tori are parameterized by the elliptic conjugacy classes in the Weyl group.   On the other hand, elliptic conjugacy classes in the Weyl group have unique (normalized) Kac coordinates, and  these coordinates determine, up to $G$-conjugacy,  the point $x_T$ (see Lemma~\ref{lem:kacinbarC}).

Having described, albeit under modest tameness assumptions, the complete set of $G$-conjugacy classes of totally ramified maximal $K$-tori, in Section~\ref{sec:existenceKktori} we   turn to the question of when such a $G$-conjugacy class  contains a torus that is defined over $k$.  We show that this will happen if and only if $\dorbit_T$, the $G$-conjugacy class of $\bT$, is stable under the action of $\Gal(K/k)$.

In Section~\ref{sec:existenceintame}  we again invoke a tameness assumption and use the parameterization of Section~\ref{sec:tamepoints} to convert the criterion described above into an easy to check condition about Weyl group conjugacy classes.   We then use this condition to show that (when $\bG$ is $k$-split) every $G$-orbit of totally ramified $K$-tori in $\bG$ contains a torus which is defined over $k$.  See Corollary~\ref{cor:5.4.4} for the precise statement.

When we assume that our totally ramified maximal torus $\bT$ is defined over $k$, we can prove several results about $\dorbit_T^k$, the set of tori in $\dorbit_T$ that are defined over $k$.  In particular, in Section~\ref{sec:resultsonKktori} we establish natural parameterizations of both  the $k$-stable classes in $\dorbit_T^k$ and  the rational classes in $\dorbit_T^k$.  We also establish a natural parameterization of the $k$-embeddings, up to rational conjugacy, of $\bT$ into $\bG$.   Recall that a $k$-embedding of $\bT$ into $\bG$ is a $k$-morphism $f\colon \bT \rightarrow \bG$ for which there exists $g \in G$ such that $f(t) = g t g\inv$ for all $t \in T$.
Along the way to establishing these results, in Section~\ref{subsec:omega} we show that there is a natural identification of a quotient of $T$ with the $\Omega$ group of $G$.  Recall that the $\Omega$ group of $G$ is the quotient of the stabilizer of an Iwahori subgroup by the Iwahori subgroup.

Our main result  is Theorem~\ref{thm:main}; it provides a complete accounting of the rational classes, $k$-stable classes, and $k$-embeddings of all tame totally ramified maximal $k$-tori in $\bG$.   We also work out several examples in Section~\ref{sec:tameexamples}.   In particular, up to $G^{\Fr}$-conjugacy,
we enumerate all of the tame totally ramified $k$-tori in $\bG$ and describe all of their $k$-embeddings into $\bG$ for $\bG$ being  $\SL_n$, unramified $\SU_n$,  $\Sp_4$, $\Gtwo$, or ramified $\SU_3$.
These calculations are quite challenging (for me).  Although I've done all of the calculations multiple times, I am sure that errors remain.  So, be cautious.

Motivated by the examples discussed above, in Section~\ref{sec:isogenous} we look at how the parameterizations introduced in Section~\ref{sec:resultsonKktori}  behave under certain isogenies.

 Finally, in Section~\ref{sec:Kminisotropicoxeter} we establish the existence of a totally ramified torus  for which the associated point in the building is the barycenter of an alcove.  We call such a torus a $K$-minisotropic Coxeter torus, and we explore some of its properties.

The results of Sections~\ref{sec:pointswild} and~\ref{sec:tamepoints} hold under less restrictive hypotheses on $k$; see the discussion at the start of each of these sections.

\section*{Acknowledgements}  I thank Mark Reeder for sharing with me the key ideas of \S\ref{sec:markssec}.  I thank Cheng-Chiang Tsai for providing proofs of Proposition~\ref{prop:parahoric} and Corollary~\ref{cor:goodatKpoints}.  This paper benefited greatly from discussions with Jeffrey Adams, Jeffrey Adler, Ram Ekstrom, Jessica Fintzen, Jacob Haley, Tasho Kaletha, David Kazhdan, Karol Koziol, Paul Levy, Mark Reeder, David Schwein, Loren Spice (see, in particular, Remark~\ref{rem:loren}),  John Stembridge, and Cheng-Chiang Tsai.  It is a pleasure to thank all of these people.  I also thank the   American Institute of Mathematics whose SQuaRE program created a wonderful  environment where this research was partially carried out.

\section{Notation, the Tits group, and some facts about tori}

Suppose $k$ is a local field with non-trivial discrete valuation $\nu$ and residue field $\ff$ of characteristic $p$ and cardinality $q$.
Let $\bG$ denote a connected reductive $k$-group.    Let $\absW$ denote the absolute Weyl group of $\bG$.   Let $\bG'$ denote the derived group of $\bG$, $\bG_{\scon}$ the simply connected cover of $\bG'$, and $\bG_{\ad}$ the adjoint group of $\bG$.
Let $\eta \colon \bG_{\scon} \rightarrow \bG$ denote the composition of the $k$-maps $\bG_{\scon} \rightarrow \bG'$ and $\bG' \hookrightarrow \bG$.  Let $\bZ$ denote the center of $\bG$.

\subsection{Basic notation}
Let $\bar{k}$ denote a fixed separable closure of $k$, let $\nu$ denote the unique extension of $\nu$ to $\bar{k}$, and let $\varpi \in k$ be a uniformizer.  Let $K$ denote the maximal unramified extension of $k$ in $\bar{k}$, and let $\ffc$ denote the residue field of $K$; it is an algebraic closure of $\ff$. Let $R_K$ denote the ring of integers in $K$.   Let $I = \Gal(\bar{k}/K)$ denote the inertial subgroup of $\Gal(\bar{k}/k)$.  Let  $\Fr$ be a topological generator for $\Gal(K/k) \cong \Gal(\bar{k}/k)/I$, which we identify with $\Gal(\ffc/\ff)$.
We suppose that $\Fr\inv(x) = x^q$ for all $x \in \ffc$.   Choose a lift of $\Fr$ to an element, which we will also call $\Fr$, of $\Gal(\bar{k}/k)$.   Let $L$ denote the completion of $K$, and let $R_L$ denote the ring of integers in $L$.

Let $\bA$ denote a maximal $K$-split $k$-torus in $\bG$ that contains a maximal $k$-split torus of $\bG$; such a torus exists and is unique up to rational conjugacy (see~\cite[Theorem~6.1]{prasad:unramified} or in~\cite[Theorem~3.4.1]{debacker:unramified} take an unramified torus corresponding to a pair of the form $(F,\bfT) \in I^m$  with $F$ an alcove).
Since $\bG$ is $K$-quasi-split, there exists a Borel $K$-subgroup $\bB$ that contains $\bA$, and hence $\Ebtorus = C_{\bG}(\bA)$ is a maximally $K$-split $k$-torus in $\bG$ that is contained in $\bB$.   Let
 $\Phi = \Phi(\bG,\bA)$  denote the set of roots of $\bG$ with respect to $\bG$ and $\bA$ and $\Pi = \Pi(\bG,\bA, \bB)$ the set of simple roots with respect to $\bG$, $\bA$, and $\bB$.   Similarly,
 let  $\Eroot = \Phi(\bG,\Ebtorus)$  denote the set of roots of $\bG$ with respect to $\bG$ and $\Ebtorus$ and $\Esimple = \Pi(\bG,\Ebtorus, \bB)$ the set of simple roots with respect to $\bG$, $\Ebtorus$, and $\bB$.

Let $\Esplits$ denote the splitting field over $K$ of $\Ebtorus$.

If $\bH$ is an algebraic $K$-group, then we let $H$ denote the group of $K$-points of $\bH$, and we let $\tilde{H}$ denote the group of $\Esplits$ points of $\bH$.
If $\bsH$ is an $\ffc$-group, we will often denote the group of $\ffc$-points of $\bsH$ by $\bsH$ as well.

We let $\tame$ denote the maximal tame extension of $k$ in $\bar{k}$, and we let $\sigma$ denote a topological generator of $\Gal(\tame/K)$.

If $\mathcal{G}$ is a group and $x,y \in \mathcal{G}$, then $\lsup{x}y = xyx\inv$.   
If $\tau$ is an automorphism of $\mathcal{G}$, then
two elements $x, y \in \mathcal{G}$ are said to be $\tau$-conjugate (in $\mathcal{G}$) provided that there exists $g \in \mathcal{G}$ such that $g\inv x \tau(g) = y$.

If $X$ and $Y$ are sets, $f \colon X \rightarrow Y$ is a function, $A \subset X$, and $B \subset Y$,  then $f[A] \subset Y$ denotes the image of $A$ under $f$ while $f\inv[B] \subset X$ denotes the preimage of $B$.

If $F$ is a field, $\bH$ is an $F$-group,  and $\tau$ is an $F$-automorphism of $\bH$, then we let 
$\Fix_\tau(\bH)$ denote the $\tau$-fixed points in $\bH$ and we let $\bH^\tau$ denote $\Fix_\tau(\bH)^\circ$, the connected component of $\Fix_\tau(\bH)$.  Both $\Fix_\tau(\bH)$ and $\bH^\tau$  are  $F$-groups.   An $F$-homomorphism between reductive $F$-groups $\bH$ and $\bH'$ is called an isogeny provided that it is surjective with finite central kernel.

\subsection{Notation for Bruhat-Tits theory} \label{sec:BTnotation}
Let $\Psi = \Psi(\bG, \bA, K, \nu)$ denote the set of affine roots  of $\bG$ with respect to  $\bA$, $K$, and $\nu$. For $\psi \in \Psi$ we let $\dot{\psi} \in \Phi$ denote its gradient. The elements of $\Psi$ define the facet structure on the apartment $\AA(A) = \AA(\bA, K)$.  Fix a $\Fr$-stable alcove $C$  in   $\AA(A)$ and let $\Delta = \Delta(\bG, \bA, K, \nu, C)$ denote the set of simple affine roots in $\Psi$ determined by $C$.  Without loss of generality, we assume that $\Pi = \Pi(\bG,\bA, \bB)$ is a subset of $\{ \dot{\psi} \, | \, \psi \in \Delta\}$.

For each facet $H$ in the Bruhat-Tits building $\BB(G) = \BB(\bG,K)$ of $G$ we let $G_{H,0}$ denote the parahoric subgroup attached to $H$ and we let $G_{H,0^+}$ denote the pro-unipotent radical of $G_{H,0}$.  The  quotient $G_{H,0}/G_{H,0^+}$ is the group of $\ffc$-points of a connected reductive $\ffc$-group which we denote by $\bfG_H$.

Suppose $H$ is a facet  in $\AA(\bA, K)$ such that $H \subset \bar{C}$.
Let $\bfA = \bfA_H$  denote the torus in $\bfG_H$ corresponding to the image of the parahoric subgroup $A_0 = A \cap G_{H,0} $ of $A$  in $\bfG_H$.  (We may drop the subscript $H$ since these tori are independent of the choice of $H$ in $\AA(\bA,K)$.)  Let $\bfB^C_H$ denote the Borel subgroup of $\bfG_H$ corresponding to $G_{C,0}/G_{H,0^+}$.  Note that $\bfA$ is a maximal  $\ff$-torus in $\bfB^C_H$, and hence in $\bfG_H$.   Note that $\X^*(\bA)$ may be canonically identified with $\X^*(\bfA_H)$.

Let $F_0$ denote the unique facet contained in the closure of $\bar{C}$ such that (a) $F_0$ is special and (b) the set of simple roots determined by $\bfG_{F_0}$, $\bfA$, and $\bfB^C_{F_0}$ corresponds to $\Pi$. See~\cite[Section 2.10]{prasad-raganuthan:I}  and~\cite[Section 7.1]{prasad-raganuthan:II} for the existence of such a facet when $\bG$ is  absolutely simple and simply connected; since the existence of such a minimal facet is independent of isogeny class, the general case follows by considering the almost simple factors of $\bG$. The uniqueness of such a facet follows from condition (b).  Note that $F_0$ will be a special facet over the splitting field of $\Ebtorus$, and so is absolutely special in $\BB(\bG,K)$; see~\cite[Section 5]{haines-richarz:smoothness} or~\cite[Appendix G]{kaletha:supercuspidal}
for more about absolutely special facets.

\subsection{A realization of the Tits group}  \label{sec:realization}

Recall that $\Esplits$ denotes the splitting field of $\Ebtorus$ over $K$.   Set $\EGal = \Gal(\Esplits/K)$ and $\Egroup=\bG(\Esplits)$, $\EbtorusErat = \Ebtorus(\Esplits)$, and  $\EBorel = \bB(\Esplits)$.  We may and do identify $\absW $ with $N_{\Egroup}(\EbtorusErat)/\EbtorusErat$.  For a facet $F \subset \BB(\Egroup)$, let $\Equotient_{F}$ denote the connected reductive $\ffc$-group whose group of $\ffc$-rational points is $\Egroup_{F,0}/\Egroup_{F,0^+}$.
Let $\bshA$  denote the $\ffc$-torus in $\Equotient_{F_0}$ whose group of $\ffc$-points coincides with the  image of the parahoric subgroup $\EbtorusErat_0 = \EbtorusErat \cap \Egroup_{F_0,0}$  of $\EbtorusErat$ in $\Equotient_{F_0}(\ffc)$.  Similarly, Let $\tilde{\bfB}$  denote the Borel $\ffc$-group in $\Equotient_{F_0}$ whose group of $\ffc$-points coincides with the  image of $\EBorel \cap \Egroup_{F_0,0}$ in $\Equotient_{F_0}(\ffc)$.
     Choose a pinning $(\bG, \Ebtorus, \bB, \{X_a\}_{a \in \Esimple})$ that is $\EGal$-stable and compatible with $F_0$.   When we say the pinning is compatible with $F_0$, we mean that  $X_a \in \Lie(\Egroup_{F_0,0})$ for all $a \in \Esimple$ and if $\bar{X}_a$ denotes the image of $X_a$ in $\Lie(\Equotient_{F_0})$, then  $(\Equotient_{F_0}, \bshA, \tilde{\bfB}, \{\bar{X}_a\}_{a \in \Esimple})$  is a pinning of $(\Equotient_{F_0}, \bshA ,  \tilde{\bfB})$.

For $a \in \Esimple$ we let $\bG_a$ denote the $\Esplits$-subgroup of $\bG$ generated by the root groups $\bU_a$ and $\bU_{-a}$ in $\bG$ and set $\Ebtorus_a = \bG_a \cap \Ebtorus$.    Since for all nontrivial $u \in \bU_a(\Esplits)$, the intersection $u \bU_{-a}(\Esplits) u \cap N_{\bG_a(\Esplits)}(\Ebtorus_a)$ has cardinality one, the pinning $(\bG, \Ebtorus, \bB, \{X_a\}_{a \in \Esimple})$  uniquely identifies an element $n_a \in N_{\bG_a(\Esplits)}(\Ebtorus_a) \leq N_{\Egroup}(\Ebtorus) $.  We have $n_a^2 = \check{a}(-1) \in \EbtorusErat$ where $\check{a} \in \X_*(\Ebtorus)$ is the coroot associated to $a$.  Suppose $\tau \leq \EbtorusErat$ denotes the elementary abelian two-group  generated by $\{ \check{a}(-1) \, | \, a \in \Esimple\}$ and $\titsW$ denotes the (finite) subgroup of $N_{\Egroup}(\Ebtorus)$ generated by $\{ n_{a} \, | \, a \in \Esimple\}$.
Note that since our pinning is $\EGal$-stable, we have that $\EGal$ acts on both $\tau$ and $\titsW$.
If $\pi \colon N_{\Egroup}(\Ebtorus) \rightarrow \absW$ denotes the usual projection, then thanks to~\cite{tits:normal} we have an exact sequence
$$1 \longrightarrow \tau \longrightarrow \titsW \stackrel{\pi}{\longrightarrow} \absW \longrightarrow 1.$$
 of finite groups with $\EGal$-action. Note that $\tau$ is a subgroup of $\Egroup_{F_0,0}$, and, since our pinning is compatible with $F_0$, we have that  $\titsW$ is also a subgroup of $\Egroup_{F_0,0}$.   The group $\titsW$ is called the Tits group (with respect to our chosen pinning).

\begin{remark} \label{rem:1.3.1}
Since the root groups $\bU_a$ for $a \in \Esimple$ belong to $\bG'$, we conclude that $\titsW \leq \Egroup'_{F_0,0} \leq \Egroup'$.
\end{remark}

\subsection{Facts about tori}
We recall some facts about tori that will be used later.  Recall that every $K$-torus in $\bG$ splits over a separable extension~\cite[Corollary~12.19]{milne:algebraic}.

\subsubsection{On the structure of parahoric subgroups of tori}

Suppose $\bT$ is a  $k$-torus.  Since $\bT$ is abelian and any two maximal $K$-split tori in $\bT$ are $T = \bT(K)$-conjugate, there is a unique maximal $K$-split $K$-torus, $\bT^K$,  in $\bT$.  Since it is unique and $\bT$ is a $k$-torus, $\bT^K$ is a $k$-torus as well.   Recall that $R_K$ denotes the ring of integers in $K$.   The maximal bounded subgroup of $T^K$ is $T^K_0 = \bT^K(R_K)$, and  $T^K_0$ is \emph{the} parahoric subgroup of $T^K$.

Let $T_0$ denote the parahoric subgroup of $T$ and let $T_{0^+}$ denote its pro-unipotent radical.  The quotient $T_0/T_{0^+}$ is isomorphic to  the $\ffc$-points of an $\ff$-torus $\overline{\bfT}$.   Moreover, $T^K_0 = T_0 \cap T^K$ and the image of $T^K_0$ in $T_0/T_{0^+}$ may be identified with  the $\ffc$-points of  $\overline{\bfT}$.   From this we conclude that $T_0 = T^K_0 T_{0^+}.$

Since $T_0$, $T^K_0$, and $T_{0^+}$ are $\Fr$-stable and $\cohom^1(\Fr,T^K_0 \cap T_{0^+})$ is trivial, we have $T_0^{\Fr} = ({T^{K}_0})^{\Fr} \, T_{0^+}^{\Fr}$.

\begin{remark}
Let $\ttt$ denote the Lie algebra of $\bT$.  If $\bG$ is $K$-split, we can (canonically) write $\ttt = \bar{\ttt} \oplus \bar{\ttt}^{\perp}$ where $\bar{\ttt}$ is the Lie algebra of ${\bT}^K$ and $\bar{\ttt}^{\perp} = (C_{\gg}(\bar{\ttt}))' \cap \ttt$, where $(C_{\gg}(\bar{\ttt}))'$ denotes the derived Lie algebra of $C_{\gg}(\bar{\ttt})$.   If $\bG$ is not $K$-split, then we need to replace $\bar{\ttt}$ with the center of $C_{\gg}(\bar{\ttt})$.
\end{remark}

\subsubsection{Splitting fields of tori in quasi-split groups}

Suppose $F$ is a field and $\bH$ is an $F$-quasi-split group.   Let $(\bB,\bS)$ be a Borel-torus pair for $\bH$; so $\bB$ is a Borel $F$-subgroup of $\bH$ and $\bS$ is a maximal $F$-torus of $\bH$ that is contained in $\bB$.

\begin{lemma}  \label{lem:tamealltheway}    Suppose $\bT$ is a maximal $F$-torus in $\bH$.   If $E$ is the splitting field of $\bT$, then $\bS$ is $E$-split.
\end{lemma}

\begin{proof}
  Since $\bT$ is $E$-split, there exists a Borel $E$-subgroup $\bB'$ of $\bH$ such that $\bT \leq \bB'$.  Since all Borel $E$-subgroups of $\bG$ are $\bH(E)$-conjugate~\cite[Theorem~20.9 (i)]{borel:linear}, there exists $h \in \bH(E)$ such that $\lsup{h}\bT \leq \bB$.  Since $\lsup{h}\bT$ and $\bS$ are  maximal $E$-tori in $\bB$, they are $\bB(E)$-conjugate~\cite[Proposition~20.5]{borel:linear}.   Consequently, since $\bS$  is $\bH(E)$-conjugate to the $E$-split torus $\bT$, we conclude that $\bS$ is $E$-split as well.
\end{proof}

\section{Points in the reduced building attached to \texorpdfstring{$K$-minisotropic}{K-minisotropic} maximal tori}  \label{sec:pointswild}

In this section we loosen our restrictions on $k$:  it can be any complete field with nontrivial discrete valuation $\nu$ and perfect residue field $\ff$ such that $K$ is strictly Henselian and has cohomological dimension $\leq 1$.

Recall from \S\ref{sec:BTnotation} that $C$ is
an alcove in the apartment $\AA(A)$ in $\BB(G)$.
Let $C'$ denote  the image of $C$ in $\BB^{\red}(G)$, the reduced building of $G$.
Similarly, $\AA'(A)$ will denote the image of $\AA(A)$ in $\BB^{\red}(G)$.

\begin{defn}
A $K$-torus $\bT$ in $\bG$ is said to be $K$-minisotropic provided that $\bX^*(\bT/\bZ)^{\Gal(\bar{k}/K)}$ is trivial.
\end{defn}

In this section we show that to every $K$-minisotropic maximal torus in $\bG$ we can associate a unique point in the closure of the alcove $C'$.  We then describe the stabilizer of this point in Lemma~\ref{lemma:KHR} (see also Remark~\ref{rem:KHR}).

Under the assumption that the torus splits over a tame extension, we describe this point explicitly in Section~\ref{sec:tamepoints}.

\subsection{Results on \texorpdfstring{$G$-conjugacy}{G-conjugacy}}  \label{sec:corr}

  An element $\gamma \in \bG(\bar{k})$  is called \emph{semisimple} provided that it belongs to a maximal torus in $\bG$.   The element $\gamma$ is said to be \emph{strongly regular semisimple} provided that  there exists a maximal torus $\bT$ in $\bG$ such that
  \begin{itemize}
      \item $\gamma \in \bT(\bar{k})$,
      \item $\alpha(\gamma) \neq 1$ for all $\alpha \in \Phi(\bG, \bT)$, and
      \item the stabilizer of $\gamma$ in the Weyl group of $\bT$ is trivial.
  \end{itemize}
  The set  $\bG^{\srss}(\bar{k})$ of strongly regular semisimple elements in $\bG(\bar{k})$  is both dense and open.   In fact, for any maximal $k$-torus $\bT$ of $\bG$ the   set $\bT(\bar{k}) \cap \bG^{\srss}(\bar{k})$ of strongly regular semisimple elements in $\bT(\bar{k})$  is both dense and open.  Since $\bT$ is unirational, this implies that $\bT(k)$ contains a strongly regular semisimple element.   If $\gamma \in \bG^{\srss}(\bar{k})$, then the centralizer of $\gamma$ is a maximal torus in $\bG$.    An element that satisfies the first two conditions, but not the third, is said to be \emph{regular semisimple}.

\begin{lemma}  \label{lem:gside}
 Suppose $\bH$ is a connected reductive $K$-group.  Two 
 elements of $H^{\srss} = \bH^{\srss}(K)$ are $H$-conjugate if and only if
 they are $\bH(\bar{k})$-conjugate.
 \end{lemma}

 \begin{proof}
 Suppose $x,y \in H = \bH(K)$ are strongly regular semisimple.   It will be enough to show that if there exists $g \in \bH(\bar{k})$ for which $\lsup{g}x = y$, then there exists $h \in H$ for which $\lsup{h}x = y$.   Since $x$ is strongly regular semisimple, its centralizer is a maximal $K$-torus in $\bH$,  call it $\bT$.   For all $\gamma \in \Gal(\bar{k}/K)$ we have $\lsup{\gamma(g)\inv g} x = x$, hence $g\inv \gamma(g) \in \bT(\bar{k})$.  Thanks to~\cite[Section 8.6]{borel-springer:rationality}, we have that $\cohom^1(K,\bT)$ is trivial.   Hence there exists $t \in \bT(\bar{k})$ such that $\gamma(gt) = gt$ for all $\gamma \in \Gal(\bar{k}/K)$.  Set $h = gt$.
\end{proof}

\begin{remark}[Loren Spice]  \label{rem:loren}
Suppose that $k^\text{alg}/\bar k$ is an algebraic closure of $\bar{k}$.   Two  semisimple elements $\gamma$ and $\gamma'$ of  $\bG(\bar k)$ are $\bG(k^\text{alg})$-conjugate if and only if they are $\bG(\bar{k})$-conjugate.  Indeed, since every $\bar{k}$-torus is $\bar{k}$-split~\cite[Corollary~12.19]{milne:algebraic}, we may, and do, assume, after conjugating by an element of $\bG(\bar{k})$, that both $\gamma$ and $\gamma'$ belong to the group of $\bar{k}$-points of a $\bar{k}$-split, maximal torus $\bS$ in \(\bG_{\bar k}\).
Suppose that \(g \in \bG(k^\text{alg})\) satisfies \(\Int(g)\gamma = \gamma'\).  Since both $\bS$ and  \(\ \Int(g\inv)\bS \) are maximal $k^\text{alg}$-split tori in $\bH := C_{\bG}(\gamma)^\circ$, there exists $h \in \bH(k^\text{alg})$ such that $\Int(g\inv)\bS  = \Int(h)\bS$.
Then \(\Int(gh)\bS_{k^\text{alg}} = \bS_{k^\text{alg}}\), so that \(gh\) belongs to \(N_\bG(\bS)(k^\text{alg})\).  Since \(N_\bG(\bS)\) is smooth with identity component \(C_\bG(\bS) = \bS\), we have that there is some \(n \in N_\bG(\bS)(\bar k)\) whose image in the Weyl group \((N_\bG(\bS)/\bS)(k^\text{alg})\) is the same as the image of \(gh\); i.e., so that \(n\) belongs to \(gh\bS(k^\text{alg})\).  In particular, \(\Int(n)\gamma\) equals \(\Int(gh)\gamma = \Int(g)\gamma = \gamma'\).
 \end{remark}

 \begin{defn}
Suppose $\bT_i$ is a maximal $K$-torus in $\bG$.   We say that $\bT_1$ is \emph{$K$-stably conjugate} to $\bT_2$ in $\bG$ provided that  there exists $h \in \bG(\bar{k})$ such that $\Ad(h) \colon \bT_1 \rightarrow \bT_2$ is a $K$-morphism.
\end{defn}

 \begin{cor}  \label{cor:torusconj}
  Suppose $\bH$ is a connected reductive $K$-group.  Two maximal $K$-tori in $\bH$ are $H$-conjugate if and only if they are $K$-stably-conjugate in $\bH$.
 \end{cor}

 \begin{proof}
If $\bT$ is a maximal $K$-torus in $\bH$, then $T$ contains a strongly regular semisimple element.  The result follows from Lemma~\ref{lem:gside}.
 \end{proof}

\begin{corollary}  \label{cor:adwalk}
Two maximal $K$-tori in $\bG$ are $G$-conjugate if and only if the corresponding maximal $K$-tori in $\bG_{*}$ are $G_{*}$-conjugate.   Here $*$ is either $\scon$ or $\ad$.
\end{corollary}

\begin{proof}
The canonical $k$-maps $  \bG \rightarrow \bG_{\ad}$ and $ \bG_{\scon} \rightarrow \bG_{\ad}$
establish natural bijective correspondences among (a) the set of maximal $K$-tori in $\bG$, (b) the set of maximal $K$-tori in $\bG_{\ad}$, and (c) the set of maximal $K$-tori in $\bG_{\scon}$.
Moreover, under these correspondences two maximal $K$-tori in $\bG$ are $K$-stably-conjugate if and only if the corresponding maximal $K$-tori in $\bG_{*}$ are $K$-stably-conjugate.   Here $*$ is either $\scon$ or $\ad$.

The result follows from Corollary~\ref{cor:torusconj}.
\end{proof}

 If we assume that $K$ has characteristic zero, then there are similar results on the Lie algebra level.

\begin{lemma}
Suppose $K$ has characteristic zero.
 Suppose $\bH$ is a connected reductive $K$-group with Lie algebra $\bhh$.  Two regular semisimple elements of $\bhh(K)$ are $H$-conjugate if and only if they are $\bH(\bar{k})$-conjugate.
 \end{lemma}

 \begin{proof}
The proof is nearly identical to that of Lemma~\ref{lem:gside}.
\end{proof}

   Let $\bgg$ denote the Lie algebra of $\bG$, and denote by $\bgg'$ the Lie algebra of $\bG'$ the derived group of $\bG$.  When $K$ has characteristic zero, we may and do identify
the Lie algebras of $\bG_{\scon}$ and $\bG_{\ad}$ with  $\bgg'$.

\begin{corollary}
\label{cor:conjugacy-derived}
Suppose $K$ has characteristic zero.   Suppose $X_1, X_2 \in \bgg'(K)$ are regular semisimple.   The following are equivalent:
\begin{itemize}
\item $X_1$ is $G$-conjugate to $X_2$.
\item  $X_1$ is $G_{sc}$-conjugate to $X_2$.
\item  $X_1$ is $G_{ad}$-conjugate to $X_2$.   \qed
\end{itemize}
\end{corollary}

Since, in characteristic zero, every maximal $k$-torus in $\bG$ or  $\bG_*$ arises as the centralizer of a regular semisimple element of $\bgg'(k)$, we have that a maximal $k$-torus $\bT$ in  $\bG$ corresponds to the maximal $k$-torus $\bT_{*}$ in $\bG_{*}$ provided that there exists a regular semisimple $X \in \bgg'(k)$ such that $\bT = C_{\bG}(X)$ and $\bT_{*} = C_{\bG_{*}}(X)$.  This correspondence is well defined since $C_{\bG}(C_{\bgg'}(X)) = C_{\bG}(X')$ and $C_{\bG_{*}}(C_{\bgg'}(X)) = C_{\bG_{*}}(X')$ for all regular semisimple $X'$ belonging to the  Cartan subalgebra $C_{\bgg'}(X)(K)$.  Thus, in characteristic zero Corollary~\ref{cor:conjugacy-derived} gives another proof of Corollary~\ref{cor:adwalk}.

\subsection{Points in the reduced building associated to \texorpdfstring{$K$-minisotropic maximal tori of $\bG$}{K-minisotropic maximal tori of G}}   \label{subsec:point}
Suppose $\bT$ is a $K$-minisotropic maximal torus in $\bG$.

Let $T = \bT(K)$.
We may associate to $\bT$ a point $x_T$ in the reduced building of $G$ as follows: if $E$ is a Galois extension of $K$ over which $\bT$ splits and $\AA'$ denotes the apartment of $\bT$ in $\BB^{\red}(\bG,E)$, then the set of $\Gal(E/K)$-fixed points of $\AA'$ in $\BB^{\red}(\bG,E)$ is a point.  If $E/K$ is not tame, it is possible that this point may not lie in $\BB^{\red}(G)$.   However, thanks to the non-positive curvature of buildings, there is a unique closest point, $x_T$, in $\BB^{\red}(G) = \BB^{\red}(\bG,K)$ to the $\Gal(E/K)$-fixed point of $\AA'$.  Moreover, thanks to uniqueness, the point $x_T \in  \BB^{\red}(G)$ is fixed by $T$.
Let $F$ denote the facet in the building of $G$ whose image under the projection to the reduced building contains $x_T$.

\begin{remark} \label{rem:indofisogeny} Since $\BB^{\red}(G)$ is independent of the isogeny type of $\bG'$, we have $\BB^{\red}(G) = \BB(G_{\scon}) = \BB(G_{\ad})$.    Thanks to Corollary~\ref{cor:adwalk},  the $G$-orbit of  $x_T$  is equal to the $G_{*}$-orbit of $ x_T$ where $*$ is either $\scon$ or $\ad$.  That is, the subset $G \cdot x_T$ in $\BB^{\red}(G)$ is independent of the isogeny type of  $\bG'$.
\end{remark}

\begin{lemma}   \label{lem:card}
The set  $G \cdot x_T \cap \bar{C}'$ has cardinality one.
\end{lemma}

\begin{proof}
 Thanks to Corollary~\ref{cor:adwalk},  the $G$-orbit of  $x_T$  is equal to the $G_{\scon}$-orbit of $ x_T$.
 Since $\bar{C}'$ is a fundamental domain for the action of $G_{\scon}$ on $\BB(G_{\scon})$, the result follows.
\end{proof}

\subsection{Some consequences of results  of Haines and Rapoport and of Kottwitz}  \label{sec:someconsequences}

In~\cite[Section 7]{kottwitz:isocrystalII} Kottwitz shows that there is a functorial surjective homomorphism $\kappa_G \colon G \longrightarrow  \X^*(\hat{Z}(\bG)^I)$.    Here $I$ denotes $\Gal(\bar{k}/K)$, and $\hat{Z}(\bG)$ denotes the center of the dual group of $\bG$.   If $\pi_1(\bG)_I$  denotes the $I$-coinvariants of the fundamental group of $\bG$, then   $\X^*(\hat{Z}(\bG)^I) \cong \pi_1(\bG)_I$ and  we have a functorial surjective homomorphism $\kappa_G \colon G \longrightarrow  \pi_1(\bG)_I$  (see~\cite[Section 3.3]{kaletha:regular}).   In \cite[Proposition~3]{haines-rapoport:parahoric}  Haines and Rapoport show that for any facet $F'$ in $\BB(G)$ the restriction of $\kappa_G$ to $\Stab_G(F')$ has kernel $G_{F',0}$.

These results have many interesting consequences.  As an example, which can be vastly generalized as in~\cite[Corollary~3.3.1]{kaletha:regular}, we have $\eta[(G_{\scon})_{x,0}] \leq G_{x,0}$ where $\eta \colon \bG_{\scon} \rightarrow \bG$ denotes the composition $\bG_{\scon} \rightarrow \bG' \hookrightarrow \bG$ and $x$ is any point in $\BB(G_{\scon}) = \BB(G') = \BB^{\red}(G)$.  In this section, we record some additional consequences.

The following result plays a fundamental role throughout the remainder of this paper.

\begin{lemma}  \label{lemma:KHR}   Suppose $\bT$ is a $K$-minisotropic maximal torus in $\bG$.   Let $x_T$ denote the point in the reduced building of $G$ attached to $T$, and recall that $F$ denotes the facet in the building of $G$ whose image under the projection to the reduced building contains $x_T$.   We have
$$\Stab_{G} (F) = T G_{F,0}.$$
\end{lemma}

\begin{rem}   Recall that the building of $T$ does not always embed into that of $G$.
In the proof below we
use that $F$ is $T$-stable.
\end{rem}

\begin{proof}   
For any maximal $K$-torus $\bS$ in $\bG$, the map $\pi_1(\bS)_I \stackrel{\iota}{\rightarrow} \pi_1(\bG)_I$ is surjective.
Consequently, we have that  for all $g \in \Stab_{G} (F)$ there is a $t \in T$ such that $\kappa_G(g) = \iota ( \kappa_T(t))$. We then have $t\inv g \in \ker(\kappa_G)$.  Thus, $t\inv g \in  \Stab_G(F)  \cap \ker(\kappa_G)= G_{F,0}$.
\end{proof}

\begin{rem}   \label{rem:KHR}
This result can be restated as:
$$\Stab_{G} (x_T) = T G_{x_T,0}.$$
\end{rem}

\begin{corollary}  \label{cor:2.3.4} The map $T \rightarrow \Stab_G(F)/G_{F,0}$ is surjective with kernel $\eta[T_{\scon}] T_0$. Here $\bT_{\scon}$ is the torus in $\bG_{\scon}$ corresponding to $\bT$
\end{corollary}

\begin{proof}
The surjectivity of the map follows from Lemma~\ref{lemma:KHR}.
Since taking coinvariants is right exact, the exact sequence
\[1 \longrightarrow \X_*(\bT_{\scon})  \stackrel{\eta_*}{\longrightarrow} \X_*(\bT) \longrightarrow \pi_1(\bG) \longrightarrow 1\]
yields
\[\X_*(\bT_{\scon})_I  {\longrightarrow} \X_*(\bT)_I \longrightarrow \pi_1(\bG)_I \longrightarrow 1.\]
Using the results of Kottwitz and Haines and Rapoport discussed above and the fact that $\pi_1(\bS) = \X_*(\bS)$ for any $K$-torus $\bS$, we conclude that we have an exact sequence
\[T_{sc}/(T_{sc})_0 \stackrel{\bar{\eta}}{\longrightarrow} T/T_0 \longrightarrow \Stab_G(F)/G_{F,0} \longrightarrow 1.\]
The result follows.
\end{proof}

\begin{lemma}  \label{lem:2.3.4}
Suppose $g \in G'$ and $y \in \BB(G')$ is special.  If $g  \in \Stab_{G'}(y)$, then $g \in G'_{y,0}$.
\end{lemma}

\begin{proof}
Let $\bar{g}$ denote the image of $g$ in $G_{\ad}$ and let $\Ebtorus_{\ad}$ denote the maximal $k$-torus in $\bG_{\ad}$ corresponding to $\Ebtorus$.  We first show that $\bar{g} \in (G_{\ad})_{y,0}$.     Since $\bar{g} \in \Stab_{G_{\ad}}(y)$, it is enough to show that $\Stab_{G_{\ad}}(y) \subset (G_{\ad})_{y,0}$.   Suppose $h \in \Stab_{G_{\ad}}(y)$.  Since $(G_{\ad})_{y,0}$ acts transitively on the apartments that contain $y$,  there exists $k \in (G_{\ad})_{y,0}$ such that $\lsup{kh} \Ebtorus_{\ad} = \Ebtorus_{\ad}$; that is, $kh  \in N_{G_{\ad}}(\Ebtorus_{\ad})$.  Since $y$ is special, there exists $j \in (G_{\ad})_{y,0}$ that has the same image in $\absW$ as $kh$.  Thus $j\inv k h \in \EbtorusKrat_{\ad} \cap \Stab_{G_{\ad}}(y)$.   However, from~\cite[Propositions~4.4.3 and 4.4.16]{bruhat-tits:II} we have $\EbtorusKrat_{\ad} \cap \Stab_{G_{\ad}}(y) = (\EbtorusKrat_{\ad})_0 \leq (G_{\ad})_{y,0}$.

Since $(G_{\ad})_{y,0}$ is the kernel of  the restriction to ${\Stab_{y}(G_{\ad})}$ of $\kappa_{G_{\ad}}$, $G'_{y,0}$ is the kernel of the restriction to $\Stab_{y}(G')$ of $\kappa_{G'}$, and
the map from $ \X^*(\hat{Z}(\bG')^I)$   to $ \X^*(\hat{Z}(\bG_{\ad})^I)$
  is injective, by functoriality we conclude that since $\bar{g} \in (G_{\ad})_{y,0}$ we must have  $g \in G'_{y,0}$.
\end{proof}

\begin{remark}  \label{rem:neededlater}
Suppose $\bH$ and $\bL$ are connected semisimple $k$-groups and $\rho \colon \bH \rightarrow \bL$ is an isogeny.   Suppose $x \in \BB(H) = \BB(L)$.
We have that  $\rho$ carries $\Fix_H(x)$ into $\Fix_L(x)$ and $(\res_{H}\rho)\inv[\Fix_L(x)]$, the preimage of $\Fix_L(x)$ under $\res_H\rho$, is a subgroup of $\Fix_H(x)$. Moreover, since the map from $\X^*(\hat{Z}(\bH)^I)$   to $\X^*(\hat{Z}(\bL)^I)$ is injective, using the results of Kottwitz and Haines-Rapoport, we conclude
that $(\res_H\rho)\inv[L_{x,0}]$ is $H_{x,0}$.    That is, for all $h \in \Stab_H(x)$ we have $h \in H_{x,0}$ if and only if $\rho(h) \in L_{x,0}$.
\end{remark}

\section{Points in the reduced building attached to tame  maximal \texorpdfstring{$K$-tori of $\bG$}{K-tori of G}}  \label{sec:tamepoints}

As in Section~\ref{sec:pointswild}, we loosen our restrictions on $k$:  it can be  any complete field with nontrivial discrete valuation $\nu$ and perfect residue field $\ff$ such that either (a) $\ff$ is finite or (b) $k$ is strictly Henselian and quasi-finite.  Note that in the latter case we have that $k = K$.  In either case, we have that $\Gal(\tame/K)$ has a topological generator $\sigma$, $K$ is strictly Henselian,  and  $K$ has cohomological dimension $\leq 1$ (see~\cite[II.3.3.c)]{serre:galois} for case (a) and~\cite[XIII, Props. 3 and 5]{serre:local} for case (b)).  If $\ff$ has characteristic zero, then we set $p = 0$.

A $K$-torus $\bT$ of $\bG$ will be called $\emph{tame}$ provided that $\bT$ is $\tame$-split.  A main result of this section is the explicit identification, up to $G$-conjugacy, of the point $x_T$ attached to a  tame $K$-minisotropic maximal torus of $\bG$.

\subsection{Tame elements in Weyl groups}
We let $x_0 \in \bar{C}'$ denote the image of $F_0$ in $\BB^{\red}(G)$.   (The facet $F_0 \subset \bar{C}$ was introduced in Section~\ref{sec:BTnotation}.)

If $\bG$ contains a tame maximal torus, then from Lemma~\ref{lem:tamealltheway} we know that $\Ebtorus$ splits over a tame extension.   Thus, in this section we assume that $\Esplits$, the splitting field of $\Ebtorus$ over $K$, is a tame extension of $K$.  Under this assumption there exists $\ell' \in \Z_{\geq 1}$ such that $\Esplits = {\tame}^{\sigma^{\ell'}}$, $\Esplits^{\sigma} = K$, and $\EGal$ is the cyclic group $\Gal(\tame/K) / \Gal(\tame / \Esplits)$.  We let $\barsigma$ denote the image of $\sigma$ in this quotient.   Denote by $\Aut(\Esimple)$ the automorphism group of the Dynkin diagram associated to $\Esimple$ and let $\tau_{\barsigma}$ denote the image of $\barsigma$ in $\Aut(\Esimple)$.

\begin{defn}
An element $w \in \absW$ is called \emph{tame} provided that  $p$ does not divide the order of $(w,\tau_{\barsigma})$ in $\absW \rtimes \Aut(\Esimple)$.
An element $n \in \titsW$ will be called \emph{tame} provided that  its image, under the projection map $\titsW \rightarrow \absW$, is tame.
\end{defn}

\begin{rem}
Since the order of $\Aut(\Esimple)$ always divides the order of $\absW$, we conclude from Cauchy's theorem that every element of $\absW$ is tame if and only if $p$ does not divide the order of $\absW$.
\end{rem}

\begin{lemma}
 Suppose $n \in \titsW$.  We have $n$ is tame if and only if $p$ does not divide the order of $n \barsigma$ in $\titsW \rtimes \EGal$.
\end{lemma}

\begin{proof}
Recall that $\Esplits$ is a tame extension of $K$.   The kernel of the projection map from $\titsW \rtimes \EGal$ to $\absW \rtimes \Aut(\Esimple)$ is of the form $\tau \rtimes \langle \barsigma^j \rangle$ where $\tau \subset \EbtorusErat$ is an elementary abelian two group and $j$ divides the order of $\EGal$.
Since $p$ does not divide the order of $\EGal$ and two divides the order of every nontrivial Weyl group, the result follows.
\end{proof}

\begin{defn}   \label{defn:nsplits}
Suppose  $n \in \titsW$ is tame, and let $\ell$ denote the order of $n \barsigma$ in $\titsW \rtimes \EGal$.
Since $p$ does not divide $\ell$, we can form the tame degree $\ell$ extension  $\nsplits = {\tame}^{\sigma^\ell}$  of $K$; note that $\Esplits \leq \nsplits$.
\end{defn}

\subsection{On the \texorpdfstring{$\sigma$-conjugacy classes of tame elements of $\titsW$}{sigma-conjugacy classes of tame elements of W}}

The results of this section will feel familiar to those who have studied conjugacy classes in disconnected groups.  I was influenced by the presentation in~\cite{mohrdieck:conjugacy}.

Let
$\EbtorusErat_0$ denote the parahoric subgroup of $\EbtorusErat$ and $\EbtorusErat_{0^+}$ its pro-unipotent radical. Similarly,  let $\EbtorusKrat_0$ denote the parahoric subgroup of $\EbtorusKrat$ and $\EbtorusKrat_{0^+}$ its pro-unipotent radical.   Since $\Ebtorus$ is $\Esplits$-split, we have $\EbtorusErat_0 = \EbtorusErat \cap \Egroup_{x_0,0}$.

\begin{lemma}  \label{lem:ABsimplyconn}
If $n \in \titsW$ is tame, then there exists a $\Esplits$-Borel-torus pair $(\bA',\bB')$  such that \begin{itemize}
    \item  $n\sigma (\bA',\bB') = (\bA',\bB')$  and  \item $x_0 \in \AA'(\bA',\Esplits)$.
    \end{itemize}
\end{lemma}

\begin{remark}
Since the automorphism $n \sigma$ fixes a Borel-torus pair, it is an example of a quasi-semisimple (or quass) automorphism.
\end{remark}

\begin{proof}
 Recall that $\Esplits$ is a tame extension of $K$.

Choose tame $n \in \titsW$.  Let $\ell$ denote the order of $n \barsigma$ in $\titsW \rtimes \EGal$.
Define the continuous 1-cocycle $\tau \in Z^1(\Gal(\tame/K),\bG(\tame))$ by $$\tau(\sigma^j) = N_j^{\sigma}(n) := n \sigma(n) \sigma^2(n) \cdots \sigma^{j-2}(n)  \sigma^{j-1}(n).$$
Since the surjection $\Gal(\bar{k}/K) \rightarrow \Gal(\tame/K)$ yields $\cohom^1(\Gal(\bar{k}/K), \bG) \cong \cohom^1(\Gal(\tame/K),\bG(\tame))$ (see~\cite[I.5.8]{serre:galois} and use the fact that $\tame$ has cohomological dimension $\leq 1$), the 1-cocycle $\tau$ defines a twisted $K$-group, which we denote by $\tautwist$.  In fact,  since $K$ has cohomological dimension $\leq 1$, the groups $\tautwist$ and $\bG$ are isomorphic as $K$-groups.

Since $(\Ad(n) \circ \sigma)^\ell = \sigma^\ell$, we have $\tautwist(\nsplits) = \bG(\nsplits)$,  $\tautwist(K) = \bG(\nsplits)^{n \sigma}$, and $\tautwist(\Esplits) = \bG(\nsplits)^{\tilde{n} \sigma^{\ell'}}$ where $\tilde{n} = \tau(\sigma^{\ell'})$.    (Recall that $\Esplits = \tame^{\sigma^{\ell'}}$.)  Since ${n} \in \Egroup_{x_0,0}$, we have $\tilde{n} \in \Egroup_{x_0,0}$. In fact, since $\lsup{n\sigma}\tilde{n} = \tilde{n}$, we conclude that $\tilde{n} \in \Stab_{\tautwist(K)}(x_0)$.   From Remark~\ref{rem:1.3.1} we have that $n$, and hence $\tilde{n}$, belongs to $\bG'(\nsplits) = \tautwist'(\nsplits)$ and hence $\tilde{n} \in \tautwist'(K)$. From Lemma~\ref{lem:2.3.4} we conclude that $\tilde{n} \in \tautwist'(K)_{x_0,0} \leq \tautwist(K)_{x_0,0}$.

Since $x_0 \in \BB^{\red}(\bG,\nsplits) = \BB^{\red}(\tautwist,\nsplits)$ is fixed by both $\sigma$ and $n$, it is fixed by $n \sigma$.   Since $\tautwist$ is $K$-quasi-split, we can choose a maximally $K$-split $K$-torus $\tautwistA'$ in $\tautwist$ and a Borel $K$-subgroup $\tautwistB'$ of $\tautwist$ containing $\tautwistA'$ such that $x_0 \in \AA'(\tautwistA',K) = \AA'(\tautwistA',\nsplits)^{n\sigma}$ and the topologically semisimple element $\tilde{n} \in  \tautwist(K)_{x_0,0}$ belongs to $\tautwistA'(K)$.

Since $n \sigma \tautwistA' = \tautwistA'$ and $\tilde{n} \in \tautwistA'(K)$, we have $\sigma^{\ell'} \tautwistA' = \tilde{n}\inv (n\sigma)^{\ell'} \tautwistA' = \tautwistA'$.  Similarly, since $n \sigma \tautwistB' = \tautwistB'$ and $\tilde{n} \in \tautwistA'(K) \leq \tautwistB'(K)$, we have $\sigma^{\ell'} \tautwistB' = \tilde{n}\inv (n\sigma)^{\ell'} \tautwistB' = \tautwistB'$.

Thus, there is a $\Esplits$-Borel-torus pair  $(\bA', \bB')$ in $\bG$ such that $\bA'(\Esplits) = \tautwistA'(\Esplits)$ and  $\bB'(\Esplits) = \tautwistB'(\Esplits)$.   Note that   $n\sigma (\bA',\bB') = (\bA',\bB')$  and   $x_0 \in \AA'(\tautwistA',K) \subset \AA'(\tautwistA',\Esplits) =   \AA'(\bA',\Esplits)$.
 \end{proof}

\begin{cor}  \label{cor:Asharpimproved}
If $n \in \titsW$ is tame, then
$n$ is $\sigma$-conjugate by an element of $\Egroup_{x_0,0}$ to an element of $\EbtorusErat_{0}$.
\end{cor}

\begin{proof}
 Recall that $\Esplits$ is a tame extension of $K$.

Thanks to Lemma~\ref{lem:ABsimplyconn} we can find
a  $\Esplits$-Borel-torus pair $(\bA',\bB')$ in $\bG$ such that  $n \sigma (\bA',\bB') = (\bA',\bB')$   and $x_0 \in \AA'(\Ebtorus,\Esplits) \cap \AA'(\bA',\Esplits)$.
  Choose  $h \in \Egroup_{x_0,0}$ such that $h (\Ebtorus,\bB) = (\bA',\bB')$.   Since $\sigma(\Ebtorus,\bB) = (\Ebtorus,\bB)$ and $n\sigma(\bA',\bB') = (\bA', \bB')$, we then have
 $$h\inv n \sigma(h) (\Ebtorus,\bB) = h\inv n \sigma h (\Ebtorus,\bB) = h\inv n \sigma (\bA',\bB') = h\inv (\bA',\bB') = (\Ebtorus,\bB).$$
 Thus $h\inv n \sigma(h) \in \EbtorusErat \cap \Egroup_{x_0,0} = \EbtorusErat_0$.
\end{proof}

 \begin{lemma}  \label{lem:Asharpmodimproved}
If $n \in \titsW$ is tame, then
$n$ is $\sigma$-conjugate by an element of $\Egroup_{x_0,0} $ to an element of $\EbtorusKrat_0 \EbtorusErat_{0^+} = A_{0} \EbtorusErat_{0^+}$.
\end{lemma}

\begin{proof}
 Recall that $\Esplits$ is a tame extension of $K$.

From Corollary~\ref{cor:Asharpimproved} we know that there exists $h \in \Egroup_{x_0,0}$ such that $h\inv n \sigma(h) \in \EbtorusErat_{0}$.

  We identify $\bfA$ in $\Equotient_{x_0}$ by  $\bfA \leq \bfG_{x_0}  = {\Equotient_{x_0}^{\barsigma}} \leq \Equotient_{x_0}$.
Note that $\bshA$ equals  
$C_{\Equotient_{x_0}}(\bfA)$ and $\bfA$ has finite index in $\Fix_{\barsigma}(\bshA)$.

The map $x \mapsto \barsigma(x)$ defines a finite order automorphism of ${\bshA}$.  Since $\bfA$ has finite index in $\Fix_{\barsigma}(\bshA)$ and $\cohom^1(\EGal,\bfA) \cong \Hom(\EGal,\bfA)$ is finite, we conclude that
  $\Fix_{\barsigma}({\bshA}/\bfA)$ is finite.  Thus,  by~\cite[Theorem~10.1]{steinberg:endomorphism}, we conclude that the map $\bar{a} \mapsto \barsigma(\bar{a})\inv \bar{a}$ from $\bshA/\bfA$ to itself is surjective.  Consequently,  there exists $\bar{y} \in \bshA$ such that, modulo $\bfA$,  $\barsigma(\bar{y})\inv \bar{y}$ is congruent to the image of $h\inv n \sigma(h)$ in $\bshA$. Let $y$ be an element of $\EbtorusErat_0$ that lifts $\bar{y}$. Since $\Ebtorus$ is abelian, we may replace $h$ by $hy$ to conclude that  $h\inv  n \sigma(h) \in  A_0 \EbtorusErat_{0^+}$.
\end{proof}

\begin{lemma}  \label{lem:Asharpdone}
Suppose $n \in \titsW$ is tame and $n \barsigma$ has order $\ell$ in $\titsW \rtimes \EGal$.  If $\xi$ is a primitive  $\ell^{\text{th}}$ root of unity in $K$, then
$n$ is $\sigma$-conjugate by an element of $\Egroup_{x_0,0}$ to an element of $A_{0}$ of the form $\lambda(\xi)$ where $\lambda \in \X_*(\bA)$.
\end{lemma}

\begin{proof}
 Recall that $\Esplits$ is a tame extension of $K$.

Let $\xi$ be a primitive  $\ell^{\text{th}}$ root of unity in $K$ and let $\bar{\xi}$ denote the image of $\xi$ in $\ffc$.

From Lemma~\ref{lem:Asharpmodimproved} we can choose $h \in \Egroup_{x_0,0}$ such that  $h\inv n \sigma(h) \in A_{0}\EbtorusErat_{0^+}$.

  Choose a sequence $0 < r_1 < r_2 < \cdots$ such that for all $s > 0$ we have  $\Egroup_{x_0,s} \neq \Egroup_{x_0,s^+}$ if and only if $s = r_j$ for some $j$.   Here  $\Egroup_{x_0,s}$ and $\Egroup_{x_0,s^+}$ are  Moy-Prasad subgroups~\cite{moy-prasad:jacquet, moy-prasad:unrefined} of $\Egroup$.   For $s > 0$ set $\EbtorusErat_s = \EbtorusErat \cap \Egroup_{x_0,s}$, $\EbtorusErat_{s^+} = \EbtorusErat \cap \Egroup_{x_0,s^+}$ and $\EbtorusErat_{s:s^+} = \EbtorusErat_s / \EbtorusErat_{s^+}$.
   Similarly, define $\EbtorusKrat_{s} =\EbtorusKrat \cap G_{x_0,s}$, $\EbtorusKrat_{s^+}= \EbtorusKrat \cap G_{x_0,s^+}$, and $\EbtorusKrat_{s:s^+} = \EbtorusKrat_s / \EbtorusKrat_{s^+}$.

 For every $j \geq 1$
the linear map $\bar{y} \mapsto (1- \barsigma)\bar{y}$ from the $\ffc$-vector space
$\EbtorusErat_{r_j:r_j^+}/\EbtorusKrat_{r_j:r_j^+}$ to itself is an isomorphism.   Thus
 for every element $x \in \EbtorusErat_{r_j}$ there exists ${y} \in  \EbtorusErat_{r_j}$ such that $\barsigma(y)\inv x y \in  \EbtorusKrat_{r_j} \EbtorusErat_{r_j^+}$.

 Applying this to $h \inv n \sigma(h) \in  A_0  \EbtorusErat_{0^+} = A_0  \EbtorusErat_{r_1}$, we conclude, using the fact that $\Ebtorus$ is abelian, that there exists $y_1 \in \EbtorusErat_{r_1}$  such that if $h_1 = hy_1$  we have $h_1\inv n \sigma(h_1) \in A_0 \EbtorusKrat_{0^+} \EbtorusErat_{r_2}$.    We may repeat this process to conclude that there  exists $y_j \in \EbtorusErat_{r_j}$ such that if $h_j = (h_{j-1})y_j$,   then   $h_j\inv n \sigma(h_j) \in A_0 \EbtorusKrat_{0^+} \EbtorusErat_{r_{j+1}}$.
Letting $j$ go to infinity we conclude that there is an $h_\infty \in \bG(\tilde{L})_{x_0,0}$ such that $h_\infty \inv n \sigma(h_\infty) \in \Ebtorus(L)_{0} =A_{0}\Ebtorus(L)_{0^+}$.
Here $L$ denotes the completion of $K$ and $\tilde{L} = L \otimes_K \Esplits$ denotes the completion of $\Esplits$; note that $\barsigma$ extends (continuously) to an automorphism of $\tilde{L}$ whose fixed point set is $L$.

 Choose $a \in A_{0}\Ebtorus(L)_{0^+}$
 such that  $h_\infty \inv n \sigma(h_\infty) = a$.
 Since the order of $n \barsigma$ in $\titsW \rtimes \EGal$ is $\ell$, we have
$$N_{\ell}^{\barsigma}(n) = n \barsigma(n) \barsigma^2(n) \cdots \barsigma^{\ell-2}(n) \barsigma^{\ell-1}(n) = 1.$$
Thus, since $a = \sigma^j (a)$ for all $j \in \Z_{\geq 1}$ and  $\sigma^\ell(h_\infty) = h_\infty$, we have
 $$a^{\ell} = \prod_{j=0}^{\ell-1} \sigma^j(a) = \prod_{j=0}^{\ell-1} \sigma^j(h_{\infty}\inv n \sigma(h_\infty)) = h_\infty\inv N_{\ell}^{\barsigma}(n) \sigma^{\ell}(h_\infty) =1.$$

 Let $\bar{a}$ denote the image of $a$ in $\bfA$.   
 Since $\bar{a}^\ell = 1$,  there exists $\lambda \in \X_*(\bfA)$ such that $\bar{a} = \lambda(\bar{\xi})$.   Since $\X_*(\bfA) = \X_*(\bA)$, we may and do think of $\lambda$ as  an element of  $\X_*(\bA)$.   Thus, $a = \lambda(\xi) \tilde{a}$ where $\tilde{a} \in \Ebtorus(L)_{0^+}$.   Since $1 = a^\ell = \tilde{a}^\ell$ and $\ell$ is coprime to $p$, we conclude that $\tilde{a} = 1$.   Thus, $a = \lambda(\xi) \in A_0$.

 If $k \neq K$, in which case $L$ may not be $K$, we
 need to show that we can take $h_\infty \in \Egroup_{x_0,0}$.   We proceed as follows.
Choose finite extensions $k' \leq \tilde{k}'$ of $k$ with $k' \leq K$ and $\tilde{k}' \leq \Esplits$ such that  $\Gal(\tilde{k}'/k')$ is isomorphic to $\Gal(\Esplits/K)$,  $\bA$ is a maximal $k'$ split torus in $\bG$, $\Ebtorus$ is a maximal $\tilde{k}'$-split torus in $\bG$,
$x_0 \in \BB(\bG,k')$, we have $h, n \in \bG(\tilde{k}')_{x_0,0}$, $\xi \in k'$, etc.  The entire proof above goes through with $\EbtorusErat$ replaced by $\Ebtorus(\tilde{k}')$, $\EbtorusKrat$ replaced by $\Ebtorus(k')$, etc.  Since $k'$ is complete, we conclude that we have $h_\infty \in \bG(\tilde{k}')_{x_0,0} \leq \Egroup_{x_0,0}$.
  \end{proof}

By taking the reductive quotient, we recover a very special version of the known general result that if $\bar{g} \in \bfG$ such that $\bar{g} \sigma$ has finite order coprime to $p$, then $g$ is $\sigma$-conjugate in $\bfG$ to an element of $\bfA$:

\begin{corollary}
    Suppose $n \in \titsW$ is tame and $n \barsigma$ has order $\ell$ in $\titsW \rtimes \EGal$. Let $\bar{n} \in N_{\bfG}(\bfA)$ denote the image of $n$ in $\bfG$.   If $\xi$ is a primitive  $\ell^{\text{th}}$ root of unity in $\ffc$, then
$\bar{n}$ is $\sigma$-conjugate by an element of $\bfG$ to an element of $\bfA$ of the form $\lambda(\xi)$ where $\lambda \in \X_*(\bfA)$.
\end{corollary}

\begin{proof}  This follows from Lemma~\ref{lem:Asharpdone} by looking at the reductive quotient.
\end{proof}

\begin{remark}  \label{rem:allaboutkac}
 Note that
the element $\lambda \in \X_*(\bA) = \X_*(\bfA)$ encodes what is known as the \emph{Kac coordinates}  or, equivalently, the \emph{Kac diagram} of the image of $\bar{n}\sigma $ in $\Aut(\Lie(\bfG))$.  Note that since $\mu(\xi) = 1$ for all $\mu \in \ell \X_*(\bA)$, the element $\lambda$ cannot be unique.   First introduced by Kac
in~\cite{kac:automorphisms} (see also~\cite[Chapter X, Section~5]{helgason:differential} or~\cite[Chapter 8]{kac:infinite}), the use of Kac coordinates for $p$-adic groups was initiated  in~\cite{reederetal:gradings} and~\cite{reederetal:epipelagic}.   We will have more to say about Kac coordinates in Section~\ref{sec:kaccoordinates}.
\end{remark}

 \subsection{From tame elements of \texorpdfstring{$\titsW$ to  maximal $K$-tori in $\bG$}{W to maximal K-tori in G}}  \label{sec:markssec}
 \label{sec:3.3}

 Mark Reeder is responsible for the key ideas of this section: the definition of the element $g_n$ and how to use $g_n$  to define the maximal $K$-torus $\bT_n$ of $\bG$.  I thank him for explaining these ideas to me; I am responsible for all errors.

 Choose a tame $n \in \titsW$, and suppose $n\barsigma$ has order $\ell_n$ in $\titsW \rtimes \EGal$.

Recall that  $\varpi \in k$ is a uniformizer.   Choose a uniformizer $\pi \in \nsplits$ such that  $\pi^{\ell_n} = \varpi$.   Set $\xi = \sigma(\pi)/\pi$; note that $\xi \in K$ is a primitive $\ell_n^{\text{th}}$ root of unity.
 Thanks to Lemma~\ref{lem:Asharpdone}
there exist $h \in \Egroup_{x_0,0}$ and $\lambda_n \in \X_*(\bA)$ such that $h\inv n \sigma(h) = \lambda_n(\xi)$.  Define $x_n \in \AA'(A)$ by $x_n = x_0  + \lambda_n/\ell_n$, define $g_n \in \bG(\nsplits)_{x_n,0}$ by $g_n = \lambda_n(\pi) h\inv \lambda_n(\pi)\inv$, and define the maximal $\nsplits$-split torus $\bT_n$  of $\bG$ by $\bT = \lsup{g_n} \Ebtorus$.

\begin{remark}
The point $x_n = \lambda_n(\pi) \cdot x_0$ is an absolutely special vertex in $\BB^{\red}(\bG,\nsplits)$.
\end{remark}

\begin{lemma}  \label{lem:thepoint}
The torus $\bT_n$ is defined over $K$.
\end{lemma}

\begin{proof}
We first compute $g_n\inv \sigma(g_n)$:
\begin{equation*}
\begin{split}
g_n\inv \sigma(g_n)&=  (\lambda_n(\pi) h \lambda_n(\pi)\inv) (\sigma(\lambda_n(\pi) h\inv \lambda_n(\pi)\inv)) =   \lambda_n(\pi) h  \lambda_n(\pi)\inv  \lambda_n(\xi) \lambda_n(\pi) \sigma(h)\inv \lambda_n(\pi \xi)\inv \\
& =  \lambda_n(\pi) h  \lambda_n(\xi) \sigma( h)\inv \lambda_n(\pi \xi)\inv  = \lambda_n(\pi) n \lambda_n(\pi \xi)\inv = n  [n\inv \lambda_n(\pi) n  \lambda_n(\pi \xi)] \\
&= n a'
 \end{split}
 \end{equation*}
 where $a' = {\lsup{n\inv}\lambda_n(\pi) }  \lambda_n(\pi \xi) \in \Ebtorus(\nsplits)$.  Consequently, $ \sigma(\bT_n) = \sigma(\lsup{g_n} \Ebtorus) = \lsup{g_n n a'} \Ebtorus = \lsup{g_n} \Ebtorus = \bT_n$, and so $\bT_n$ is defined over $K$.
\end{proof}

 While the maximal $K$-torus $\bT_n$ of $\bG$ depends on many choices, we now show that the $G$-conjugacy class of $\bT_n$ depends only on the $\sigma$-conjugacy class of the image of $n$ in $\absW$.

 Suppose $w \in \absW$. If $w'$ is $\sigma$-conjugate to $w$, then $w'$ and $w$ have the same order in $\absW \rtimes \EGal$.   Thus, if $w$ is tame, then every element in the $\sigma$-conjugacy class of $w$ is tame as well.  Let $\Wtame$ denote the set of tame elements in $\absW$.  Let $W_{\sim_\sigma}$ (resp. $\Wtamesim$) denote the set of $\sigma$-conjugacy classes in $\absW$ (resp. $\Wtame$).

\begin{defn}  We let $\mathcal{C}$ denote the set of $G$-conjugacy classes of maximal $K$-tori in $\bG$.
\end{defn}

 \begin{cor}  \label{cor:welldefined} The assignment $n \mapsto T_n$ from tame elements in $\titsW$ to maximal $K$-tori in $\bG$ induces an injective map
 $$\varphi_\sigma \colon \Wtamesim \rightarrow \mathcal{C}.$$
 \end{cor}

 \begin{proof}

For $w_i \in \Wtame$ with $i \in \{1,2\}$ we  let $n_i \in \titsW$ be in the preimage of $w_i$ under the natural projection $\titsW \rightarrow \absW$.   Recall that to the element $n_i$ we can associate $g_i \in \bG(\tame)$ such that $g_i\inv \sigma(g_i) \Ebtorus = n_i \Ebtorus$.  Let $\dorbit_i \in \mathcal{C}$ denote the $G$-orbit of $\bT_{n_i} = \lsup{g_i}\Ebtorus$.

We first show that the orbit $\dorbit_i$ depends only on the $\sigma$-conjugacy class of $w_i$.
Suppose ${\bT} \in \dorbit_i$.   Then $\bT$ is $\tame$-split and so there exists $g \in \bG(\tame)$ such that ${\bT} = \lsup{{g}}\Ebtorus$.   Let ${w}'$ denote the image of $g \inv \sigma(g) \in N_{\bG(\tame)}(\Ebtorus)$ in $\absW$.  Since ${\bT}$ is $G$-conjugate to  $\bT_i$, there exists $k_i \in G$ such that ${\bT} = \lsup{k_i}{\bT_i}$.   Thus $\lsup{g} \Ebtorus = \lsup{k_i {g_i}} \Ebtorus$.   Consequently, there exists $n \in N_{\bG(\tame)}(\Ebtorus)$ such that $g_i\inv k_i\inv {g} = n$, so $k_i\inv {g} = g_i n$.   Note that
$$w' = g\inv \sigma(g) \Ebtorus = (k_i\inv g)\inv \sigma(k_i\inv g) \Ebtorus =
(g_i n)\inv \sigma(g_i n) \Ebtorus = {n}\inv n_i  \sigma(n) \Ebtorus = u \inv w_i \sigma(u)$$
where $u$ denotes the image of $n$ in $\absW$.
Thus, $w'$ is $\sigma$-conjugate to $w_i$.

We now show that if $w_1$ is $\sigma$-conjugate to $w_2$, then $\dorbit_1 = \dorbit_2$.   If $w_1$ is $\sigma$-conjugate to $w_2$, then there exists $n \in N_{\Egroup}(\Ebtorus)$ such that $n\inv n_2 \sigma(n)\Ebtorus =  n_1 \Ebtorus$, or
$(g_2n)\inv \sigma(g_2n) \Ebtorus = g_1\inv \sigma(g_1) \Ebtorus$.  Consequently, there exists $t \in \bT_1(\tame)$ such that
$g_2 n g_1\inv t = \sigma(g_2 n g_1\inv) $.  Thus, for all $\gamma \in T_1$, we have
$$\sigma(\lsup{g_2ng_1\inv} \gamma)  = \lsup{\sigma(g_2ng_1\inv)} \gamma =  \lsup{g_2ng_1\inv t} \gamma = \lsup{g_2ng_1\inv} \gamma.$$
That is, $g_2ng_1\inv$ maps $K$-points of $T_1$ to $K$-points of $T_2$.   Thus, by Lemma~\ref{lem:gside}, we have that  $T_1$ and $T_2$ are $G$-conjugate.  That is, $\dorbit_1 = \dorbit_2$.    We can therefore define a  function $\varphi_\sigma \colon \Wtamesim \rightarrow  \mathcal{C}$ by sending the $\sigma$-conjugacy class of $w_i \in \Wtame$ to the $G$-conjugacy class of $\bT_i$.

Finally, we show that $\varphi_\sigma$ is injective. Suppose $c_1$ and $c_2$ are two $\sigma$-conjugacy classes in $\Wtame$ with ${\varphi_\sigma}(c_1) = {\varphi_\sigma}(c_2)$.
Suppose $w_i \in c_i$.  Since ${\varphi_\sigma}(c_1) = {\varphi_\sigma}(c_2)$, there exists $g \in G$ such that $\lsup{g}\bT_{1} = \bT_{2}$, and so $\lsup{g g_1}\Ebtorus = \lsup{g_2}\Ebtorus$.   Consequently, there exists $n' \in N_{\bG(\tame)}(\Ebtorus)$ such that $g_1\inv g\inv {g_2} = n'$, so $g\inv {g_2} = g_1 n'$.   Note that
$$n_2 \Ebtorus  = g_2\inv \sigma(g_2)  \Ebtorus = (g\inv g_2)\inv \sigma(g\inv g_2)  \Ebtorus=
(g_1 n')\inv \sigma(g_1 n') \Ebtorus  = {n'}\inv n_1  \sigma(n')  \Ebtorus.$$
Consequently, $c_1 = c_2$.
\end{proof}

\begin{defn}  An element $w \in \Wtame$ is called $\sigma$-elliptic provided that the $\sigma$-conjugacy class of $w$ does not intersect a proper parabolic subgroup $\absW_\theta$ of $\absW$ where $\theta \subset \Esimple$, $\theta \neq \Esimple$, and $\sigma (\theta) = \theta$.   Here $\absW_\theta$ is the subgroup of $\absW$ generated by the simple reflections corresponding to the roots in $\theta$.  An element $n \in \titsW$ will be called \emph{$\sigma$-elliptic} provided that its image in $\absW$ is 
$\sigma$-elliptic.    A $\sigma$-conjugacy class in $\absW$  is called $\sigma$-elliptic provided that some (hence any) element in $c$ is $\sigma$-elliptic.
\end{defn}

\begin{lemma} \label{lem:elliptictoweyl}
Suppose $n \in \titsW$ is tame.  We have
$n$ is $\sigma$-elliptic if and only if $\bT_n$ is $K$-minisotropic.  Moreover, if $\bT_n$ is $K$-minisotropic, then  $x_{T_n} = x_n$.
\end{lemma}

\begin{proof}
Let $\bar{n}$ denote the image of $n$ in $\absW$.

Suppose $n$ is not $\sigma$-elliptic.  Without loss of generality, we may assume that $\bar{n} \in\absW_\theta$ for some proper $\theta \subset \Esimple$ with $\sigma(\theta) = \theta$.  Let $\bM_\theta$ denote the Levi subgroup of $\bG$ that contains $\Ebtorus$ and corresponds to $\theta$. That is, $\bM_\theta$ is the centralizer of $(\cap_{\alpha \in \theta} \ker (\alpha))^\circ$. Note that $(\bar{n},\barsigma)$ is tame in $\absW_\theta \rtimes \EGal$ and  we can construct $g_n$ in $\bM_\theta(\tame)$ for which $g_n\inv \sigma(g_n)$ has image $\bar{n}$ in $\absW_\theta$. By construction, $\lsup{g_n}\Ebtorus \leq \bM_\theta$ is a $K$-torus which is not $K$-minisotropic for $\bG$.

Suppose $\bT_n$ is not $K$-minisotropic.   After conjugating by an element of $G$, we may assume that $\bT_n \subset \bM_\theta$ for some $\sigma$-stable proper $\theta \subset \Esimple$. Since $\lsup{g_n}\Ebtorus$ and $\Ebtorus$ are $\tame$-split tori in $\bM_\theta$, there exists $m \in \bM_\theta(\tame)$ such that $\lsup{g_n}\Ebtorus = \lsup{m}\Ebtorus$.  As in the proof of Corollary~\ref{cor:welldefined} we have that the images of $m\inv \sigma(m)$ and $g_n\inv \sigma(g_n)$ in $\absW$ are $\sigma$-conjugate.  Since the image of $m\inv \sigma(m)$ belongs to $\absW_\theta$, we conclude that $n$  is not $\sigma$-elliptic.

Suppose $\bT_n$ is $K$-minisotropic.  Since $T_n$ splits over the tame extension $\nsplits$, the point $x_n$ is the $\sigma$-fixed point of the image of $\AA'(\bT_n,\nsplits)$ in $\BB^{\red}(\bG,\nsplits)$.  Hence, we have $x_n = x_{T_n}$.
\end{proof}

\begin{cor}
 $w \in \Wtame$ is $\sigma$-elliptic if and only if the only fixed point for the action of $w\sigma$  on  $\X_*(\Ebtorus)/\X_*(\bZ)$ is $0$.
 \end{cor}

 \begin{proof}
 Suppose $n \in \titsW$ is a lift of $w$.
 Since the action of $\sigma$ on $\AA(\bT,\nsplits)$ corresponds to the action of $w \sigma$ on $\AA(\Ebtorus,\nsplits)$, the result follows from Lemma~\ref{lem:elliptictoweyl}.
 \end{proof}

\begin{remark}  \label{rem:allindependent}
    Suppose $w \in \Wtame$ is $\sigma$-elliptic.  Since any two lifts of $w$ into $N_{\Egroup'}(\Ebtorus)$ are $\EbtorusErat$-conjugate, we conclude that any two lifts of $w$ into $\titsW \leq \Egroup'$ have the same order.  Thus, when $w \in \Wtame$ is $\sigma$-elliptic, the $\ell$ and $\nsplits$ of Definition~\ref{defn:nsplits} depend only on the $\sigma$-conjugacy class of $w$.
\end{remark}

 \begin{lemma}  \label{lem:thepoint2} If $w \in \Wtame$ is $\sigma$-elliptic, then we may choose $g_n$ in Lemma~\ref{lem:thepoint}  such that $g_n\inv \sigma(g_n) = z  \cdot \lsup{\lambda_n(\pi)}n $ for some $z \in \tilde{Z}_0$.
 \end{lemma}

 \begin{remark}
 Note that $g_n\inv \sigma(g_n) \in \bG(K_n)_{x_n,0}$.   Also, since $\absW \leq \Egroup'$, we can replace $\bG$ by $\bG'$ in every result of Section~\ref{sec:3.3} to this point.  Thus, we can choose $h, g_n \in \Egroup'$, hence $z \in \tilde{Z}_0 \cap \Egroup'$.
 \end{remark}

 \begin{proof}
 Let $\tilde{\bfZ}$ denote the maximal split torus in the center of $\Equotient_{x_0}$; note that this is the image of $ \tilde{Z}_0 \cap \Egroup_{x_0,0}$ in $\Equotient_{x_0}$.
Since $w$ is $\sigma$-elliptic, the map $1 - w \sigma \colon  \bshA/\tilde{\bfZ} \rightarrow \bshA/\tilde{\bfZ}$ is surjective.   Thus, by using an approximation argument as in the proof of Lemma~\ref{lem:Asharpdone}, we can produce $a \in \EbtorusErat$ and $z \in \tilde{Z}_0 $ such that  $a n \sigma(a)\inv n\inv = z \cdot \lsup{n} \lambda_n(\xi)$.  In the proof of Lemma~\ref{lem:thepoint} replace $h$, which has the property that $h \lambda_n(\xi) \sigma(h)\inv = n$, 
with  $\tilde{h} := ah$  and replace $g_n$ with $\tilde{g} := \lsup{\lambda_n(\pi)} \tilde{h}\inv$ to obtain
\[\tilde{g}_n\inv \sigma(\tilde{g}_n) =  \lsup{\lambda_n(\pi)}( a h  \lambda_n(\xi) \sigma( h)\inv \sigma(a)\inv \lambda_n(\xi)\inv) = \lsup{\lambda_n(\pi)}( a n \sigma(a)\inv n\inv n \lambda_n(\xi)\inv) = z \cdot \lsup{\lambda_n(\pi)}n.
\]
Note that $\lsup{\tilde{g}_n} \Ebtorus = \lsup{{g}_n} \Ebtorus$.
 \end{proof}

 As in~\cite[Theorem~4.1]{reederetal:epipelagic} is also possible to explicitly describe the reductive quotient at the point $x_n$:

\begin{lemma} Suppose $n \in \titsW$ is tame.   The map $\Ad(\lambda_n(\pi)) \colon \bG(\nsplits)_{x_0,0} \rightarrow \bG(\nsplits)_{x_n,0}$ 
identifies $\sfG_{x_n}$ with $\sfG_{x_0}^{\lambda_n(\xi)}$.
\end{lemma}

\begin{proof}
Since $\Ebtorus(\nsplits) \cap \bG(\nsplits)_{x_0,0} = \Ebtorus(\nsplits) \cap \bG(\nsplits)_{x_n,0}$ and $\lambda_n(\xi)$ centralizes $\Ebtorus(\nsplits)$,
it suffices to look at what happens on root subgroups of $\bfG_{x_0}$.

Suppose $\alpha \in \Phi(\bG,\bA)$.  Choose  $\dot{\alpha} \in \Phi(\bG,\Ebtorus)$ satisfying $\res_{\bA}(\dot{\alpha}) = \alpha$.  Since $\lambda_n \in \X_*(\bA)$, we have $\langle \lambda_n, \alpha\rangle = \langle \lambda_n, \dot{\alpha} \rangle = \langle \lambda_n ,\sigma \dot{\alpha}\rangle $ where $\langle \, , \, \rangle \colon \X_*(\Ebtorus) \times \X^*(\Ebtorus)$ denotes the natural pairing.  For all $X$ in the root space $\gg(\nsplits)_{\dot{\alpha}}$, we have $\sigma(\lsup{\lambda_n(\pi)}X) = \xi^{\langle \lambda_n, \alpha \rangle} \cdot \lsup{\lambda_n(\pi)}(\sigma(X))$.  Thus, if $U_\alpha$ denotes the unipotent subgroup of $G$ corresponding to $\alpha$, then $\Ad(\lambda_n(\pi))$ maps  $U_\alpha \cap G_{x_0,0}$ to $U_\alpha \cap G_{x_n,0}$ if and only if  $\xi^{\langle \lambda_n, \alpha \rangle} = 1$.  Since $x_0$ is absolutely special, 
the result follows.
\end{proof}

\begin{lemma}   \label{lem:tobesharpened}
Suppose $n \in \titsW$ is tame.
We have $G_{x_0,0}^{\lambda_n} \leq G_{x_n,0}$.
\end{lemma}

\begin{proof}
If $k \in G_{x_0,0}^{\lambda_n}$, then $k \in G_{x_0 + r \lambda_n,0}$ for all $r \in \R$.   Thus $G_{x_0,0}^{\lambda_n} \leq G_{x_0 + \lambda_n/\ell_n,0} = G_{x_n,0}$.
\end{proof}

When $n$ is $\sigma$-elliptic, we revisit and sharpen this result in Section~\ref{sec:kaccoordinates}.

\subsection{Parameterizing maximal tame \texorpdfstring{$K$-tori of $\bG$}{K-tori of G}}

Recall that a $K$-torus $\bT$ of $\bG$ will be called $\emph{tame}$ provided that $\bT$ is $\tame$-split.   Note that if $\bT$ is tame, then every $G$-conjugate of $\bT$ is tame as well.   We let $\Ctame$ denote the set of $G$-conjugacy classes of tame maximal $K$-tori of $\bG$.

\begin{lemma} \label{lem:imageofphi}
The image of the injective map
 $$\varphi_\sigma \colon \Wtamesim \rightarrow \mathcal{C}$$
of Corollary~\ref{cor:welldefined}
 is $\Ctame$.
\end{lemma}

\begin{proof}  By construction, the image of $\varphi_\sigma$  lies in $\Ctame$.  So, we only need to show that if $\dorbit \in \Ctame$, then there exists $c \in \Wtamesim$ such that $\varphi_\sigma(c) = \dorbit$.  Fix $\dorbit \in \Ctame$.   Choose $\bT \in \dorbit$.   Let $E \leq \tame$ be the Galois extension over which $\bT$ splits.   From Lemma~\ref{lem:tamealltheway} we have that $\Ebtorus$ splits over $E$, and so $\Esplits \leq E$.   Choose $g \in \bG(E)$ such that $\bT = \lsup{g} \Ebtorus$.   Let $w$ denote the image of $g\inv \sigma(g)$ in $\absW$.  If $E = \tame^{\sigma^j}$, then $\dabs{\Gal(\Esplits/K)}$ divides $j$ and $w^j$ is the image of
$$( g\inv \sigma(g))^{j} = (g\inv \sigma(g)) \sigma(g\inv \sigma(g)) \sigma^2 (g\inv \sigma(g)) \cdots \sigma^{(j-1)} (g\inv \sigma(g))= g\inv \sigma^{j}(g) = 1$$
in $\absW$.   Thus, the order of  $w \barsigma \in\absW \times \EGal$ divides $j$, hence $w$ is tame.  Let $c$ denote the $\sigma$-conjugacy class of $w$.   Note that $c \in \Wtamesim$ and $\varphi_\sigma (c) = \dorbit$.
\end{proof}

\begin{defn}  \label{defn:Ktame} We will say that $\bG$ is \emph{$K$-tame} provided that $\mathcal{C} = \Ctame$.
\end{defn}

\begin{cor}   \label{cor:CandW}
Recall that we are assuming that $\Esplits$ is a tame extension of $K$.
We have
$\bG$ is $K$-tame if and only if $\Wtame =\absW$.  Moreover, if $\bG$ is $K$-tame, then the map $\varphi_\sigma \colon\absW_{\sim_\sigma} \rightarrow \mathcal{C}$ is bijective.
\end{cor}

\begin{proof}
From Lemma~\ref{lem:imageofphi} it is enough to prove that $\bG$ is $K$-tame if and only if $\Wtame =\absW$.

If $K$ is quasi-finite, then this is~\cite[Lemma~4.2.1]{debacker:unramified}.   Thus, for the remainder of the proof we may and do assume that $\ff$ is finite.

We first show that we may assume that $\bG$  is $k$-quasi-split. Indeed, $\bG$ is an inner form of a $k$-quasi-split group $\bG^*$.  This means there is a $\bar{k}$-isomorphism $\varphi \colon \bG^* \rightarrow \bG$ such that for all $\gamma \in   \Gal(\bar{k}/k)$ there exists $g_\gamma \in \bG^*(\bar{k})$ such that  $c_{\gamma} = \varphi \inv \circ \gamma \circ \varphi = \Ad(g_\gamma) \in \bG^*_{\ad}(\bar{k})$.  Since $K$ has cohomological dimension $\leq 1$, we may assume that $c_\gamma  = 1 \in G^*_{\ad}$  for all $\gamma \in  I = \Gal(\bar{k}/K)$.   Thus $\gamma(\varphi) = \varphi$ for all $\gamma \in I$ and  we conclude that $\varphi \colon \bG^* \rightarrow \bG$ is a $K$-isomorphism that identifies $G^*$ with $G$.   Thus, without loss of generality, we may replace $\bG$ with $\bG^*$ and assume $\bG$ is $k$-quasi-split.

Suppose $\absW = \Wtame$.   Suppose $\bT \in \dorbit \in \mathcal{C}$.  The action of $\Gal(\bar{k}/k)$ on $\bT$ factors through $\absW \rtimes \Aut(\Esimple) \cong N_{\bG}(\bT)/\bT \rtimes \Aut(\Esimple)$.  Since $\absW = \Wtame$, we conclude that $p$ does not divide the order of $\absW \rtimes \Aut(\Esimple)$, hence $\Gal(\bar{k}/\tame)$ acts trivially on $\bT$.  That is, $\bT$ splits over a tame extension of $K$ and $\dorbit \in \Ctame$.

Suppose $\mathcal{C} = \Ctame$.    It will be enough to show that if $p \mid \dabs{\Wtame}$, then there is a $K$-torus in
$\bG$  that doesn't split over $\tame$.    Since we are assuming that $\bG$ is  $k$-quasi-split, this follows from Part 1. of~\cite[Corollary~2.6]{fintzen:tame}.
\end{proof}

\begin{rem}
Suppose $\ff$ is a finite field.   In ~\cite[Theorem 2.4]{fintzen:tame} Fintzen explores the extent to which the condition $\Wtame = \absW$ can be relaxed when we are interested in maximal $k$-tori rather than maximal $K$-tori.   One can think of this as asking: when does the $G$-orbit of a maximal $K$-torus of $\bG$ that splits over a tame extension of $K$ contain a maximal $k$-torus of $\bG$?   This question is addressed in Sections~\ref{sec:existenceKktori} and~\ref{sec:existenceintame}.
\end{rem}

\begin{remark}   \label{rem:oncohomology} Suppose $\ff$ is a finite field.
In the notation of the proof of Corollary~\ref{cor:CandW} we have for all $\gamma \in \Gal(\bar{k}/K)$ that
$$\Ad(g_{\Fr}) = c_{\Fr} = \varphi\inv \circ \Fr \circ ( \varphi \circ \Ad({g}_{\Fr\inv \gamma  \Fr})) = c_{\gamma \Fr} = c_{\gamma} \gamma(c_{\Fr}) = \Ad({g}_\gamma) \circ \gamma(\Ad({g}_\Fr)) = \gamma(\Ad({g}_{\Fr})).$$
Here we are using that $\Fr$ normalizes $\Gal(\bar{k}/K)$ and $\Ad(g_\gamma) = 1 \in G^*_{\ad}$ for all $\gamma \in \Gal(\bar{k}/K)$.   Thus we have $\Ad({g}_{\Fr}) \in G^*_{\ad}$.   Moreover, as groups with $\Gal(K/k)$-action we can identify $G$ with $G^*$ where $\Fr$ acts on $G^*$ by $\Ad(g_\Fr) \circ \Fr$.   Since $\ff$ is finite, there is an alcove $D$ in $\BB(G^*) = \BB(G)$ that is $\Ad(g_\Fr) \circ \Fr$-stable.   Choose $h \in G^*$ such that $h C = D$.  Since $c_\Fr \in \cohom^1(\Fr,G^*_{\ad})$ is cohomologous to $\Ad(h\inv) c_\Fr \Ad(\Fr (h))$, we may replace $\Ad(g_\Fr)$ with $\Ad(h\inv) \Ad(g_\Fr) \Ad(\Fr(h))$ and assume $\Ad(g_\Fr)$ stabilizes $C$.   That is, at the level of $K$-points all inner forms of $\bG$ can be thought of as $G^*$ with $\Fr$ acting by  $\Ad(g) \circ \Fr$ for some $\Ad(g) \in \Stab_{G^*_{\ad}}(C)$.  Moreover, one can show that for every $\Ad(g) \in \Stab_{G^*_{\ad}}(C)$, the group $G^*$ with $\Fr$ acting by  $\Ad(g) \circ \Fr$ defines an inner form of $\bG$.   Since, by Lang-Steinberg, we have $(G_{\ad})_{C,0} = (1 - \Ad(g) \Fr)(G_{\ad})_{C,0}$, it follows that every element of the coset $\Ad(g) (G_{\ad})_{C,0}$ is cohomologous to $\Ad(g)$.  So, for example, $\Ad(g)$ is cohomologous to an element of $N_{G_{\ad}}(\bA)$: since $C \subset \AA(\Ad(g) \cdot A)$, there exists $\Ad(k) \in (G_{\ad})_{C,0}$ such that $(\Ad(g)\Ad(k)) \cdot \bA = \bA$.
\end{remark}

\subsection{Kac coordinates} \label{sec:kaccoordinates}
As discussed in Remark~\ref{rem:allaboutkac} the
 element $\lambda_n \in \X_*(\bA)$ encodes  \emph{Kac coordinates} of the image of $n \sigma $ in $\Aut(\gg)$.  We  illustrate how this works for an $n$ that is $\sigma$-elliptic.

 The point $x_n = x_0 + \lambda_n/\ell_n = \lambda_n(\pi) \cdot x_0$ is an element of $\AA'(A)$ and so there exists  $d \in N_G(\bA)$ such that $y_n := d \cdot x_n \in \bar{C}'$.   Note that $\psi(x_n) = \psi(x_0) + \langle \dot{\psi}, \lambda_n \rangle / \ell_n \in \Q$ for all $\psi \in \Psi$.  Thus, it makes sense to define  $j_n$ to be the least natural number for which $\psi(x_n) \in \frac{1}{j_n} \Z$ for all $\psi \in \Psi$.  The $\dabs{\Delta}$-tuple $(j_n \cdot \psi(y_n) \, | \, \psi \in \Delta)$ defines \emph{Kac coordinates} of $n$.  If the point $x_n$ belongs to the closure of $C'$, then the coordinates are called \emph{normalized Kac coordinates}.   Tables of some Kac coordinates (or, equivalently, Kac diagrams) and values of $j_n$ may be found in~\cite{adams-he-nie:from}, \cite{Bouwknegt:Lie}, \cite{Levy:KW}, ~\cite{reeder:thomae}, \cite{reederetal:gradings}, and \cite{reederetal:epipelagic}.  An exhaustive list of Kac diagrams for $\sigma$-elliptic Weyl group elements may be produced by combining the results of~\cite[\S9]{adams-he-nie:from} and~\cite[\S4]{debacker-haley:kac}.

 Thanks to Lemmas~\ref{lem:card} and~\ref{lem:elliptictoweyl},   since $n$ is  $\sigma$-elliptic we have that $x_n$ is uniquely determined, up to the action of $N_{G}(\bA)$,  by $j_n$ and the Kac coordinates of $n$.  In fact,  we can do better as the following lemma shows (see also~\cite[Lemma~6.4]{adams-he-nie:from}).

 \begin{lemma} \label{lem:kacinbarC} Suppose $n \in \titsW$ is tame and $\sigma$-elliptic.  We can choose  $\lambda_n \in \X_*(\bG' \cap \bA)$ such that $x_n = x_0 + \lambda_n/\ell_n \in \bar{C}'$.   In fact, this uniquely identifies $\lambda_n$.
 \end{lemma}

 \begin{proof}
 Choose $d \in N_G(\Ebtorus)$ such that $y_n = d \cdot x_n \in \bar{C}'$.
Write $d =  am $ where $ m \in  N_{G_{x_0,0}}(\bA)$ and $a \in \EbtorusKrat$.
 The element $a$ acts on $\AA'(A)$ by translation, say $\mu \in \X_*(\bA)$.  Then $y_n = am \cdot x_n = x_0 + \mu + m \cdot \lambda_n/\ell_n$.   After replacing the $h$ and  $\lambda_n$ that are used to define $g_n$ at the start of Section~\ref{sec:markssec} by $hm\inv$ and ${m} \cdot \lambda_n$, we may assume that $y_n = x_0 + (\ell_n \mu + \lambda_n)/\ell_n$.   Since $(\ell_n \mu + \lambda_n)(\xi) = \lambda_n(\xi)$, we may also replace $\lambda_n$ with $\ell_n \mu + \lambda_n$.     Thus, $y_n = x_n = x_0 + \lambda_n/\ell_n \in \bar{C}'$.

 Since $\titsW$ is in  $\Egroup_{x_0,0} \cap \Egroup' \leq \Egroup'$, we can carry out all of the proofs involving $\titsW$, $\lambda_n$, etc. inside of $\bG'$.  In particular, we can assume $\lambda_n \in \X_*(\bG' \cap \bA)$.   The statement about uniqueness follows immediately from the fact that since $n$ is $\sigma$-elliptic, we have that the $G$-orbit of $x_n$ intersects $\bar{C}'$ exactly once.
\end{proof}

When $\lambda_n$ is chosen such that  $x_n = x_0 + \lambda_n /\ell_n \in \bar{C}'$, it is 
often, but not always, true that we can improve on Lemma~\ref{lem:tobesharpened} by identifying $\sfG_{x_n}$ with $\sfG_{x_0}^{\lambda_n}$.  An example where we can't identify them is
discussed in  Example~\ref{ex:tameG2}: for the group $\Gtwo$ and the element  $-1 \in A_1 \times \tilde{A}_1$ in the Weyl group we have $\sfG_{x_{n}}$ is isomorphic to $\SO_4$ while $\sfG_{x_0}^{\lambda_{n}}$ is isomorphic to $\GL_2$.     In Lemma~\ref{lem:leviandgenlevi} we discuss conditions under which we can identify $\sfG_{x_n}$ with $\sfG_{x_0}^{\lambda_n}$.

\begin{defn}
 For a facet $F$ in $\AA'(A)$,  let $\bM(F)$ denote
the Levi $K$-subgroup  generated by $\Ebtorus$ and  the root groups $\bU_{\alpha}$  where $\alpha$ runs over those roots in $\Phi(\bG,\Ebtorus)$ for which there exists  $\psi \in \Psi$ such that $\dot{\psi} = \res_{\bA} \alpha$ and $\psi$ is constant on $F$.
\end{defn}

\begin{lemma}   \label{lem:leviandgenlevi}
Suppose $n \in \titsW$ is tame and $\sigma$-elliptic.  Choose $\lambda_n$ such that $x_n = x_0 + \lambda_n/\ell_n \in \bar{C}'$.
 Let $F_n$ denote the facet in $\AA'(A)$ to which $x_n$ belongs. If $F_n$  contains $x_0$ in its closure,
then we can identify  $\sfG_{x_n}$ with $\sfG_{x_0}^{\lambda_n}$ via the equality $G_{x_0,0}^{\lambda_n} = M(F_n)_{x_n,0}$.
\end{lemma}

\begin{proof}
If $x_0 = \bar{F}_n$, then since $x_0 = x_0 + \lambda_n$, we conclude that $\lambda_n$ is central.  The result follows.

Suppose $x_0 \neq \bar{F}_n$. Since $F_n$ contains $x_n = x_0 + \lambda_n/\ell_n$ we conclude that the intersection of the ray $\vec{r} = \{x_0 + r \lambda_n  \, | \, r \geq 0 \}$ with $F_n$ contains an open segment of $\vec{r}$.
 Thus $M(F_n) = G^{\lambda_n}$ and so $M(F_n)_{x_n,0}  = G_{x_0,0}^{\lambda_n}$.      Since the reductive quotient of $M(F_n)_{x_n,0}$ is $\bfG_{x_n}$, the result follows.
\end{proof}

 \subsection{Examples}
 We illustrate the results of this section by looking at a number of examples.   Since the tori of interest depend on the $\sigma$-conjugacy  class of the image of $n$ in $\absW$ (rather than $n$), below we label the points and tori that occur  by their corresponding $\sigma$-conjugacy  class in $\absW$.   Also, recall that when illustrating the various $x_T$ that arise,  the isogeny class of $\bG'$ doesn't matter (see Remark~\ref{rem:indofisogeny}).

\begin{example}  Suppose $\bG$ has $\Phi \cong \Eroot \cong A_{n-1}$ and either $\ff$ has characteristic zero or $p$ does not divide $n$.  Fix a basis   $\Pi = \{ \alpha_1,  \alpha_2, \ldots ,  \alpha_{n-1} \}$ for the root system of $\bG$ with respect to the torus $\bA$.  Let $w = \prod_{i=1}^{n-1}  w_i$ denote a Coxeter element where $w_i$ is the simple reflection in $\absW$ corresponding to $\alpha_i \in \Pi$.  Let $c$ denote the $\absW$-conjugacy class of $w$ in $\absW$ and let $\dorbit_c$ denote the corresponding $G$-conjugacy class of $K$-minisotropic maximal $K$-tori in $\bG$.  The barycenter of $C'$ is the  unique point in the alcove $C'$  corresponding to $\dorbit_c$.
\end{example}

\begin{example} \label{ex:tamesp4} Suppose either $\ff$ has characteristic zero or $p > 2$ and $\bG$ has type $C_2$.  Without loss of generality we assume $\bG$ is $\Sp_4$.
For each $\Sp_4(K)$-conjugacy class of  $K$-anisotropic maximal tori in $\Sp_4(K)$ we identify, in Figure~\ref{fig:Sp4embeddings}, the unique point in an alcove corresponding to that class.

Fix a basis $\Pi = \{ \alpha, \beta \}$ for the root system of $\Sp_4$ with respect to the torus $\bA$.  Assume $\alpha$ is the short root. The  (absolutely) special vertex $x_0$ of $\Sp_4(K)$ in Figure~\ref{fig:Sp4embeddings} lies on hyperplanes defined by affine roots with gradients $\alpha$ and $\beta$, and the pictured alcove is the unique alcove that contains $x_0$ in its closure and on which these affine roots are positive.   If $w_\alpha$ and $w_\beta$ denote the simple reflections corresponding to $\alpha$ and $\beta$, then the two elliptic conjugacy classes in the Weyl group are $C_2 = \{w_\alpha w_\beta, w_\beta w_\alpha \}$ and $-1 \in A_1 \times {A}_1 = \{w_\alpha w_\beta w_\alpha w_\beta \}$.   Let $\dorbit_{C_2}$ and $\dorbit_{-1}$ denote the corresponding $\Sp_4(K)$-conjugacy classes of $K$-minisotropic maximal $K$-tori in $\Sp_4$.  Since the square of every element in $C_2$ is $-1$,  we can read off the unique points in the alcove that correspond to $\dorbit_{C_2}$ and $\dorbit_{-1}$ from~\cite[Section 8.2.2]{reederetal:gradings}.
For $\dorbit_{C_2}$ the point is   $x_{C_2} =
x_0 + (3 \check{\alpha} + 4 \check{\beta})/8 $, and for $\dorbit_{-1}$ the point is $x_{-1} =
x_0 + (\check{\alpha} + 2 \check{\beta})/4$.

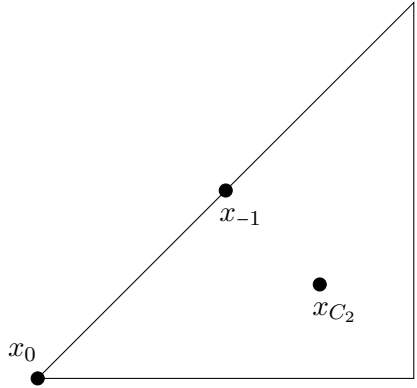
\begin{figure}[ht]
\centering
\begin{tikzpicture}
\draw (0,0) 
  -- (5,0)
  -- (5,5)
  -- cycle;
 
  \draw[black,fill=black] (0,0) circle (.5ex);
\draw(-.2,.1) node[anchor=south]{$x_0$};

  \draw[black,fill=black] (3.75,1.25) circle (.5ex);
\draw(3.95,1.15) node[anchor=north]{$x_{C_2}$};

  \draw[black,fill=black] (2.5,2.5) circle (.5ex);
\draw(2.7,2.4) node[anchor=north]{$x_{-1}$};

\end{tikzpicture}
\caption{The embeddings of buildings of tame $K$-minisotropic maximal tori in  type $C_2$ \label{fig:Sp4embeddings}}
\end{figure}

\end{example}

\begin{example}   \label{ex:tameG2}
Suppose either $\ff$ has characteristic zero or $p > 3$.  For each $\Gtwo(K)$-conjugacy class of  $K$-anisotropic maximal tori in $\Gtwo$ we identify, in Figure~\ref{fig:g2embeddings}, the unique point in an alcove corresponding to that class.

Fix a basis $\Pi = \{ \alpha, \beta \}$ for the root system of $\Gtwo$ with respect to the torus $\bA$.  Assume $\alpha$ is the short root. The  special vertex $x_0$ of $\Gtwo(K)$ in Figure~\ref{fig:g2embeddings} lies on hyperplanes defined by affine roots with gradients $\alpha$ and $\beta$, and the pictured alcove is the unique alcove that contains $x_0$ in its closure and on which these affine roots are positive.   If $w_\alpha$ and $w_\beta$ denote the simple reflections corresponding to $\alpha$ and $\beta$, then the three elliptic conjugacy classes in the Weyl group are $G_2 = \{w_\alpha w_\beta, w_\beta w_\alpha \}$, $A_2 = \{w_\alpha w_\beta w_\alpha w_\beta  , w_\beta w_\alpha w_\beta  w_\alpha \}$ and $-1 \in A_1 \times \tilde{A}_1 = \{w_\alpha w_\beta w_\alpha w_\beta  w_\alpha w_\beta\}$.   Let $\dorbit_{G_2}$, $\dorbit_{A_2}$, and $\dorbit_{-1}$ denote the corresponding $\Gtwo(K)$-conjugacy classes of $K$-minisotropic maximal $K$-tori in $\Gtwo$.   From~\cite[Section 5.1]{reederetal:epipelagic} the unique points in the alcove that correspond to $\dorbit_{G_2}$, $\dorbit_{A_2}$, and $\dorbit_{-1}$  are: $x_{G_2} =
x_0 + (3 \check{\alpha} + 5 \check{\beta})/6 $; $x_{A_2} = 
x_0 + (\check{\alpha} + 2 \check{\beta})/3$; and $x_{-1} =
x_0 + (\check{\alpha} + 2 \check{\beta})/2$.

\begin{figure}[ht]
\centering
\begin{tikzpicture}

\draw (0,0)
  -- (4.04,7)
  -- (0,7)
  -- cycle;

   \draw[black,fill=black] (0,0) circle (.5ex);
\draw(-.4,-.2) node[anchor=south]{$x_0$};

  \draw[black,fill=black] (0,7) circle (.5ex);
\draw(.45,6.9) node[anchor=north]{$x_{-1}$};

  \draw[black,fill=black] (0,4.67) circle (.5ex);
\draw(.4,4.57) node[anchor=north]{$x_{A_2}$};

  \draw[black,fill=black] (2.02,5.83) circle (.5ex);
\draw(2.32,5.73) node[anchor=north]{$x_{G_2}$};

\end{tikzpicture}
\caption{The embeddings of buildings of tame $K$-minisotropic maximal tori in  $\Gtwo$ \label{fig:g2embeddings}}
\end{figure}

\end{example}

\begin{example}  \label{ex:ramsu3}
 Suppose  $p > 3$ and $\bG$ has type $\lsup{2}A_2$.  Without loss of generality we assume $\bG$ is simply connected.
For each $G$-conjugacy class of  $K$-anisotropic maximal  tori in a ramified $\SU(3)$ we identify, in Figure~\ref{fig:su3embeddings}, the unique point in an alcove corresponding to that class.
 Thinking of $\SU(3,K)$ as the $\sigma$-fixed points of $\SL_3(\Esplits)$, the dotted equilateral triangle is a $\sigma$-stable alcove of $\SL_3(\Esplits)$.   If $\Esimple = \{\alpha, \beta\}$, then $\sigma(\alpha) = \beta$.
The absolutely special vertex, $x_0$, of $\SU(3)$ pictured in Figure~\ref{fig:su3embeddings}  lies on the hyperplanes defined by  affine roots with gradients $\alpha$ and $\beta$ while the other vertex, $z$, lies on a hyperplane defined by an affine root with gradient $\alpha + \beta$.   Since $p \neq 2$, the reductive quotient at $x_0$ is $\PGL_2$ and at $z$ it is $\SL_2$.

The $\sigma$-conjugacy classes in $\absW$ are $c_{1} = \{1, w_\alpha w_{\beta}, w_\beta w_{\alpha} \}$, $c_{w_\alpha} = \{ w_\alpha, w_\beta \}$, and $c_{w_0} = \{ w_0 = w_\alpha w_\beta w_\alpha \}$.   Note that $w_\alpha \sigma$ acts by $-1 \times w_\beta w_\alpha$ on $\X_*(\Ebtorus)$ and $w_0  \sigma$ acts by $-1$ on $\X_*(\Ebtorus)$.  In particular, both  $c_{w_\alpha}$ and $c_{w_0}$ are $\sigma$-conjugacy classes of $\sigma$-elliptic elements in $\absW$.    Let $\dorbit_{c_{w_\alpha}}$ and $\dorbit_{c_{w_0}}$ denote the corresponding $\SU_3(K)$-conjugacy classes of $K$-minisotropic maximal $K$-tori in $\SU_3$.  From~\cite[Table~9]{reederetal:gradings} the unique points in the alcove that correspond to $\dorbit_{c_{w_\alpha}}$ and $\dorbit_{c_{w_0}}$  are: $x_{c_{w_\alpha}} = x_0 + \frac{\check{\alpha} + \check{\beta}}{6}$ and $x_{c_{w_0}} = x_0$.

\begin{figure}[ht]
\centering
\begin{tikzpicture}

\draw[dotted, gray] (0,0)
  -- (5.196,3)
  -- (5.196,-3)
  -- cycle;
 \draw (0,0) -- (5.196,0);

\draw[black,fill=black] (0,0) circle (.5ex);
\draw(0,0) node[anchor=south]{$x_0$};

 \draw[black,fill=black] (0,0) circle (.5ex);
\draw(0,0) node[anchor=north]{$x_{c_{w_0}}$};

\draw[black,fill=black] (5.196,0) circle (.5ex);
\draw(5.196,0) node[anchor=south]{$z$};

\draw(3.481,0) node[anchor=north]{$x_{c_{w_\alpha}}$};
\draw[black,fill=black] (3.481,0) circle (.5ex);

\end{tikzpicture}
\caption{The embeddings of buildings of tame $K$-minisotropic maximal tori in  type $\lsup{2}A_2$ \label{fig:su3embeddings}}
\end{figure}

\end{example}

\begin{remark}   
In at least one respect, Example~\ref{ex:ramsu3} is a little misleading.  In general, facets do not behave well under extensions that are not unramified.  For example, for ramified $\SU(4)$ an alcove over $K$ intersects multiple alcoves over $\Esplits$.    It can also happen, even when $k$ is residually quasi-split, that an alcove over $k$ is not contained in any alcove over a totally ramified splitting extension (consider, for example, $\SL_1(D)$ for $D$ an index two division algebra over $k$).
\end{remark}

\begin{example}   \label{ex:tame3D4} 
Suppose either $\ff$ has characteristic zero or $p > 3$. Suppose $\bG$ is a  group of type $\lsup{3}D_4$. For each $\bG(K)$-conjugacy class of  $K$-anisotropic maximal tori in $\bG$ we identify, in Figure~\ref{fig:3D4embeddings}, the unique point in an alcove corresponding to that class.

Fix a basis $\Esimple = \{ \alpha_1, \alpha_2, \alpha_3, \alpha_4 \}$ for the root system of $D_4$ with respect to the torus $\Ebtorus$.  Assume $\alpha_2$  corresponds to the node with three adjacent edges in the Dynkin diagram of $D_4$ and that  ${\sigma}$ acts by ${\sigma}(\alpha_1) = \alpha_3$, 
${\sigma}(\alpha_3) = \alpha_4$, 
and
${\sigma}(\alpha_4) = \alpha_1$.
The four $\sigma$-elliptic $\sigma$-conjugacy classes in $\absW$ correspond to the conjugacy classes $F_4$, $F_4(a_1)$, $A_2 \times \tilde{A}_2$, and $C_3 \times A_1$ in the Weyl group of type $F_4$, and we will label them as such (this correspondence is discussed prior to~\cite[Lemma~4.11]{Levy:KW} and the labeling is that of~\cite[Table~8]{carter:weyl}).   
  If $w_i$ denotes the simple reflection corresponding to $\alpha_i$, then from~\cite[Theorem~7.5 and \S7.22]{he:minimal} the $F_4$  class is represented by $w_1 w_2$, the $F_4(a_1)$  class is represented by $w_3w_2w_1w_3$, the $A_2 \times \tilde{A}_2$  class is represented by $w_1w_2w_4w_3w_2w_1w_2w_4$,  and the $C_3 \times A_1$ class is represented by $w_3w_2w_1w_2w_3w_2$.

 Let $C'$ denote an alcove of  $\AA(\bA,K)$,
 and let $x_0$ denote the absolutely special vertex  in the closure of $C'$.  If $\check{\omega}_i \in \X_*(\Ebtorus) \otimes \R$ denotes the fundamental coweight for $\alpha_i$, then the vertices of $C'$ can be taken to be $x_0$, $x_0 + \vec{v}_{1}$, and $x_0 + \vec{v}_{2}$ where $\vec{v}_1 = (\check{\omega}_1 + \check{\omega}_3 + \check{\omega}_4)/6$ and $\vec{v}_{2} = \check{\omega}_2 / 3$  (see~\cite[\S4.4]{reeder:torsion}).
The reductive quotient at $x_0 + \vec{v}_{1}$ is of type $A_1 \times \tilde{A}_1$, and at $x_0 + \vec{v}_{2}$ it is of type $A_2$.
 The unique points in the alcove that correspond to $\dorbit_{F_4}$, $\dorbit_{F_4(a_1)}$, $\dorbit_{A_2 \times \tilde{A}_2}$, and $\dorbit_{C_3 \times A_1}$  are: $x_{F_4} = x_0 + \vec{v}_1/2 + \vec{v}_2/4$; $x_{F_4(a_1)} = 
x_0 + \vec{v}_2/2$; $x_{A_2 \times \tilde{A}_2} = 
x_0 + \vec{v}_2$; and $x_{C_3 \times A_1} =
x_0 + \vec{v}_1$.

\begin{figure}[ht]
\centering
\begin{tikzpicture}

\draw (0,0)
  -- (4.04,7)
  -- (0,7)
  -- cycle;

   \draw[black,fill=black] (0,0) circle (.5ex);
\draw(-.4,-.2) node[anchor=south]{$x_0$};

 \draw[black,fill=black] (0,7) circle (.5ex);
\draw(.65,6.9) node[anchor=north]{$x_{C_3 \times A_1}$};

 \draw[black,fill=black] (2.02,3.5) circle (.5ex);
\draw(2.60,3.4) node[anchor=north]{$x_{F_4(a_1)}$};

 \draw[black,fill=black] (4.04,7) circle (.5ex);
\draw(4.6,6.9) node[anchor=north]{$x_{A_2 \times \tilde{A}_2}$};

  \draw[black,fill=black] (1.01,5.25) circle (.5ex);
\draw(1.31,5.15) node[anchor=north]{$x_{F_4}$};

\end{tikzpicture}
\caption{The embeddings of buildings of tame $K$-minisotropic maximal tori in  $^3D_4$ \label{fig:3D4embeddings}}
\end{figure}

\end{example}
 
\begin{remark}
In the tame case, the location of a point $x_{T_n}$ in $\bar{C}'$ that occurs for a $K$-minisotropic maximal torus  of $\bG$ depends only on the existence of $\ell_n^{\text{th}}$ roots of unity in $K$.  Thus, in the tame setting the set of points that occur depends only on $\Esimple$ and the image of $\sigma$ in the automorphism group of the Dynkin diagram associated to $\Esimple$.
It seems likely that no additional points occur when we remove the tameness assumption, but I don't know how to prove this.  See also Remark~\ref{rem:moreinwildcase} and the discussion immediately preceding Definition~\ref{defn:coxetertorus}.
\end{remark}

\subsection{A comment on other vertices in \texorpdfstring{$\bar{C}'$}{bar C}}
In this subsection we assume that $\bG$ is $K$-split; that is, $\Ebtorus = \bA$.
See Remark~\ref{rem:thingsbreakdown} for a discussion of the situation when $\bG$ is not $K$-split.

Suppose $y$  is a vertex in $\bar{C}'
$.   Let $\Phi(y) \leq \Phi$ denote the root system of $\bfG_{y}$ with respect to $\bfA$.   Let $\bG(y)$ denote the connected reductive group generated by $\bA = \Ebtorus$ and the root groups $\bU_\alpha$ where  $\alpha$ runs over those elements of $\Phi(\bG,\bA)$ belonging to $\Phi(y)$.
As affine spaces, we may and do identify the apartment of $A$ in $\BB^{\red}(G(y))$ with $\AA'(A)$; under this identification, every $G(y)$-facet in $\AA'(A)$ will be a union of $G$-facets.

Note that since $\bG$ is $K$-split, we have that $y$ is absolutely special in $\AA'(A) \subset \BB^{\red}(G(y))$.   Let $\titsW(y) \leq G(y)_{y,0}$ denote a Tits group for $\absW(y) := N_{\bG(y)}(\Ebtorus)/\Ebtorus = N_{\bG(y)}(\bA)/\bA$.   Let $C(y)'$ denote the unique $G(y)$-alcove in $\AA'(A)$ that contains $C'$.  The chamber $C(y)'$ determines a set of simple affine roots $\Delta(y)$ for  $\Psi(\bG(y), \bA, K, \nu)$.

Suppose $n \in \titsW(y)$ is tame and $\sigma$-elliptic.  Let  $\ell_n$ denote the order of $n$ and choose a uniformizer $\pi \in K_n = \tame^{\sigma^{\ell_n}}$.   Define a $\ell_n^{\text{th}}$ root of unity by $\sigma(\pi) = \xi \pi$.  Thanks to the work above with $y$ in place of $x_0$,  we can find $\lambda_n \in \X_*(\bT)$ and $h\in G(y)_{y,0}$ such that
\begin{itemize}
    \item $h\inv n \sigma(h) = \lambda_n(\xi)$ and
    \item $y_n = y + \lambda_n/\ell_n \in \overline{C(y)}'$.
\end{itemize}
  If we define   $g_n = \lambda_n(\pi) h\inv \lambda_n(\pi)\inv \in \bG(y)(K_n)_{y_n,0}$, then $n$ and  $g_n\inv \sigma(g_n) \in N_{\bG(y)(K_n)}(\bA)$ have the same image in $\absW(y)$.  The maximal $K$-torus $\lsup{g_n}\bA = \lsup{g_n}\Ebtorus$  of $\bG(y)$ is $K$-minisotropic and the unique point in $\BB^{\red}(G(y))$ associated to it is $y_n$.

\begin{example}   \label{ex:so4intameG2} We offer a cautionary example to show that the point identified in $\bar{C}'$ for a given $G$-conjugacy class of maximal $K$-minisotropic tori might not be the point identified in $\overline{C(y)}'$ for a corresponding $G(y)$-conjugacy class.
Adopt the notation of Example~\ref{ex:tameG2}. Suppose $y$ is the vertex at the right angle in  Figure~\ref{fig:g2embeddings}; i.e., the vertex there with the label $x_{-1}$.  In Figure~\ref{fig:so4gtwo} we identify the unique point, denoted $y_{-1}$, in $\overline{C(y)}'$ corresponding to the Coxeter element $-1 \in A_1 \times \tilde{A}_1$ in the Weyl group of $\bG(y) \cong \SO_4$.  The dotted lines show the $\Gtwo$-facet structure, and the shaded region is $C'$.  Note that $y$ and $y_{-1}$ are $\Gtwo$-conjugate.

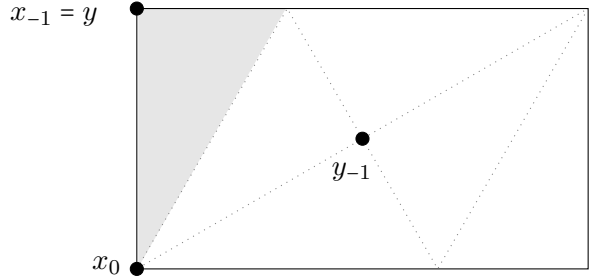
\begin{figure}[ht]
\centering
\begin{tikzpicture}

    \fill[fill=gray!20] (0.02,0.02)--(1.98,3.462)--(.02,3.462);

\draw (0,0)
  -- (6,0)
   -- (6,3.464)
  -- (0,3.464)
  -- cycle;

\draw[dotted, gray] (0,0)
  -- (2,3.464)
  -- (4,0)
  -- (6,3.464)
  -- cycle;

   \draw[black,fill=black] (0,0) circle (.5ex);
\draw(-.4,-.2) node[anchor=south]{$x_0$};

 \draw[black,fill=black] (0,3.464) circle (.5ex);
\draw(-1.1,3.1) node[anchor=south]{$x_{-1} = y$};

  \draw[black,fill=black] (3,1.732) circle (.5ex);
\draw(2.86,1.55) node[anchor=north]{$y_{-1}$};

\end{tikzpicture}
\caption{An embedding of the building of a $K$-minisotropic maximal torus in $\SO_4$ (inside $\Gtwo$) \label{fig:so4gtwo}}
\end{figure}

\end{example}

\begin{remark}  We continue to assume that $\bG$ is $K$-split.  Observe that for $y \neq x_0$, the pinning $(\bG,  \bA = \Ebtorus , \bB, \{X_a\}_{a \in \Pi})$ of $\bG$ from Section~\ref{sec:realization} does not define a pinning of $\bfG_y$.  However, if $\bB(y)$ denotes the Borel subgroup of $\bG(y)$ generated by $\bA$ and those root groups  $\bU_{\alpha}$  where $ \alpha$ is the gradient of some $\psi \in \Delta(y)$ with $\psi(y) = 0$, then  it is possible to choose a pinning $(\bG(y), \bA = \Ebtorus, \bB(y), \{Y_a\}_{a \in \Pi(y)})$ of $\bG(y)$ that is compatible, in the obvious sense, with $y$ such that $Y_b = X_b$ for all  $b \in \Pi \cap \Pi(y)$.  In fact, we can fix  pinnings  in such a way that if $y, y'$ are two vertices in $\bar{C}'$ with  pinnings $(\bG(y), \bA, \bB(y), \{Y_a\}_{a \in \Pi(y)})$  and $(\bG(y'), \bA, \bB(y'), \{Y'_a\}_{a \in \Pi(y')})$ that are compatible with $y$ and $y'$ respectively, then $Y_b = Y'_b$ for all  $b \in \Pi(y) \cap \Pi(y')$.
\end{remark}

\begin{remark} \label{rem:thingsbreakdown}
 If $\bG$ is not $K$-split, then one needs to be more careful.  For example, suppose $y$ is a non-special vertex in the building $\BB(\bG,K)$ where $\bG = \SU(4)$ splits over a quadratic extension $E$ of $K$.  In this case $y$ is special in $\BB(\bG,E)$, and there does not exist  a full rank reductive subgroup $\bH$ of $\bG$ such that $y$ is absolutely special in $\BB(\bH,K)$ and $\bfG_y = \bfH_y$.   In general,  one might hope to choose a subgroup $\bH$ of $\bG$ such that (a) the $K$-rank of $\bH$ is the same as the $K$-rank of $\bG$, (b)  $y$ is absolutely special in $\BB(\bH,K)$, and (c)  we have $\bfG_y = \bfH_y$; this is being investigated in~\cite{ekstrom:thesis}.
\end{remark}

\section{On the existence of \texorpdfstring{$K$-minisotropic maximal $k$-tori of $\bG$}{K-minisotropic maximal k-tori of G}}  \label{sec:existenceKktori}

From now until the end of the paper we assume that $k$ is a nonarchimedean local field; in particular, $\ff$ is finite.  Recall that $\Fr$ is an element of $\Gal(\bar{k}/k)$ that lifts a topological generator of $\Gal(K/k)$.

In this section we provide, in Lemma~\ref{lem:overk}, a criterion for determining when the $G$-orbit of a $K$-minisotropic maximal $K$-torus  of $\bG$ contains a torus defined over $k$.
Recall that $\mathcal{C}$ denotes the set of $G$-conjugacy classes of maximal $K$-tori in $\bG$.

\begin{lemma}  \label{lem:Frobonorbits}
If $\dorbit \in \mathcal{C}$, then $\Fr(\dorbit) \in \mathcal{C}$.
\end{lemma}

\begin{proof}
Suppose $\bT \in \dorbit \in \mathcal{C}$. Choose a strongly regular semisimple $\gamma \in T$.   Note that $\Fr(\bT) = C_{\bG}(\Fr(\gamma))$ is again a maximal $K$-torus of $\bG$.  Since $\Fr(G) = G$, the result follows.
\end{proof}

\begin{lemma}  \label{lem:Fr(T)isOK} If $\bT$ is a $K$-minisotropic maximal torus in $\bG$, then the torus $\Fr(\bT)$ is also a $K$-minisotropic maximal torus in $\bG$.
\end{lemma}

\begin{proof}   For all $\delta \in \Gal(\bar{k}/K)$ we have $\delta(\Fr(\bT)) = \Fr (\delta' (\bT)) = \Fr (\bT)$ where $\delta ' = \Fr \inv \delta \Fr \in \Gal(\bar{k}/K)$.  Thus, $\Fr(\bT)$ is a maximal $K$-torus in $\bG$.    An element  $\chi \in \X^*(\bT)$ is fixed by $\Gal(\bar{k}/K)$ if and only if $\Fr(\chi) \in \X^*(\Fr(\bT)) $ is fixed by $\Gal(\bar{k}/K)$; so $\Fr(\bT)$ is also a $K$-minisotropic torus.
\end{proof}

\begin{lemma}   \label{lem:overk}
Assume that $\bT$ is a $K$-minisotropic maximal torus in $\bG$, not necessarily defined over $k$.
 Let $\dorbit$ denote the $G$-orbit of $\bT$.   There exists a torus in $\dorbit$  that is $\Fr$-stable if and only if $\Fr(\dorbit) = \dorbit$.
\end{lemma}

\begin{proof}
If there is a torus in $ \dorbit$ that is $\Fr$-stable, then without loss of generality we may assume $\Fr(\bT) = \bT$.  This immediately implies  $\Fr(\dorbit) = \dorbit$.

Suppose $\Fr(\dorbit) = \dorbit$.
Without loss of generality the point $x_T$ corresponding to $\bT$ belongs to $\bar{C}'$.
Since  $\Fr(\dorbit) = \dorbit$, there exists $h \in G$ such that $\Fr(\bT) = \lsup{h}\bT$.  Since $\bar{C}'$ is $\Fr$-stable, we have $\Fr(x_T)  \in \bar{C}'$.  On the other hand, by uniqueness we have  $h \cdot  x_T = \Fr(x_T)$. Thanks to Lemma~\ref{lem:card}  we conclude that $h \cdot x_T = x_T$ and so $h \in \Stab_G(x_T)$.   From Lemma~\ref{lemma:KHR} we know $\Stab_G(x_T) = G_{x_T,0} T$, so without loss of generality we may assume $h \in G_{x_T,0}$.  From Lang-Steinberg, there exists $k \in G_{x_T,0}$ for which $h\inv = k\inv \Fr(k)$.   So $\Fr(T) = \lsup{h}\bT = \lsup{\Fr(k\inv) k} \bT$, which implies $\Fr(\lsup{k}\bT) = \lsup{k}\bT$.
\end{proof}

\begin{cor}  \label{cor:rationalinclass} Suppose $\gamma \in G$ is $K$-elliptic and strongly regular semisimple.    If the $G$-orbit of $\gamma$ is $\Fr$-stable, then the $G$-orbit of $\gamma$ contains a Frobenius-fixed point.
\end{cor}

\begin{proof}
Let $\bT$ denote the  centralizer of $\gamma$.  Since $\Fr(\bT)$ is the centralizer of $\Fr(\gamma)$ and $\Fr(\gamma)$ is $G$-conjugate to $\gamma$, we conclude that the $G$-orbit of $\bT$ is $\Fr$-stable.  Thanks to Lemma~\ref{lem:overk} there exists $h\in G$ such that $\lsup{h}T$ is Frobenius fixed.  Replacing $\gamma$ with $\lsup{h}\gamma$ we may and do assume that $\gamma \in T$ and $\bT$ is defined over $k$.  Note that $\Fr(\gamma) \in T$ and $\Fr(\gamma) = \lsup{n}\gamma$ for some $n \in G$.  Since $\gamma$ is strongly regular semisimple, we must have $n \in N_G(T) \leq \Stab_G(x_T)$.     From Lemma~\ref{lemma:KHR} we know $\Stab_G(x_T) = G_{x_T,0} T$, so without loss of generality we may assume $n \in G_{x_T,0}$.  From Lang-Steinberg, there exists $j \in G_{x_T,0}$ for which $n\inv = j\inv \Fr(j)$.   The element $\lsup{j}\gamma$ is Frobenius fixed.
\end{proof}

\begin{remark}
When the derived group of $\bG$ is simply connected, Corollary~\ref{cor:rationalinclass} may be derived from~\cite[Theorem~4.1 and Lemma~3.3]{kottwitz:rational}.
\end{remark}

\section{On the existence of  \texorpdfstring{$K$-minisotropic maximal $k$-tori of $\bG$}{K-minisotropic maximal k-tori of G} in the tame situation}  \label{sec:existenceintame}

 Note that $\Fr \inv \sigma \Fr = \sigma^q$.   Moreover, $\Fr$ acts on the set of tame extensions of $K$ as well as $\absW$ and $\Ebtorus$.

\subsection{A simple criterion}

In this section we  translate, in the tame setting, the criterion of Lemma~\ref{lem:overk} into a condition on the $\sigma$-conjugacy classes in the Weyl group.

\begin{lemma}  \label{lem:Frobonorbitstame}
If $\dorbit \in \Ctame$, then $\Fr(\dorbit) \in \Ctame$.
\end{lemma}

\begin{proof}
Suppose $\bT \in \dorbit \in \Ctame$.    Let $E$ be a tame extension of $K$ over which $\bT$ splits.
Thanks to Lemma~\ref{lem:Frobonorbits} we know that $\Fr(\bT)$ is again a maximal $K$-torus of $\bG$.  Since $\Fr(\bT)$ splits over the tame extension $\Fr(E)$ of $K$, we have $\Fr(\dorbit) \in \Ctame$.
\end{proof}

Since $\sigma^q$ is a generator for $\Gal(\tame/K)$, it also stabilizes every tame extension of $K$; in particular, it stabilizes $\Esplits$.
 Let $\Wtamesimq$ denote the set of $\sigma^q$-conjugacy classes in $\Wtame$ and let $\varphi_{\sigma^q} \colon \Wtamesimq \rightarrow \mathcal{C}$ denote the resulting  map from the set of $\sigma^q$-conjugacy classes in $\Wtame$ to the set of $G$-conjugacy classes of maximal $K$-tori in $\bG$.

For $d \in \Z_{\geq 1}$ and $w \in\absW$ set
$$N_d(w) = w \cdot \sigma(w) \cdots \sigma^{(d-2)}(w) \cdot \sigma^{(d-1)} (w).$$ The map $N_q$ induces  a well-defined map from  $W_{\sim_{\sigma}}$ to $W_{\sim_{\sigma^q}}$ and $\Fr$ defines a map from $W_{\sim_{\sigma^q}}$ to $W_{\sim_{\sigma}}$.

\begin{lemma}   \label{lem:equivmapnotsplitexperiment}
The diagram below commutes.
\begin{center}
\begin{tikzcd}
\Wtamesim \arrow[r, "N_q"] \arrow[d, "\varphi_\sigma"]
&\Wtamesimq \arrow[r, "\Fr"] \arrow[d, "\varphi_{\sigma^q}"]
&\Wtamesim \arrow[d,  "\varphi_{\sigma}" ] \\
\Ctame \arrow[r,  "\Id" ]
& \Ctame  \arrow[r, "\Fr"]
& \Ctame  \\
\end{tikzcd}
\end{center}
Moreover, each map is a bijection.
\end{lemma}

\begin{proof}
The statement (and proof) of Lemma~\ref{lem:imageofphi} holds when $\sigma$ is replaced by any generator of $\Gal(\tame/K)$, so we conclude that all three vertical arrows are bijections.

Suppose $c$ is a $\sigma$-conjugacy class in $\Wtame$.
  Then there exists $g \in \bG(\tame)$ such that the image of $g\inv \sigma(g)$ in $\absW$ is an element of $c$.   We have $(\Id \circ \varphi_{\sigma})(c) =  \{ \lsup{h g} \Ebtorus \, | \, h \in G \}$.  The image of
$ g\inv \sigma^q(g) = (g\inv \sigma(g)) \sigma(g\inv \sigma(g)) \cdots \sigma^{q-1}(g\inv \sigma(g))$
in $\absW$ belongs to $N_q(c)$ and so
$$(\varphi_{\sigma^q} \circ N_q)(c) = \varphi_{\sigma^q}(N_q(c)) = \{ \lsup{h g} \Ebtorus \, | \, h \in G \}.$$
Thus $\varphi_{\sigma^q} \circ N_q = \Id \circ \varphi_\sigma$.

Suppose $\tilde{c}$ is a $\sigma^q$-conjugacy class in $\Wtame$.
Let $\tilde{\dorbit} = \varphi_{\sigma^q}(\tilde{c}) \in \mathcal{C}$.      There exists $\tilde{g} \in \bG(\tame)$ such that the image of $\tilde{g}\inv \sigma^q(\tilde{g})$ in $\absW$ is an element of $\tilde{c}$.   We have $(\Fr \circ \varphi_{\sigma^q})(\tilde{c}) =  \Fr(\tilde{\dorbit}) = \{ \lsup{h \Fr(\tilde{g})} \Ebtorus \, | \, h \in G \}$.
The image of
$ \Fr(\tilde{g}\inv \sigma^q(\tilde{g})) = \Fr(\tilde{g})\inv \sigma(\Fr(\tilde{g}))$ in $\absW$ belongs to $\Fr(\tilde{c})$ and so
$$(\varphi_\sigma \circ \Fr)(\tilde{c}) = \varphi_\sigma( \Fr(\tilde{c})) = \{ \lsup{h \Fr(\tilde{g})} \Ebtorus \, | \, h \in G \}.$$
Since this is $\Fr(\tilde{\dorbit})$, we have $\varphi_{\sigma} \circ \Fr = \Fr \circ \varphi_{\sigma^q}$.

Since $G$ is $\Fr$-stable, the map $\Fr\inv \colon \Ctame \rightarrow \Ctame$ is the inverse to $\Fr \colon \Ctame \rightarrow \Ctame$.  The result follows.
\end{proof}

\begin{cor}   \label{cor:tameoverkexperiment}   Suppose $c \in \Wtamesim$ is $\sigma$-elliptic.
  There exists a torus in $ \varphi_\sigma (c) $  that is $\Fr$-stable if and only if $c = (\Fr \circ N_q )(c)$.
\end{cor}

\begin{proof}
From Lemma~\ref{lem:equivmapnotsplitexperiment} we know that $\Fr(\varphi_\sigma (c) ) = \varphi_\sigma (c) $ if and only if $c = (\Fr \circ N_q )(c)$.  The result follows from Lemmas~\ref{lem:elliptictoweyl} and~\ref{lem:overk}.
\end{proof}

\subsection{An application of the rationality of Weyl groups}   \label{sec:rational}

A finite group $H$ is said to be \emph{rational} provided that for any $h \in H$ we have that $h^j$ is $H$-conjugate to $h$ whenever $j$ is relatively prime to the order of $h$.   There are many equivalent definitions of rationality, and the nomenclature is explained by the fact~\cite[Section 13.1]{serre:linear} that $H$ is rational if and only if every irreducible character of $H$ takes values in $\Q$.

It is known that every Weyl group is a rational group (see, for example,~\cite[Lemma~2.1.2]{debacker-haley:kac},~\cite[Section 3 of Chapter 2 and Chapter 5]{kletzing:structure}, or~\cite[Theorem~8.5]{springer:regular}).

\begin{cor}   \label{cor:equivmaptame}
If $\bG$ is  $K$-split, then
the diagram
\begin{center}
\begin{tikzcd}
\Wtamesimnon \arrow[r, "\Fr"] \arrow[d, "\varphi_{\sigma}"]
&\Wtamesimnon \arrow[d,  "\varphi_{\sigma}" ] \\
 \Ctame  \arrow[r, "\Fr"]
& \Ctame
\end{tikzcd}
\end{center}
commutes and  each map is a bijection.  Here $\Wtamesimnon$ is the set of $\absW$-conjugacy classes in $\Wtame$. Moreover, the vertical maps are independent of the choice of topological generator $\sigma$ of $\Gal(\tame /K)$.  When $\bG$ is  $k$-split, all maps are independent of the choice of topological generators $\Fr$ for $\Gal(K/k)$ and $\sigma$ for $\Gal(\tame /K)$.
\end{cor}

\begin{proof}
Since $\bG$ is $K$-split, we have $\bA = \Ebtorus$ and $\sigma$ acts trivially on $\absW$.   Thus, $N_q(\dot{w}) = \dot{w}^q$ for all $\dot{w} \in\absW$.

Choose $w \in \Wtame$.
Since $\absW$ is a rational group and $q$ is relatively prime to the order of $w$, we have that $w^{q}$ and $w$ are conjugate in $\absW$.  Thus $N_q(c) = c$ for all $c \in \Wtamesim$.
The claims about the commutativity of the diagram and bijectivity of its maps follow from Lemma~\ref{lem:equivmapnotsplitexperiment}.

 We now show that the vertical maps are independent  of the choice of the topological generator $\sigma$.  Let $\sigma'$ be another choice of topological generator for $\Gal(\tame/K)$.   Suppose $w \in c \in \Wtamesimnon$.   Choose $g \in \bG(\tame)$ such that $g\inv \sigma(g) \in N_{\bG(\tame)}(\bA)$ has image $w$ in $\absW$.   Let $E/K$ be the splitting field of $\lsup{g}\bA$.   Set $\ell = [E:K]$.  Without loss of generality, we may assume that $g \in \bG(E)$.  Note that for all $k \in \Z_{\geq 1}$,  the elements $(g\inv \sigma(g))^k$ and $g\inv \sigma^k(g)$ of $N_{\bG(E)}(\bA)$ have the same image in $\absW$, namely $w^k$.    Putting $k = \ell$, we conclude that  $w^{\ell} = 1$.  Let $\bar{\sigma}$ (resp. $\bar{\sigma}'$) denote the image of $\sigma$ (resp. $\sigma'$) in $\Gal(E/K)$.   Since $\bar{\sigma}'$ generates $\Gal(E/K)$, there exists $m \in \Z_{\geq 1}$ with $(m,\ell) = 1$ such that $\bar{\sigma}' = \bar{\sigma}^m$.   Putting $k = m$, we conclude that the image $g\inv \sigma'(g) \in N_{\bG(E)}(\bA)$ is $w^m$.   Since $m$ is relatively prime to the order of $w$ and  $\absW$ is a rational group,  $w$ and $w^m$ are conjugate in $\absW$.  Thus $w^m \in c$.  We conclude that $\varphi_{\sigma} (c) = \varphi_{\sigma'} (c) = \lsup{G}(\lsup{g}\bA)$.

Suppose now that $\bG$ is $k$-split. In this case any choice of $\Fr$ acts trivially on $\absW$; hence the horizontal maps are both the identity.   Hence, when $\bG$ is $k$-split, the maps in the diagram are independent of the choice of both $\Fr$ and $\sigma$.
\end{proof}

\begin{remark}  Suppose $\bG$ is not $K$-split and $\sigma'$ is another choice of topological generator for $\Gal(\tame/K)$.  If the image of $\sigma$ and $\sigma'$ in $\Aut(\Esimple)$ agree, then by arguing as in the proof of Corollary~\ref{cor:equivmaptame}  it follows from~\cite[Proposition~2.6]{adams-he-nie:from}  that  $\varphi_\sigma$ and 
$\varphi_{\sigma'}$ induce the same map from $\Wtamesimnon$ to $\Ctame$.
\end{remark}

\begin{remark}   \label{rem:G2example} If $\bG$ is $k$-split, then from Corollary~\ref{cor:equivmaptame} the condition $c = (\Fr \circ N_q)(c)$  of Corollary~\ref{cor:tameoverkexperiment} holds for every tame elliptic conjugacy class $c$ of the Weyl group.   However, if we remove the  assumption that $\bG$ is $k$-split, then this condition is no longer guaranteed to hold.   For example, suppose $k'$ is an unramified degree three extension of $k$ and let $\bG = R_{k'/k} \Gtwo$.  As a $K$-group, we have $\bG_K \cong \Gtwo \times \Gtwo \times \Gtwo$, and $\Fr$-acts by cyclically permuting the three copies of $\Gtwo$.  By placing one of each of the three distinct elliptic conjugacy classes $A_1 \times \tilde{A}_1$, $A_2$, and $\Gtwo$ on each of the three $\Gtwo$ factors of $\bG$, we obtain an elliptic conjugacy class in the Weyl group of $R_{k'/k} \Gtwo$ which is not $\Fr \circ N_q$-stable.
\end{remark}

\subsection{Cautionary example}

Suppose $\bT$ is a maximal $K$-torus in $\bG$. We know from our work above that if $\bT$ is $K$-minisotropic and $\Fr(\lsup{G}\bT) = \lsup{G}\bT$, then   $\lsup{G}\bT$ contains a $k$-torus.  At the other extreme, we also know  that if $\bT$ is a maximally $K$-split maximal $K$-torus in $\bG$, then  $\lsup{G}\bT$ is $\Fr$-stable and,  since every maximal $k$-split torus of $\bG$ is contained in a maximally $K$-split maximal $k$-torus of $\bG$, $\lsup{G}\bT$ contains a $k$-torus.

These examples and the results above suggest that there might be an elegant general theory relating the existence of  $k$-tori in an orbit $\dorbit$ in $\Ctame$  to the $\Fr$-invariance of $\dorbit$. In this section we present an example showing that the natural condition $\Fr (\dorbit) = \dorbit$ does not guarantee that $\dorbit$ contains a $k$-torus.

\begin{example}
Suppose that the residual characteristic of $k$ is larger than three and let $\bH$ be a connected reductive group of type $A_{2}$ such that $H^{\Fr} \cong \SL_1(D)$ where $D$ is a division algebra of index $3$ over $k$.  Recall that we may identify  $H = \bH(K)$ with $\SL_3(K)$.  Let $\bA$ denote a maximal $K$-split $k$-torus in $\bH$; it is unique up to $H^\Fr$-conjugacy.  Let $\absW = N_{\bH}(\bA)/\bA$.  Suppose $C$ is the alcove in $\AA(A) \subset \BB(H)$ for which $C^{\Fr} \neq \emptyset$.  Let $\{\psi_0, \psi_1 ,  \psi_{2} \}$ be the simple affine $K$-roots determined by $\bH$, $\bA$, $\nu$, and $C$.   We assume that the $\psi_i$ are labeled such that $\Fr(\psi_i) = \psi_{i+1}$ mod $3$.

Let $\bT$ denote a maximal $K$-torus in $\bH$ that corresponds, under Corollary~\ref{cor:CandW}, to the $\absW$-conjugacy class of the simple reflection $w_{\dot{\psi}_1}$.   Since this $\absW$-conjugacy class is $\Fr$-stable, from Corollary~\ref{cor:equivmaptame} we conclude that $\Fr(\bT)$ is $H$-conjugate to $\bT$.  Since $\bT$ splits over a quadratic extension of $K$ and all maximal $k$-tori in $\bG$ correspond to extensions $E \leq D$ whose ramification degree over $k$ is not $2$, we conclude that $\lsup{H}\bT$ cannot contain a $k$-torus.
\end{example}

\subsection{Results when \texorpdfstring{$\bG$ is $K$-tame}{G is K-tame}}  We gather here the statements of the main results of this section under the hypothesis that $\bG$ is $K$-tame.  Recall that when $\bG$ is $K$-tame we have $\absW = \Wtame$ and $\mathcal{C} = \Ctame$ (see Definition~\ref{defn:Ktame} and Corollary~\ref{cor:CandW}).    The following three results follow immediately from Lemma~\ref{lem:equivmapnotsplitexperiment}, Corollary~\ref{cor:tameoverkexperiment}, and Corollary~\ref{cor:equivmaptame}.

\begin{lemma}   \label{lem:equivmapnotsplit}
Suppose $\bG$ is $K$-tame.  The diagram below commutes.
\begin{center}
\begin{tikzcd}
\absW_{\sim_{\sigma}} \arrow[r, "N_q"] \arrow[d, "\varphi_\sigma"]
&\absW_{\sim_{\sigma^q}} \arrow[r, "\Fr"] \arrow[d, "\varphi_{\sigma^q}"]
& \absW_{\sim_{\sigma}} \arrow[d,  "\varphi_{\sigma}" ] \\
\mathcal{C} \arrow[r,  "\Id" ]
& \mathcal{C}  \arrow[r, "\Fr"]
& \mathcal{C}  \\
\end{tikzcd}
\end{center}
Moreover, each map is a bijection. \qed
\end{lemma}

\begin{cor}   \label{cor:tameoverk}  Suppose that $\bG$ is $K$-tame.   Suppose $c \in\absW_{\sim_{\sigma}}$ is $\sigma$-elliptic.
  There exists a torus in $ \varphi_\sigma (c) $  that is $\Fr$-stable if and only if $c = (\Fr \circ N_q )(c)$. \qed
\end{cor}

\begin{cor}   \label{cor:equivmap}
If $\bG$ is $K$-tame and $K$-split, then
the diagram
\begin{center}
\begin{tikzcd}
\absW_{\sim} \arrow[r, "\Fr"] \arrow[d, "\varphi_{\sigma}"]
&\absW_\sim \arrow[d,  "\varphi_{\sigma}" ] \\
 \mathcal{C}  \arrow[r, "\Fr"]
& \mathcal{C}
\end{tikzcd}
\end{center}
commutes and each map is a bijection.  Here $\absW_\sim$ is the set of $\absW$-conjugacy classes in $\absW$.  When $\bG$ is $K$-tame and $k$-split, all maps are independent of the choice of generators $\Fr$ for $\Gal(K/k)$ and $\sigma$ for $\Gal(\tame /K)$. \qed
\end{cor}

 \begin{cor} \label{cor:5.4.4}
  If $\bG$ is $K$-tame and $k$-split, then  every $G$-orbit of  $K$-minisotropic maximal $K$-tori in $\bG$ contains a torus defined over $k$.
 \end{cor}

 \begin{proof}
Since $\bG$ is $k$-split, the top horizontal map of the diagram in Corollary~\ref{cor:equivmap} becomes the identity.  The result follows from Lemma~\ref{lem:overk}.
 \end{proof}

\section{Results about \texorpdfstring{$K$-minisotropic maximal $k$-tori in $\bG$}{K-minisotropic maximal k-tori in G}} \label{sec:resultsonKktori}

Suppose that $\bT$ is a $K$-minisotropic maximal $k$-torus of $\bG$.
 The set of $\Fr$-stable tori in the $G$-orbit of $\bT$ is denoted $\CTk$.  In this section we parameterize three sets related to  $\CTk$:    we parameterize the $G^\Fr$-conjugacy classes in $\CTk$;   we parameterize the $k$-stable-conjugacy classes in $\CTk$ (see Definition~\ref{defn:kstability});  and we  parameterize, up to $G^\Fr$-conjugacy, the $k$-embeddings of $\bT$ into $\bG$ (see Definition~\ref{defn:embedding}).

Recall from \S\ref{subsec:point} that we can associate a
point $x_T$ in $\BB^{\red}(G)$ to $T$.
Since $\bT$ is defined over $k$, we have  $\Fr (T) = T$ and so by uniqueness we conclude that $x_T$ is Frobenius fixed.  Thus, we can assume   $x_T  \in \bar{C}'^\Fr \subset  \BB^{\red}(G)^\Fr = \BB^{\red}(\bG,k)$.   Recall that $F$ is the unique facet in $\BB(G)$ whose image in $\BB^{\red}(G)$  contains $x_T$.

\subsection{Points in the reduced building associated to  \texorpdfstring{$K$-minisotropic maximal $k$-tori of $\bG$}{K-minisotropic maximal k-tori of G}}

Suppose $\bT$ is a $K$-minisotropic maximal $k$-torus of $\bG$.

\begin{lemma}   \label{lem:cardk}   
The set
$$G^{\Fr} \cdot x_T \cap \bar{C}'{}^{\Fr}$$
has cardinality one.
\end{lemma}

\begin{remark}  The proof below also works if $k$ is not residually quasi-split, but we must replace the alcove $C'^{\Fr}$ in $\BB^{\red}(\bG,k)$ with an alcove $D$ in $\BB^{\red}(\bG,k)$ that is contained in the closure of ${C}'$.   The statement would then read:
$$\bG(k) \cdot x_T \cap \bar{D} \, \text{
has cardinality one.}$$
\end{remark}

\begin{proof}
From Lemma~\ref{lem:card} we know that
$G \cdot x_T \cap \bar{C}'$ has cardinality one, and so $G^{\Fr} \cdot x_T \cap \bar{C}'{}^{\Fr}$ has at most one point.  Since $x_T \in \bar{C}'{}^{\Fr}$, the result follows.
\end{proof}

\begin{lemma}  \label{lem:xtconj} Suppose $\bT'$ is a maximal $k$-torus in $\bG$.  If $\bT'$ is $G$-conjugate to $\bT$, then there exist $h \in G^{\Fr}$ and $k' \in \Stab_{G}(x_T)$ such that $\lsup{k' h} \bT' = \bT$.
\end{lemma}

\begin{proof}
Since $\bT'$ is $G$-conjugate to $\bT$, we have that $\bT'$ is a  $K$-minisotropic maximal $k$-torus in $\bG$.  Thus, we may associate to $\bT'$ a point $x_{T'}$ in $\BB^{\red}(\bG,k)$.   Choose $h \in G^{\Fr}$ such that $hx_{T'} \in \bar{C}'^{\Fr}$.

From Lemma~\ref{lem:card} we know $hx_{T'} = x_T$.   Since $\bT$ and $\bT'$ are $G$-conjugate, there exists $k' \in G$ such that $\bT = \lsup{k' h} \bT'$.  So $x_T = k' h x_{T'} = k' x_T$, which implies $k' \in \Stab_G(x_T)$.
\end{proof}

\begin{lemma} \label{lem:FrpointsofGorbit}  Recall that in this section $x_T$ is Frobenius fixed.
    \[ G \cdot x_T \cap \BB^{\red}(G)^{\Fr} = G^{\Fr} \cdot x_T\]
\end{lemma}

\begin{proof}
    It is enough to show that   \( G \cdot x_T \cap \BB^{\red}(G)^{\Fr} \subset G^{\Fr} \cdot x_T\).  Suppose $g \in G$ such that $g \cdot x_T$ is Frobenius fixed.  Then $g \cdot x_T$ belongs to some alcove of $\BB^{\red}(G)^{\Fr}$ and so there exists $h \in G^{\Fr}$ such that $h\inv g \cdot x_T \in \bar{C}'^{\Fr}$.  From Lemma~\ref{lem:card} the point $h\inv g \cdot x_T$ must be $x_T$ and so ${g}\cdot x_T = h \cdot x_T \in G^{\Fr} \cdot x_T$.
\end{proof}

\subsection{\texorpdfstring{$k$-embeddings of $\bT$ into $\bG$}{k-embeddings of T into G}}

For many questions in harmonic analysis we want to understand the ways to $k$-embed $\bT$ into $\bG$ up to $G^\Fr$-conjugacy.

 \begin{defn}  \label{defn:embedding} Suppose $\bS$ is a maximal $k$-torus in $\bG$.   A \emph{$k$-embedding} of $\bS$ into $\bG$ is a map $f \colon \bS \rightarrow \bG$ such that
 \begin{enumerate}
     \item there exists $g \in G$ such that $f(s) = {\lsup{g}s}$ for all $s \in S$ and
     \item $f$ is a $k$-morphism.
 \end{enumerate}
 \end{defn}

\begin{rem}  
    In Definition~\ref{defn:embedding} it would be more standard to take $g \in \bG(\bar{k})$.  However, since $f$ is completely determined by what it does to any strongly regular semisimple element of $T^\Fr$, from Lemma~\ref{lem:gside} we can restrict our attention to $g \in G$.
\end{rem}

If $f \colon \bT \rightarrow \bG$ is a $k$-embedding, then $f$ is completely determined by where it sends any given strongly regular semisimple element of $T^\Fr$.    So, we begin by studying strongly regular semisimple elements.

\begin{definition}
Suppose $\gamma \in G^{\Fr}$ is strongly regular semisimple.  An element $\gamma' \in G^\Fr$ is \emph{$k$-stably-conjugate} to $\gamma$ provided that there exists $g \in G$ such that $\lsup{g}\gamma = \gamma'$.  Let $\Ogammak$ denote the set of elements of $G^{\Fr}$ that are $k$-stably-conjugate to $\gamma$.
\end{definition}

\begin{definition}
$T_F := T \cap G_{F,0}$
\end{definition}

In the notation of Corollary~\ref{cor:2.3.4} we have that $T_F = \eta[T_{\scon}]T_0$.  
In general, $T_F$ is neither $T$ nor is it the parahoric subgroup $T_0$.  For example, for $p>2$ and a $K$-minisotropic maximal  $k$-torus in $\SL_2$ that splits over a quadratic ramified extension, we have $T/T_F$ is trivial and $T_F/T_0$ is isomorphic to $\Z / 2 \Z$.  However,  for $p > 2$ and a $K$-minisotropic maximal $k$-torus in $\PGL_2$ that splits over a quadratic ramified extension, we have $T/T_F$  is isomorphic to $\Z / 2 \Z$ and $T_F/T_0$ is trivial.

\begin{defn}
For a $\Fr$-module $A$, let $A_{\Fr}$ denote the $\Fr$-coinvariants.  That is
$$A_\Fr =  A/ (1 - \Fr)A.$$
\end{defn}

Thanks to Lang-Steinberg, we have $T_0 \leq (1-\Fr)T_F \leq T_F$.  We can have  $T_0 \neq (1-\Fr)T_F$ and $(1-\Fr)T_F \neq T_F$; to see this suppose $p > 3$ and consider  ramified elliptic maximal $k$-tori in $\SL_3$ for various choices of $k$.

\begin{definition}
Let $\bar{\bfT}_F := T_F/T_0$.
\end{definition}

Since $(1- \Fr)T_0 = T_0$, we have $\bar{\bfT}_F \cong (T_F)_{\Fr}$.   We will often use this identification.

\begin{lemma} \label{lem:K-stableclasses2orig}
Suppose $\gamma \in T^{\Fr}$ is strongly regular semisimple.  There is natural bijection between the set of $G^{\Fr}$-conjugacy classes in $\Ogammak$   and the Frobenius coinvariants of $T_F$.
That is
$$  \Ogammak /  \text{$G^{\Fr}$-conjugacy}  \, \, \longleftrightarrow \, \, (\bar{\bfT}_F)_{\Fr}.$$
\end{lemma}

\begin{remark}
Lemma~\ref{lem:K-stableclasses2orig} implies that $(\bar{\bfT}_F)_{\Fr}$ is a finite group.   This also follows from the fact that  $\bar{\bfT}_F^\circ = (1-\Fr)\bar{\bfT}_F^\circ$.
\end{remark}

\begin{proof}
If $\gamma' \in \Ogammak$, then  there is an $h \in G$ such that $\gamma' = \lsup{h} \gamma$.   Since $\bT$ is $k$-stably-conjugate to $C_{\bG} (\gamma')$, thanks to Lemma~\ref{lem:xtconj} there is a $g \in G^{\Fr}$ such that the point in $\BB^{\red}(G)$ attached to both  $\bT$ and $\lsup{g} (C_{\bG} (\gamma'))$ is $x_T$.   That is, we may replace $\gamma'$ by $\lsup{gh}\gamma$ and assume, thanks to Lemma~\ref{lemma:KHR}, that $\gamma' = \lsup{k} \gamma$ for $k \in G_{F,0}$.   Since $\lsup{k}\gamma = \Fr(\lsup{k}\gamma) = \lsup{\Fr(k)}\gamma$, we have $k\inv \Fr(k) \in T_F$.  Moreover, if $\gamma' = \lsup{k'} \gamma$ for $k' \in G_{F,0}$, then $k\inv k' \in T_F$ which implies that $k' =  ks$ for some $s \in T_F$.  Hence $(k')\inv \Fr(k')  = s\inv (k\inv \Fr(k)) \Fr( s) =  (k\inv \Fr(k)) (s\inv \Fr(s))$.  That is, we have a well-defined map $\lambda \colon \Ogammak \rightarrow T_F /(1 - \Fr)T_F = (T_F)_{\Fr}$.

On the other hand, if $t \in T_F$, then, thanks to Lang-Steinberg, there is a $k \in G_{F,0}$ such that $k\inv \Fr(k) = t$.  Note that $\lsup{k}\gamma \in \Ogammak$ and $\lambda(\lsup{k}\gamma) $ is the image of $t$ in $T_F /(1 - \Fr)T_F$.

To complete the proof, we need to show that if $\gamma_1, \gamma_2 \in \Ogammak$, then $\gamma_1$ is $G^{\Fr}$-conjugate to $\gamma_2$ if and only if $\lambda(\gamma_1) = \lambda(\gamma_2)$.  As above, after possibly conjugating by an element of $G^{\Fr}$, we can write $\gamma_i = \lsup{k_i}\gamma$ with $k_i \in G_{F,0}$.

Suppose  $\lambda(\gamma_1) = \lambda(\gamma_2)$.  Then there exists $s \in T_F$  such that
$ (k_1\inv \Fr(k_1)) (s\inv \Fr(s)) = k_2\inv \Fr(k_2)$.
Note that $\gamma_1 = \lsup{k_1 s k_2\inv} \gamma_2$ and $\Fr(k_1 s k_2\inv) = k_1 s k_2\inv$, so $\gamma_1$ is $G^{\Fr}$-conjugate to $\gamma_2$.

Suppose $\gamma_1$ is $G^{\Fr}$-conjugate to $\gamma_2$.  Then, thanks to Lemma~\ref{lem:xtconj}, there exists $h \in G^{\Fr}_{F,0}$ such that $\lsup{h} \gamma_1 = \gamma_2$.    Note that $k_2\inv h k_1 \in T_F$.  Thus, there is a $t \in T_F$ such that $k_2\inv h k_1 =t\inv$, or $k_2 = h k_1 t$.   Note that
$$k_2\inv \Fr(k_2) = t\inv k_1\inv h \inv \Fr(h) \Fr(k_1) \Fr(t) = (k_1\inv \Fr(k_1)) (t\inv \Fr(t)) .$$
Hence,  $\lambda(\gamma_1) = \lambda(\gamma_2)$.
\end{proof}

 \begin{cor}   \label{cor:onembbeddings}
 Suppose $t_1, t_2, \ldots , t_d \in T_F$ represent the elements of $(T_F)_{\Fr}$ and choose $k_i \in G_{F,0}$ such that $k_i\inv \Fr(k_i) = t_i$.  Then
 $$\{f_i\colon \bT \rightarrow \bG \, | \, \text{ $f_i(t) = \lsup{k_i}t$ for $t \in \bT$ }\}$$
is a complete set of representatives for the $G^{\Fr}$-conjugacy classes of $k$-embeddings  of $\bT$ into $\bG$.
 \end{cor}

 \begin{proof}
 If $\gamma \in T^{\Fr}$ is strongly regular semisimple, then the $k$-embeddings of $\bT$ into  $\bG$ are in one-to-one correspondence with the $G^\Fr$-conjugacy classes in $\Ogammak$.
 \end{proof}

\subsection{\texorpdfstring{$K$-minisotropic maximal tori in $\bG$ and the $\Omega$ group}{K-minisotropic maximal tori in G and the Omega group}}
\label{subsec:omega}

The original aim of this subsection was to prove Corollary~\ref{cor:rationaltori}, which is used in the proof of Lemma~\ref{lem:K-stableclasses2} .  The derivation of Corollary~\ref{cor:rationaltori} gave rise to some interesting results about the  $\Omega$ group.  We develop  some of these results here.

Recall that   $C$ is a $\Fr$-stable alcove in $ \BB(G)$ such that $F \subset \overline{C}$ and $x_T$ belongs to the image of $F$ in $\BB^{\red}(G)$.
In this section we show that $T/T_F$ is isomorphic to  $\Omega = \Stab_G(C)/G_{C,0}$ and that $\Stab_{G^{\Fr}}(F) = T^{\Fr} G_{F,0}^{\Fr}$.

We begin by recalling a known fact, though I am not sure where it is proved in the literature.

\begin{lemma} \label{lem:moregeneralinjection} Suppose $H$ is a facet in $\BB(G)$ with $H \subset \overline{C}$.   There is a natural  group homomorphism
$$\varphi \colon \Stab_G(H) \rightarrow \Omega$$
with kernel $G_{H,0}$.  Moreover, if $H$ is $\Fr$-stable, then $\varphi$ is $\Fr$-equivariant and  $\varphi$ descends to a group homomorphism
$$\varphi \colon \Stab_{G^{\Fr}}(H) \rightarrow \Omega^{\Fr},$$
which has kernel $G^{\Fr}_{H,0}$.
\end{lemma}

\begin{proof}
Suppose  $t \in \Stab_G(H)$.  Since $H \subset \bar{C}$, $t \cdot H = H$, and all Borel subgroups in a connected reductive group are conjugate,  there exists $k \in G_{H,0}$ such that $kt \cdot C = C$; that is, $kt \in \Stab_G(C)$.   Fix such a $k$.   If $k' \in G_{H,0}$ also has the property that $k't \cdot C   = C$, then
$$(k't) (kt)\inv = k'k\inv \in G_{H,0}$$
and
$$k'k\inv \cdot C = (k't)(kt)\inv  \cdot C = C.$$
Since parabolic subgroups of a reductive group are self-normalizing (see also~\cite[Lemma~4.2.1]{debacker:bruhat-tits}), we have $k'k\inv \in G_{C,0}$.   So, $(k't)(kt)\inv \in G_{C,0}$.  Thus, as $G_{C,0}$ is normal in $\Stab_G(C)$, we have  $k't \in G_{C,0} kt = kt G_{C,0}$.  Consequently, we have a  function $\varphi \colon \Stab_{G}(H) \rightarrow \Stab_{G}(C)/G_{C,0}$ defined by $t \mapsto ktG_{C,0}$.

To see that $\varphi$ is a group homomorphism, fix $t,t' \in \Stab_G(H)$ and choose $k,k' \in G_{H,0}$ such that $\varphi(t) = kt G_{C,0}$ and $\varphi(t') = k't' G_{C,0}$.   A calculation shows that  $\varphi(t) \varphi(t') = (k \lsup{t}k') (tt') G_{C,0}$.  Since $(k \lsup{t}k') (tt') C = C$, with $k \lsup{t}k' \in G_{H,0}$, we conclude that $\varphi(tt') = \varphi(t) \varphi(t')$.

By construction, the group homomorphism $\varphi$ has kernel $G_{H,0}$.

Suppose now that $H$ is $\Fr$-stable.  We have $\Fr(\varphi(t)) = \Fr(k) \Fr(t) \Fr (G_{C,0})$.  Since $H$ and $C$ are Frobenius stable, we have $\Fr(k) \in G_{H,0}$ and $\Fr(G_{C,0}) = G_{C,0}$.  It follows that $\Fr(\varphi(t)) = \Fr(k) \Fr(t) G_{C,0} = \varphi(\Fr(t))$.

That $\varphi$ descends to a map on Frobenius-fixed points
$$\varphi \colon \Stab_{G^{\Fr}}(H) \rightarrow
\Stab_{G^{\Fr}}(C)/G^{\Fr}_{C,0}$$
with kernel $G_{H,0}^\Fr$ follows by repeating the initial part of this argument with all the groups replaced by their Frobenius-fixed points.
\end{proof}

\begin{lemma}  \label{lem:mapfromtorus}  There is a natural surjective, $\Fr$-equivariant, group homomorphism  from $ T$  to $\Omega$ with kernel
$T_F$.
\end{lemma}

\begin{proof}
From Lemma~\ref{lem:moregeneralinjection} there is a $\Fr$-equivariant map $\varphi \colon \Stab_G(F) \rightarrow \Omega$ with kernel $G_{F,0}$. Since $T \leq \Stab_G(F)$ and $T$ is $\Fr$-stable, we may restrict $\varphi$ to $T$ to obtain  a $\Fr$-equivariant group homomorphism $\varphi \colon T \rightarrow \Omega$.

We first show that $T_F = \ker(\varphi)$.  If $t \in T_F$, then $t\inv \in G_{F,0}$ and $t\inv t \cdot C = C$.  So, $\varphi(t) = t\inv t G_{C,0} = G_{C,0}$, and we conclude that $t \in \ker(\varphi)$.  On the other hand, if $t \in \ker(\varphi)$, then there exists $k \in G_{F,0}$ such that $kt \cdot G_{C,0} = G_{C,0}$.  This means $t G_{C,0} = k\inv G_{C,0} \subset G_{F,0}$.   We conclude that $t \in G_{F,0}$, hence $t$ is in $T_F$.

We now show that $\varphi$ is surjective.   Suppose $\omega \in \Omega$ and choose $g \in \Stab_G(C)$ such that $\omega = g G_{C,0}$.  From Lemma~\ref{lem:card} we conclude that  $gF = F$.  Hence, from Lemma~\ref{lemma:KHR} there exist $t \in T$ and $k \in G_{F,0}$ such that $g = kt$.  We conclude that $\varphi(t) = \omega$.
\end{proof}

\begin{cor}
 The point $x_T$ is fixed by $\Omega$.
\end{cor}

\begin{proof}
Suppose $g \in \Stab_{G}(C)$.  From Lemma~\ref{lem:mapfromtorus}, there exist $t \in T$ and $k \in G_{F,0}$ such that $tk = g$. Since $k x_T = x_T$ and $t x_T = x_T$, we conclude that $g x_T = x_T$.
\end{proof}

\begin{remark} \label{rem:moreinwildcase} Since the point $x_T$ is independent of isogeny type (see Remark~\ref{rem:indofisogeny}), we see that $x_T$ must be invariant under the action of $\Omega_{\ad} = \Stab_{G_{\ad}}(C)/(G_{\ad})_{C,0}$.   So, for example, for groups of type $A_n$ the point $x_T$ must be the barycenter of $C'$.   
\end{remark}

\begin{remark}
If $\bG$ is almost simple, then  $\Fr(x_T) = x_T$.    In the tame case, this follows from~\cite[Proposition~6.8]{adams-he-nie:from}.    To see that this holds in general, express $x_T$ in barycentric coordinates: $x_T = \sum_{\psi \in \Delta} {x_\psi}v_\psi$ where the $v_\psi$ are the vertices of $\bar{C}$,  $x_\psi \geq 0$, and $\sum x_\psi = 1$.  As noted in Remark~\ref{rem:moreinwildcase}, we have  $\omega(x_T) = x_T$ for all $\omega \in \Omega_{\ad}$.  Thus,  $x_\psi = x_{\omega(\psi)}$ for all $\omega \in \Omega_{\ad}$ and $\psi \in \Delta$.  By case-by-case checking, one sees that this implies that $x_\psi = x_{\tau(\psi)}$ for all automorphisms $\tau$ of the affine Dynkin diagram.  Since $\Fr(C) = C$, we conclude that $\Fr(\Delta) = \Delta$, and hence $\Fr$ induces an automorphism of the affine Dynkin diagram.  Thus, $x_{\Fr(\psi)} = x_\psi$ for all $\psi \in \Delta$, and so $\Fr(x_T) = x_T$. If $\bG$ is tame, this implies that every $G$-conjugacy class of $K$-minisotropic maximal $K$-tori contains a torus defined over $k$. If $\bG$ is not almost simple, then the example discussed in Remark~\ref{rem:G2example} shows that there exists $x_T$ such that $\Fr(x_T) \neq x_T$.
\end{remark}

\begin{lemma}  \label{lem:countingkernelpoints}
We have
$$(\bar{\bfT}_F)_{\Fr} \cong T_F \cdot (1 - \Fr)T/(1-\Fr)T.$$
\end{lemma}

\begin{remark}  \label{rem:cohominterp}
    Since taking torsion points is left exact while $\cohom^1(\Fr,\bar{\bfT}_F) \cong ((\bar{\bfT}_F)_{\Fr})_{\tor}$ and $\cohom^1(\Fr,T) \cong (T_{\Fr})_{\tor}$,  we conclude from Lemma~\ref{lem:countingkernelpoints} that the inclusion $T_F \rightarrow T$ induces an injection $\cohom^1(\Fr,\bar{\bfT}_F) \rightarrow \cohom^1(\Fr,T)$.
\end{remark}

\begin{proof}
Since $T_0 = (1-\Fr)T_0$, we have $(\bar{\bfT}_F)_{\Fr} \cong T_F/(1-\Fr)T_F = (T_F)_{\Fr}$.

Fix a strongly regular semisimple $\gamma \in T^{\Fr}$.  As proved in Lemma~\ref{lem:K-stableclasses2orig}, the map $\lambda \colon \Ogammak \rightarrow T_F /(1 - \Fr)T_F = (T_F)_{\Fr}$ descends to a bijective map
from $\Ogammak  /  \text{$G^{\Fr}$-conjugacy}$ to    $(T_F)_{\Fr}$.
Define
$$\mu \colon T_F/(1-\Fr)T_F   \rightarrow T_F \cdot (1 - \Fr)T/(1-\Fr)T $$
by $\mu(t (1-\Fr)T_F) = t(1-\Fr)T$ for $t(1-\Fr)T_F \in (T_F)_\Fr$. This map is surjective, and so the map $\tilde{\lambda} := \mu \circ \lambda$ from  the set of $G^{\Fr}$-conjugacy classes in  $\Ogammak$ to  $ T_F \cdot (1 - \Fr)T/(1-\Fr)T $ is surjective.   It will be enough to show that $\tilde{\lambda}$ is injective.

Suppose $\gamma_1, \gamma_2 \in \Ogammak$ and $\tilde{\lambda}(\gamma_1) = \tilde{\lambda}(\gamma_2)$.  After possibly conjugating by an element of $G^{\Fr}$, we can write $\gamma_i = \lsup{k_i}\gamma$ with $k_i \in G_{F,0}$.
Since $\tilde{\lambda}(\gamma_1) = \tilde{\lambda}(\gamma_2)$,  there exists $s \in T$  such that
$ (k_1\inv \Fr(k_1)) (s\inv \Fr(s)) = k_2\inv \Fr(k_2)$.
Note that $\gamma_1 = \lsup{k_1 s k_2\inv} \gamma_2$ and $\Fr(k_1 s k_2\inv) = k_1 s k_2\inv$, so $\gamma_1$ is $G^{\Fr}$-conjugate to $\gamma_2$.
\end{proof}

\begin{lemma}  \label{lem:coinvariantsexact}
The sequence
$$1 \longrightarrow (\bar{\bfT}_F)_{\Fr} \longrightarrow (T/T_0)_{\Fr} \longrightarrow \Omega_{\Fr} \longrightarrow 1$$
is exact.
\end{lemma}

\begin{proof}
Since
$$1 \longrightarrow \bar{\bfT}_F \longrightarrow T/T_0 \longrightarrow \Omega \longrightarrow 1$$
is an exact sequence of $\Fr$-modules and taking coinvariants is right exact, it is enough to check that the map $(\bar{\bfT}_F)_{\Fr} \rightarrow (T/T_0)_{\Fr} $ is injective.  Note that by exactness the image of $(\bar{\bfT}_F)_{\Fr}$ in $(T/T_0)_{\Fr}$  is $\ker((T/T_0)_{\Fr} \rightarrow \Omega_{\Fr})$.

We first show that $T_F \cdot (1-\Fr)T / (1-\Fr)T = \ker((T/T_0)_{\Fr} \rightarrow \Omega_{\Fr})$.

``$\supset$'' Suppose $t \in T$ such that the image of $t$ in $(T/T_0)_{\Fr}$ belongs to $\ker((T/T_0)_{\Fr} \rightarrow \Omega_{\Fr})$.  As in the proof of Lemma~\ref{lem:mapfromtorus} we may choose $k \in G_{F,0}$ such that $kt \in \Stab_G(C)$.  Since  the image of $t$ in $(T/T_0)_{\Fr}$ belongs to $\ker((T/T_0)_{\Fr} \rightarrow \Omega_{\Fr})$ and  $(1-\Fr)G_{F,0} = G_{F,0}$,  we have $kt \in g\inv \Fr(g) G_{C,0}$ for some $g \in \Stab_G(C)$.  Again using  Lemma~\ref{lem:mapfromtorus}  we can write $g = k' s$ where $k' \in G_{F,0}$ and $s \in T$.   So, $kt \in s\inv (k')\inv \Fr(k') \Fr(s) G_{C,0}$.  Since $G_{C,0} \leq G_{F,0}$ and $\Fr(s) F = F$, after some manipulation we arrive at
$$t s \Fr(s\inv) = s t s \Fr(s\inv) s\inv \in \lsup{s}k\inv k'{}\inv \Fr(k') \lsup{\Fr(s)}G_{C,0} \subset G_{F,0}.$$
Consequently, $t \in T_F \cdot (1-\Fr)T$.   Thus, since $(1-\Fr)T_0 = T_0$, we have that $\ker((T/T_0)_{\Fr} \rightarrow \Omega_{\Fr})$ is contained in
$T_F \cdot (1-\Fr)T / (1-\Fr)T$.

``$\subset$''
On the other hand, if $\bar{t} = t (1-\Fr)T$ with $t \in T_F$, then from Lemma~\ref{lem:mapfromtorus}  we know that the image of $t$ in $\Omega$ is trivial.
Thus $T_F \cdot (1-\Fr)T / (1-\Fr)T = \ker((T/T_0)_{\Fr} \rightarrow \Omega_{\Fr})$.

From Lemma~\ref{lem:countingkernelpoints} we conclude that $\ker((T/T_0)_{\Fr} \rightarrow \Omega_{\Fr})$ and $({T}_F)_\Fr$ have the same finite cardinality.   The result follows.
\end{proof}

\begin{remark}
If $\bG$ is semisimple, then
the natural injections $T \rightarrow T$ and $\Stab_G(C) \rightarrow G$ induce isomorphisms
$$ (T/T_0)_{\Fr}  = \tor( (T/T_0)_{\Fr}) \cong \cohom^1(\Fr,T)  \text{ and }   \Omega_{\Fr}   = \tor(\Omega_{\Fr}) \cong \cohom^1(\Fr,G).$$
Under these isomorphisms, the map $(T/T_0)_{\Fr} \rightarrow \Omega_{\Fr}$ becomes $ \cohom^1(\Fr,T) \rightarrow  \cohom^1(\Fr,G)$.
Since the cardinality of $\ker( \cohom^1(\Fr,T) \rightarrow  \cohom^1(\Fr,G))$ measures the number of rational classes in the $k$-stable class of any strongly regular semisimple element of $T^\Fr$,   \,    Lemma~\ref{lem:coinvariantsexact} follows from Lemma~\ref{lem:K-stableclasses2orig}.
\end{remark}

\begin{cor} \label{cor:rationaltori}
$$\Stab_{G^\Fr}(F) = T^{\Fr} G_{F,0}^{\Fr}.$$
\end{cor}

\begin{remark}
This can also be stated as $\Stab_{G^\Fr}(x_T) = T^{\Fr} G_{x_T,0}^{\Fr}$.
\end{remark}

\begin{proof}
Suppose $g \in \Stab_{G^\Fr}(F)$.  From Lemma~\ref{lemma:KHR} we may choose $t \in T$ and $k \in G_{F,0}$ such that $g = tk$.  Since $tk = g = \Fr(g) = \Fr(t) \Fr(k)$, we conclude that $t\inv \Fr(t) \in T_F$.  Since $t\inv \Fr(t) \cdot T_0$ maps to the identity in $\Omega$, from Lemma~\ref{lem:coinvariantsexact} it  is in the image of the map  $(\bar{\bfT}_F)_{\Fr} \rightarrow (T/T_0)_{\Fr} $.
Thus, we conclude that there exists $s \in T_F \leq G_{F,0}$ such that $t\inv \Fr(t) = s\inv \Fr(s)$ mod $T_0$; replacing $t$ by $t s\inv$ we conclude that $t\inv \Fr(t) \in T_0$.   Since $\cohom^1(\Fr, T_0)$ is trivial, there exists $r \in T_0$ such that $t\inv \Fr(t) = r\inv \Fr(r)$.   Replacing $t$ by $t r\inv$, we conclude that $t \in T^{\Fr}$.  Note that $g \in t G^{\Fr}_{F,0}$.
\end{proof}

\begin{proof}[Alternate Proof]   The short exact sequence 
$$1 \longrightarrow T_F \longrightarrow T \times G_{F,0} \longrightarrow \Stab_G(F)  \longrightarrow 1$$
gives rise to the long exact sequence
$$1 \longrightarrow T^{\Fr}_F \longrightarrow T^{\Fr} \times G^{\Fr}_{F,0} \longrightarrow \Stab_{G^{\Fr}}(F)  \longrightarrow \cohom^1(\Fr,T_F) \longrightarrow \cohom^1(\Fr, T \times G_{F,0}).$$
Since $\cohom^1(\Fr, T \times G_{F,0}) \cong \cohom^1(\Fr, T)$ and, by Remark~\ref{rem:cohominterp},  $\cohom^1(\Fr,T_F)$ injects into $\cohom^1(\Fr, T)$, the result follows.
\end{proof}

\begin{defn}  \label{defn:6.3.12}
Set $\dorbit_{\bT, x_T} = \Ad({G_{x_T,0}})\bT$ and let $\dorbit_{\bT, x_T}^k$ denote the set of $\Fr$-stable tori in  $\dorbit_{\bT, x_T}$.
\end{defn}

\begin{cor}  \label{cor:rational_orbits_agree}
The injection $\dorbit_{\bT, x_T}^k \hookrightarrow \dorbit_{\bT}^k$ induces a bijection between the set of $G_{x_T,0}^{\Fr}$-conjugacy classes in $\dorbit^k_{\bT,x_T}$ and the set of $G^{\Fr}$-conjugacy classes in $\dorbit^k_{\bT}$. 
\end{cor}

\begin{proof}
Since every $G_{x_T,0}^{\Fr}$-conjugacy class in $\dorbit^k_{\bT,x_T}$ determines a $G^{\Fr}$-conjugacy class in $\dorbit^k_{\bT}$, it is enough to show that if
$\mathcal{C}$ is a $G^{\Fr}$-conjugacy class in $\dorbit_{\bT}^k$, then $\mathcal{C} \cap \dorbit^k_{\bT,x_T}$ is a single $G_{x_T,0}^{\Fr}$-conjugacy class.

Suppose $\mathcal{C}$ is a $G^{\Fr}$-conjugacy class in $\dorbit_{\bT}^k$.  We first show that $\mathcal{C} \cap \dorbit^k_{\bT,x_T}$ is nonempty.
Choose $\bT' \in \mathcal{C}$.  Let $x_{T'}$ denote the (unique) point in $\BB^{\red}(G)^{\Fr}$ associated to $\bT'$.  Since $\bT' = \lsup{g}\bT$ for some $g \in G$, we conclude that $x_{T'} = g \cdot x_T$.  Since $G^{\Fr}$ acts transitively on the alcoves in $\BB^{\red}(G)^{\Fr}$, there exists $h \in G^{\Fr}$ such that $h  \cdot x_{T'} = x_{\lsup{h}T'}$ belongs to $\bar{C}'^{\Fr}$. From Lemma~\ref{lem:card} we conclude that $h \cdot x_{T'} = hg \cdot x_T = x_T$.
Since $\lsup{h} \bT' = \lsup{h g} \bT$ and $hg \cdot x_T = x_T$, from Lemma~\ref{lemma:KHR} we conclude that there exists $j \in G_{x_T,0}$ such that $\lsup{h}\bT' = \lsup{j} \bT$.  That is, $\mathcal{C} \cap \dorbit^k_{\bT,x_T}$ is nonempty.

Suppose $\mathscr{C}_1, \mathscr{C}_2 \subset \mathcal{C} \cap \dorbit^k_{\bT,x_T}$ are two $G_{x_T,0}^{\Fr}$-conjugacy classes.   Choose $\bT_i \in \mathscr{C}_i$.   Since $\bT_1, \bT_2 \in \mathcal{C}$, there exists $g \in G^{\Fr}$ such that $\lsup{g}\bT_1 = \bT_2$.  Since $x_{T_1} = x_{T_2} = x_T$, we conclude that $g \in \Stab_{G^{\Fr}}(x_T)$.  From Corollary~\ref{cor:rationaltori} we conclude that there exists $j \in G_{x_T,0}^{\Fr}$ such that $\lsup{j}\bT_1 = \bT_2$.  That is, $\mathscr{C}_1 = \mathscr{C}_2$.
\end{proof}

\begin{cor}  \label{cor:propertyofcoxeter1}
 If $F = C$, then for every $\Fr$-stable facet $H$ contained in the closure of $C$ we can assume that the representatives for $\Stab_{G^{\Fr}}(H)/G^{\Fr}_{H,0}$ lie in $T^{\Fr}$.
\end{cor}

\begin{proof} Suppose $H$ is a $\Fr$-stable facet that is contained in the closure of $C$.  If $h \in \Stab_{G^{\Fr}}(H)$, then there exists $k \in G^{\Fr}_{H,0}$ such that $hk C = C$.  From Corollary~\ref{cor:rationaltori} we can assume that $hk \in T^{\Fr} G^\Fr_{C,0}$.   Thus,
\[ 
h G^{\Fr}_{H,0} = hk G^{\Fr}_{H,0} \subset T^{\Fr} G^{\Fr}_{C,0} G^{\Fr}_{H,0} = T^{\Fr}G^{\Fr}_{H,0}. \qedhere
\]
\end{proof}

A similar proof shows:
\begin{cor} \label{cor:propertyofcoxeter} Suppose $\bT$ is any $K$-minisotropic maximal $K$-torus, not assumed to be defined over $k$.
 If $F = C$, then for every  facet $H$ contained in the closure of $C$ we can assume that the representatives for $\Stab_{G}(H)/G_{H,0}$ lie in $T$. \qed
\end{cor}

\subsection{\texorpdfstring{$k$-stable conjugacy and $G^{\Fr}$-conjugacy of tori}{k-stable conjugacy and G conjugacy of tori}}
In this section we analyze the $\Fr$-structure of $\lsup{G}T$ where $\bT$ is a  $K$-minisotropic maximal $k$-torus.   Recall that   $\CTk$ denotes the set of  $k$-tori in $\bG$ that are $G$-conjugate  to $\bT$.

\begin{definition} \label{defn:kstability}
If $\bT', \bT'' \in \CTk$, then  we will say that $\bT'$ is \emph{$k$-stably-conjugate} to $\bT''$ provided that there is a $g \in G$ such that $\lsup{g}\bT' = \bT''$ and $\Ad(g) \colon T' \rightarrow T''$ is a $k$-morphism.  If $\bT', \bT'' \in \CTk$,  we will say that $\bT'$ is \emph{$G^{\Fr}$-conjugate} to $\bT''$ provided that there is a $g \in G^{\Fr}$ such that $\lsup{g}\bT' = \bT''$.
\end{definition}

Note that if two tori $\bT_1, \bT_2 \in \CTk$ are  {$k$-stably-conjugate}, then there is an $x \in G$ such that $\lsup{x}(\bT_1(k) ) =\bT_2(k)$.

\begin{rem}  
    As in Definition~\ref{defn:embedding} it would be more standard to define the notion of  $k$-stably-conjugate in Definition~\ref{defn:kstability} by taking $g \in \bG(\bar{k})$.  However,  Lemma~\ref{lem:gside} allows us to restrict our attention to $g \in G$.
\end{rem}

We  first describe   $\CTk/ \text{$G^{\Fr}$-conjugacy} $, the set of rational conjugacy classes in $\CTk$.  For this we introduce the group $\bar{W}_T$ and the notion of Frobenius conjugacy.

\begin{definition}
Set $W_T = N_G (\bT) / T$  and $\bar{W}_T =  N_{G_{F,0}} (\bT) / T_0$ where $T_0$ is the parahoric subgroup of $T$.
\end{definition}

\begin{remark} Note that, in general, $\bar{W}_T$ is not isomorphic to either $W_T = N_G (\bT) / T$ or  $N_G (\bT) / T_F$.  Indeed, for $\mathrm{SL}_2$ and $p>2$  we have that
$\bar{W}_T$ has order four while both $W_T$ and   $N_G (\bT) / T_F = N_{G_{F,0}} (\bT) / T_F$ have order two.
 \end{remark}

\begin{definition}
Suppose $H$ is a group on which $\langle \Fr \rangle$ acts.   Two elements $h, h' \in H$ are said to be \emph{Frobenius conjugate} provided that there exists $x \in H$ such that $x\inv h \Fr(x) = h'$.  We denote the set of  Frobenius-conjugacy classes of $H$  by $H_\simFr$.
\end{definition}

\begin{rem}
If $H$ is abelian, then $H_\simFr = H_{\Fr}$.
\end{rem}

\begin{lemma} \label{lem:K-stableclasses2}
There is a natural bijection between the set of $G^\Fr$-conjugacy classes in $\CTk$   and the set of $\Fr$-conjugacy classes in $\bar{W}_T$.  That is
$$  \CTk/ \text{$G^{\Fr}$-conjugacy}   \, \, \longleftrightarrow \, \,  (\bar{W}_T)_\simFr.$$
\end{lemma}

\begin{rem}
    Since taking coinvariants is right exact and $T_0 = (1-\Fr)T_0$, we have $(\bar{W}_T)_\simFr \cong (N_{G_{F,0}}(T))_\simFr$.
\end{rem}

\begin{proof}
Suppose $\bT' \in \CTk$.    From Lemma~\ref{lem:xtconj}  we may assume, after conjugating by an element of $G^{\Fr}$, that  $\bT$ and $\bT'$ are conjugate, as $K$-tori, by an element of $\Stab_{G}({F})$.  Thanks to Lemma~\ref{lemma:KHR}, we may assume that $\bT$ and $\bT'$ are, in fact, conjugate by an element of $G_{F,0}$.  Choose $k \in G_{F,0}$ such that $\lsup{k}\bT = \bT'$.  Since both $\bT$ and $\bT'$ are $k$-tori, they are $\Fr$-stable.  Thus, $k\inv \Fr(k) \in N_{G_{F,0}}(\bT)$.   If $k' \in G_{F,0}$ is another element  such that $\lsup{k'}\bT = \bT'$, then $m:= k\inv k' \in N_{G_{F,0}}(\bT)$ and $(k')\inv \Fr(k')  =   m\inv ( k\inv \Fr(k) ) \Fr(m)$.    Consequently, after taking the quotient by $T_0$, we have a well-defined map $\rho \colon \CTk \rightarrow (\bar{W}_T)_\simFr$.

 On the other hand, if $\bar{m} \in \bar{W}_T$, then we can choose $m \in N_{G_{F,0}}(\bT)$ lifting $\bar{m}$.  Thanks to Lang-Steinberg, there is a $k \in G_{F,0}$ such that $k\inv \Fr(k) = m$.  Note that $\lsup{k}\bT \in \CTk$ and $\rho(\lsup{k}\bT) $ is the $\Fr$-conjugacy class of $\bar{m}$.  That is, $\rho$ is surjective.

To complete the proof, we need to show that if $\bT_1, \bT_2 \in \CTk$, then $\bT_1$ is $G^{\Fr}$-conjugate to $\bT_2$ if and only if $\rho(\bT_1) = \rho(\bT_2)$.  As above, after possibly conjugating by an element of $G^{\Fr}$, we can write $\bT_i = \lsup{g_i}\bT$ with $g_i \in {G_{F,0}}$.

Suppose $\rho(\bT_1) = \rho(\bT_2)$.  Then there exist $m \in N_{G_{F,0}}(\bT)$ and $t \in T_0$ such that
$$m g_1\inv \Fr(g_1) \Fr(m)\inv = t g_2\inv \Fr(g_2) .$$
Since $g_2 t g_2 \inv \in \lsup{g_2}T_0$, by Lang-Steinberg, there is a $k \in \lsup{g_2}T_0$ such that $k\inv \Fr(k) = g_2 t g_2\inv$.  Thus, $$\Fr(kg_2 m g_1\inv ) =kg_2 m g_1\inv $$
showing that $\bT_1$ and $\bT_2$ are conjugate by $kg_2mg_1\inv  \in G^{\Fr}$.  

Suppose $\bT_1$ is $G^\Fr$-conjugate to $\bT_2$.
Then thanks to Corollary~\ref{cor:rationaltori} with $\bT = \bT_1$,  there exists $h \in G^{\Fr}_{F,0}$ such that $\lsup{h} (\bT_1(k)) = \bT_2(k)$.    Thus, for all strongly regular semisimple $\gamma \in \bT_1(k)$ we have $\lsup{h} \gamma \in \bT_2(k)$.  Since $\lsup{hg_1}\bT = \lsup{g_2}\bT$, we have $m := g_2 \inv h g_1 \in N_{G_{F,0}} (\bT)$.   Thus
\begin{equation*}
g_2\inv \Fr(g_2) = mg_1\inv h\inv \Fr( h g_1 m\inv) = m g_1\inv \Fr( g_1) \Fr(m\inv).
\end{equation*}
We conclude that $\rho(\bT_1) = \rho(\bT_2)$.
\end{proof}

\begin{remark}
It would be nice if the intersection of a Frobenius-conjugacy class in $\bar{W}_T$ with $\bar{\bfT}_F$ would lie in $ (1-\Fr)\bar{\bfT}_F$, giving us an injective map $(\bar{\bfT}_F)_{\Fr}  \rightarrow (\bar{W}_T)_\simFr$.  However, this is not true.  For example, when $p > 2$ and $-1 \not \in (k^\times)^2$ consider  a ramified elliptic maximal $k$-torus in  $\SL_2$.
\end{remark}

We wish to describe   $\CTk/\! \approx$, the set of $k$-stable-conjugacy classes in $\CTk$.

\begin{lemma} \label{lem:K-stableclasses3}
There is natural bijection between the set of $k$-stable conjugacy classes in $\CTk$   and the set of Frobenius conjugacy classes in  $W_T = N_{G}(\bT)/T$.
That is
$$  \CTk/ \! \approx  \, \, \longleftrightarrow \, \, ({W}_T)_\simFr .$$
\end{lemma}

\begin{proof}
Note that $W_T$ may be naturally identified, as groups with $\Fr$-action, with $\bar{W}_T/ \bar{\bfT}_F$.

As in the proof of  Lemma~\ref{lem:K-stableclasses2}, we have a surjective, well-defined map $\mu \colon \CTk \rightarrow (\bar{W}_T/ \bar{\bfT}_F)_\simFr$.

To complete the proof, we need to show that if $\bT_1, \bT_2 \in \CTk$, then $\bT_1$ is $k$-stably-conjugate to $\bT_2$ if and only if $\mu(\bT_1) = \mu(\bT_2)$.  Thanks to Lemma~\ref{lemma:KHR} and Lemma~\ref{lem:xtconj}, after conjugating by an element of $G^{\Fr}$, we can write $\bT_i = \lsup{k_i}\bT$ with $k_i \in {G_{F,0}}$.

Suppose  $\mu(\bT_1) = \mu(\bT_2)$.  Then there exist $m \in N_{G_{F,0}}(\bT)$ and $t \in T_F$ such that
$$m k_1\inv \Fr(k_1) \Fr(m\inv) = t\inv k_2\inv \Fr(k_2).$$
For all $\gamma \in \bT_1(k)$ we have
$\Fr(\lsup{k_2 m k_1\inv}\gamma) = \lsup{k_2 t m k_1\inv} \gamma$.  Since $\lsup{mk_1\inv} \gamma \in \bT(K)$, we have $\lsup{tmk_1\inv}\gamma = \lsup{mk_1\inv}\gamma$.  Thus,
$\Fr(\lsup{k_2 m k_1\inv}\gamma) = \lsup{k_2 m k_1\inv} \gamma$, and we conclude that $\bT_1$ is $k$-stably-conjugate to $\bT_2$.

Suppose $\bT_1$ is $k$-stably-conjugate to $\bT_2$.  Then there exists $h \in G_{F,0}$ such that $\lsup{h}( \bT_1(k)) = \bT_2(k)$.    Thus, for all strongly regular semisimple $\gamma \in \bT_1(k)$ we have $\lsup{h} \gamma \in \bT_2(k)$.   Thus $s = h\inv \Fr(h) \in \bT_1(K) = \lsup{k_1}(\bT(K))$ and $\lsup{k_1\inv}s \in T_F$.   Since $\lsup{hk_1}\bT = \lsup{k_2}\bT$, we have $m := k_2 \inv h k_1 \in N_{G_{F,0}} (\bT)$.   Thus
\begin{equation*}
k_2\inv \Fr(k_2) = (mk_1\inv h\inv) \Fr(h k_1 m\inv) = (\lsup{mk_1\inv}s) (m k_1\inv  \Fr(k_1) \Fr(m\inv)) .
\end{equation*}
Since $\lsup{mk_1\inv}s \in T_F$, we conclude that $\mu(\bT_1) = \mu(\bT_2)$.
\end{proof}

Different choices of ``base point'' $\bT$ for the orbit $\lsup{G} \bT$ will result in different realizations of the action of $\Fr$.   We explore how these different realizations are related.

\begin{defn}
Suppose $\bT' \in \CTk$.  Choose $h \in G$ such that $T' = \lsup{h}T$.  Set $m_h = h\inv \Fr(h)$ and let $\bar{m}_h$ denote the image of $m_h$ in $W_T$.   Let $W_{T,\bar{m}_h}$ denote the group $W_T$, but with $\Fr$ acting by $\Fr \cdot x = \lsup{\bar{m}_h} \Fr(x)$.   Similarly, we let $\bar{T}_{F,\bar{m}_h}$ (resp., $(T/T_{0^+})_{\bar{m}_h}$) denote the group $\bar{T}_F$ (resp. $T/T_{0^+}$), but with $\Fr$ acting by $\Fr \cdot t = \lsup{\bar{m}_h} \Fr(t)$.
\end{defn}

\begin{lemma}  \label{lem:nottransportofstructure} Recall that $\bT' = \lsup{h}\bT \in \CTk$ with $h \in G$.  Let $F'$ denote the facet corresponding to $T'$.
\begin{enumerate}
    \item  \label{lem:fiberW} The group homomorphism that sends $m' \in N_{G}(\bT')$ to $h\inv m' h \in N_{G}(\bT)$ induces a $\Fr$-equivariant isomorphism
    $$\varphi_h \colon W_{T'} \rightarrow W_{T,\bar{m}_h}.$$
\item  \label{lem:fiberT} The group homomorphism that sends $t' \in T'$ to $h\inv t' h \in T$ induces  $\Fr$-equivariant isomorphisms
    \[\pushQED{\qed}  \varphi_h \colon T'/T'_{0^+} \rightarrow (T/T_{0^+})_{\bar{m}_h} \text{ and } \varphi_h \colon \bar{T}'_{F'} \rightarrow \bar{T}_{F,\bar{m}_h}. \qedhere  \popQED \]
\end{enumerate}
\end{lemma}

\subsection{The rational structure of \texorpdfstring{$\lsup{G}\bT$:}{G orbit of T:} a summary}
\label{subsec:summary}  In this subsection we gather some of the key results and definitions of the previous subsections.   Recall that $\bT$ is a $K$-minisotropic maximal $k$-torus, and $\CTk$ denotes the set of tori in $\lsup{G}\bT$ that are defined over $k$.

 \begin{itemize}
     \item  \textbf{$k$-stable classes in $\CTk$}. Two  tori $\bT_1, \bT_2 \in \CTk$ are said to be $k$-stably-conjugate provided that there is an $x \in G$ such that $\lsup{x}t_1 \in T_2^\Fr$ for all $t_1 \in T_1^\Fr$.  The $k$-stable classes in $\CTk$ are parameterized by $(W_T)_{\simFr}$, the set of Frobenius-conjugacy classes in $N_G(\bT)/T$.  (Lemma~\ref{lem:K-stableclasses3})

 \item  \textbf{$k$-embeddings of $\bT$}. If $\bT' \in \CTk$, then a $k$-embedding of $\bT'$ into $\bG$ is a $k$-morphism $f \colon \bT' \rightarrow \bG$ that arises via conjugation for some $x \in G$.   Up to $G^\Fr$-conjugation, the set of $k$-embeddings of $\bT$ into $\bG$ is parameterized by
  $(\bar{\bfT}_F)_{\Fr}$, the group of $\Fr$-coinvariants in $\bar{\bfT}_F = (T \cap G_{F,0})/T_0$.    If $\bT' \in \CTk$ belongs to the stable class indexed by $w \in c \in (W_T)_{\simFr}$, then the set of $k$-embeddings of $\bT'$ into $\bG$ is parameterized by  $(\bar{\bfT}_F)_{w\Fr}$, the group of $w \Fr$-coinvariants in $\bar{\bfT}_F$.   (Lemmas~\ref{lem:K-stableclasses2orig}  and~\ref{lem:nottransportofstructure})

  \item \textbf{$G^\Fr$-conjugacy classes in $\CTk$}.  The set of $G^\Fr$-conjugacy classes in $\CTk$ is parameterized by $(\bar{W}_T)_{\simFr}$, the set of Frobenius-conjugacy classes in $\bar{W}_T = N_{G_{F,0}}(\bT)/T_0$.
  (Lemma~\ref{lem:K-stableclasses2})
 \end{itemize}

 \subsection{A comment on the relationship between the normalizer of \texorpdfstring{$\bT$ over $K$ and  the normalizer of $\bT$ over $\tame$}{T over K and normalizer of T over tame}}
 We now check that the parahoric subgroup $G_{F,0}$ interacts with the normalizer of $\bT$ as expected.

\begin{lemma}  \label{lem:somecontrol}
If $E$ is a Galois extension of $K$ over which $\bT$ splits and $y$ is the $\Gal(E/K)$-fixed point of $\AA'(\bT, E)$ in $\BB^{\red}(\bG,E)$, then
$$N_{G}(\bT) \cap G_{F,0} \subset \bG(E)_{y,0}.$$
If $\bG$ is  $K$-tame, then
$$N_{G}(\bT) \cap G_{F,0^+} = T  \cap G_{F,0^+} = T_{0^+}.$$
\end{lemma}

\begin{proof}
  Suppose $m \in N_{G}(\bT) \cap G_{F,0}$.
  Our first goal is to show that $m \in \bG(E)_{y,0}$.   Note that $x_T$ is the unique closest point in $\BB^{\red}(\bG,K)$ to $y$.   Similarly, $y$ is the unique point in $\AA'(\bT,E)$ closest to $x_T$.  Indeed, the unique  point, call it $z$,  in $\AA'(\bT,E)$ nearest $x_T$ has the property that  $\dist(x_T,z) = \dist(\gamma(x_T),\gamma(z)) = \dist(x_T,\gamma(z))$ for all $\gamma \in \Gal(E/K)$.   Thus, by uniqueness, we conclude that $z$ is $\Gal(E/K)$ fixed; hence it must be $y$.

  Since $y$ is the unique $\Gal(E/K)$-fixed point in $\AA'(\bT,E)$, $my \in \AA'(\bT,E)$, and $\gamma(my) = my$ for all $\gamma \in \Gal(E/K)$, we conclude that $my = y$.
  From~\cite[Lemma 2.12]{adler-debacker:murnaghan}, we have $ m \in G_{F,0}   \leq \bG(E)_{F,0}$, and so
  from~\cite[Lemma~4.2.1]{debacker:bruhat-tits} we conclude that $m \in \bG(E)_{y,0}$.

For the second claim, it is enough to show  $N_{G}(\bT) \cap G_{F,0^+} \subset T$.  Choose $m \in N_{G}(\bT) \cap G_{F,0^+}$
Note that $m \in \bG(E)_{y,0} \cap \bG(E)_{x_T,0^+}$, which means that the image of $m$ in the reductive quotient  $\bG(E)_{y,0} /\bG(E)_{y,0^+}$ must be unipotent.   Since $\bG$ is $K$-tame, we can assume $E$ is a tame extension of $K$ and so the image of $m$ in $\bG(E)_{y,0} /\bG(E)_{y,0^+}$ is semisimple.  Thus, the image of $m$ is the trivial element.
Consequently, $m$ fixes an open neighborhood $U$ in $\AA'(\bT,E)$ with $y \in \bar{U}$, hence $m$ must fix all of $\AA'(\bT,E)$ and so $m \in \bT(E)$.  Since $m \in N_{G}(\bT) \cap \bT(E)$, we conclude that $m \in T$.
\end{proof}

\begin{remark}
The proof of Lemma~\ref{lem:somecontrol}
shows that, in fact,  $N_{G}(\bT)  \leq N_{\bG(E)_{y,0}}(\bT) \cdot T$.
\end{remark}

\begin{remark}  Tameness may or may not be required for the second statement of Lemma~\ref{lem:somecontrol}.   However, it is required for the current  proof.   For example, for $\ff$ of characteristic two we have
$$\begin{bmatrix}
1 & 0 \\
1 & 1
\end{bmatrix}
\begin{bmatrix}
0 & 1 \\
1 & 0
\end{bmatrix}
\begin{bmatrix}
1 & 0 \\
1 & 1
\end{bmatrix}
=
\begin{bmatrix}
1 & 1 \\
0 & 1
\end{bmatrix}
$$
in $\SL_2(\ff)$.  That is, without some assumptions on the characteristic of our field, nontrivial elements of the normalizer of a torus can be unipotent!
\end{remark}

\begin{remark}  \label{rem:quotref} Thanks to Lemma~\ref{lem:somecontrol} when $\bT$ splits over a tame extension, we can interpret $N_{G_{F,0}}(\bT) / (T \cap G_{F,0^+})$ as $\Stab_{\bfG_{F}}(\bT)$ where $\bfG_{F}$ is the connected reductive $\ff$-group corresponding to the quotient $G_{F,0} / G_{F,0^+}$.  By $\Stab_{\bfG_{F}}(\bT)$ we mean the group of $\bar{x} \in \Stab_{\bfG_{F}}(\bT)$ for which there exists a lift $x \in G_{F,0}$  such that $\lsup{x}\bT = \bT$.
\end{remark}

\section{The Main Result and Examples}  \label{sec:tameexamples}

We begin by stating the parameterization theorem of this paper.   We then develop some machinery for calculating the objects that appear in this theorem, and we end  by looking at a number of examples.

\subsection{The main result}   Before stating the main result, we develop and recall some notation.

Let $(\Wwhistles)^{\Fr \circ N_q}$  denote the set of   $\Fr \circ N_q$-fixed points in the set of  $\sigma$-elliptic tame  $\sigma$-conjugacy classes in $\absW$.  For $c \in (\Wwhistles)^{\Fr \circ N_q}$, choose $w \in c$ and let $n \in \titsW$ be a lift of $w$. Thanks to Remark~\ref{rem:allindependent} it makes sense to denote the order of  $n \sigma \in \titsW \rtimes \EGal$ by $\ell_c$ and to
define  $K_c = {\tame}^{\sigma^{\ell_c}}$.
 Let $\pi$ denote a uniformizer in the tame extension of $K$ of degree $\ell_c$ such that $\pi^{\ell_c} = \varpi$.  Let $\lambda_c $ denote the unique element in $ \X_*(\bA \cap \bG')$  such that
\begin{itemize}
    \item $n$ is $\sigma$-conjugate to $\lambda_c(\sigma(\pi)/\pi)$ by an element of $\Egroup_{x_0,0}$ (see Lemma~\ref{lem:Asharpdone}) and
    \item   $x_c := x_0 + \lambda_c/\ell_c$ is an element of  $\bar{C}'$ (see Lemma~\ref{lem:kacinbarC}).
\end{itemize}

If $c \in (\Wwhistles)^{\Fr \circ N_q}$, then there exists a tame $K$-minisotropic maximal $k$-torus $\bT^c$ in $\bG$ for which the point associated to $T^c$ is $x_c$  (see Lemma~\ref{lem:elliptictoweyl} and Corollaries~\ref{cor:welldefined} and~\ref{cor:tameoverkexperiment}).     Set $T^c_{x_c} = T^c \cap G_{x_c,0}$.

\begin{theorem} \label{thm:main}  Recall that $k$ is a nonarchimedean local field.  For each $c \in (\Wwhistles)^{\Fr \circ N_q}$  fix a tame $K$-minisotropic maximal $k$-torus $\bT^c$ in $\bG$ for which the point associated to $T^c$ is $x_c$.
 \begin{enumerate}
     \item  \label{it:rationa} The set of $G^{\Fr}$-conjugacy classes of tame $K$-minisotropic maximal $k$-tori in $\bG$ is parameterized by  the set
     $$\{ (c, \tilde{c} ) \, | \, \text{ $c \in (\Wwhistles)^{\Fr \circ N_q}$ and $ \tilde{c} \in (\bar{W}_{T^c})_{\sim \Fr}  $}\}.$$
     Here $(\bar{W}_{T^c})_{\sim \Fr}$ denotes the set of $\Fr$-conjugacy classes in $\bar{W}_{T^c} := N_{G_{x_c,0}}(\bT^c)/T^c_0$.
     More precisely,  for each $\tilde{c} \in (\bar{W}_{T^c})_{\sim \Fr} $ we fix
     $k_{\tilde{c}} \in G_{x_c,0}$ such that $k_{\tilde{c}}\inv \Fr(k_{\tilde{c}}) \in N_{G_{x_c,0}}(\bT^c)$ is a lift of a member of  $\tilde{c}$.  Then each $G^{\Fr}$-conjugacy class in $\CTkc$ has exactly one representative in the set
$$  \{ \Ad({k_{\tilde{c}}}) \bT^c \, | \, \tilde{c} \in (\bar{W}_{T^c})_{\sim \Fr} \}.$$
 \item  \label{it:stable} The set of $k$-stable classes of tame $K$-minisotropic maximal $k$-tori in $\bG$ is parameterized by  the set
     $$\{ (c, \hat{c} ) \, | \, \text{ $c \in (\Wwhistles)^{\Fr \circ N_q}$ and $ \hat{c} \in ({W}_{T^c})_{\sim \Fr}  $}\}.$$
     Here $({W}_{T^c})_{\sim \Fr}$ denotes the set of $\Fr$-conjugacy classes in ${W}_{T^c} := N_{G_{x_c,0}}(\bT^c)/T^c_{x_c}$.
     More precisely,  for each $\hat{c} \in ({W}_{T^c})_{\sim \Fr} $ we fix
       $k_{\hat{c}} \in G_{x_c,0}$ such that $k_{\hat{c}}\inv \Fr(k_{\hat{c}}) \in N_{G_{x_c,0}}(\bT^c)$ is a lift of a member of  $\hat{c}$.  Then each $k$-stable conjugacy class in $\CTkc$ has exactly one representative in the set
$$  \{ \Ad({k_{\hat{c}}}) \bT^c \, | \, \hat{c} \in ({W}_{T^c})_{\sim \Fr} \}.$$
\item  \label{it:embedding} Suppose $c \in (\Wwhistles)^{\Fr \circ N_q}$ and $\hat{c} \in ({W}_{T^c})_{\sim \Fr}$.   Let $\bT =  \Ad({k_{\hat{c}}}) \bT^c$ where $k_{\hat{c}}$ is chosen as in (\ref{it:stable}).  Let $\bar{m}$ denote the image of $k_{\hat{c}}\inv \Fr(k_{\hat{c}})$ in  ${W}_{T^c}$.  The set of $k$-embeddings of $\bT$ into $\bG$ is parameterized by $(\bar{\bfT}^c_{x_c})_{\bar{m}\Fr}$, the group of $\bar{m} \Fr$-coinvariants in
$\bar{\bfT}^c_{x_c} = T^c_{x_c}/T^c_0$.   More precisely,  for each $\bar{c} \in (\bar{\bfT}^c_{x_c})_{\bar{m}\Fr}$ we fix $k_{\bar{c}} \in G_{x_c,0}$ such that $k_{\bar{c}}\inv (\bar{m} \Fr(k_{\bar{c}}) \bar{m}\inv) \in T^c_{x_c}$ is a lift of a member of $\bar{c}$.
 Then
 $$\{f_{\bar{c}} \, \colon \bT \rightarrow \bG \, | \, \text{ $f_{\bar{c}}(s) = \Ad({k_{\hat{c} }k_{\bar{c}} k_{\hat{c}}\inv}) s$ for $s \in \bT$  and $\bar{c} \in (\bar{\bfT}^c_{x_c})_{\bar{m}\Fr}$}\}$$
is a complete set of representatives for the $G^{\Fr}$-conjugacy classes of $k$-embeddings  of $\bT$ into $\bG$.
 \end{enumerate}
\end{theorem}

\begin{remark}    If $\hat{c} \in ({W}_{T^c})_{\sim \Fr}$, then the  fiber over $\hat{c}$ of  the natural map
$$(\bar{W}_{T^c})_{\sim \Fr} \rightarrow ({W}_{T^c})_{\sim \Fr}$$
parameterizes the set of $G^{\Fr}$-conjugacy classes of tame $K$-minisotropic maximal $k$-tori that appear in the $k$-stable class corresponding to $\hat{c}$.
\end{remark}

\begin{proof}  From Lemma~\ref{lem:imageofphi}  every $G$-conjugacy class of tame  maximal $K$-tori in $\bG$ corresponds to an element of  $\Wtamesimnon$.    Thus the assignment $c \mapsto \lsup{G}{\,}\bT^c$ from   $(\Wwhistles)^{\Fr \circ N_q}$ to the set of $G$-conjugacy classes of tame $K$-minisotropic maximal tori in $\bG$ has image consisting of exactly those   $G$-conjugacy classes of  tame $K$-minisotropic maximal tori  that contain a torus defined over $k$.

Consequently, to establish   statements~(\ref{it:rationa})  and~(\ref{it:stable}) of the theorem it is enough to show that for every $c \in (\Wwhistles)^{\Fr \circ N_q}$ the $G^{\Fr}$-conjugacy classes and $k$-stable classes in $\CTkc$ can be parameterized as claimed.  This follows from Lemmas~\ref{lem:K-stableclasses2} and~\ref{lem:K-stableclasses3}.

Statement~(\ref{it:embedding}) follows from  Corollary~\ref{cor:onembbeddings} and Lemma~\ref{lem:nottransportofstructure}.
\end{proof}

\subsection{A recasting of  part~(\ref{it:rationa}) of Theorem~\ref{thm:main} }

In Corollary~\ref{cor:recasting} below we recast part~(\ref{it:rationa}) of Theorem~\ref{thm:main} in a way that does not require the choice of a base point $\bT^c$.  While perhaps more elegant from the point of view of avoiding choices, this new description is considerably less concrete.

\subsubsection{A result about nearby tori}
We begin with a result about tori that are, in some sense, close to each other at depth zero.

\begin{lemma}  \label{lem:liftsthesameKc}  For this lemma we let $E$ be any extension of $k$ and let $\ff_E$ denote its residue field.  Suppose $x \in \BB(\bG,E)$ and $\bfT$ is a maximal $\ff_E$-torus in the connected reductive $\ff_E$-group  $\bfG^E_x$ corresponding to the quotient $\bG(E)_{x,0}/\bG(E)_{x,0^+}$.
    Suppose $\bT_1, \bT_2$ are $E$-split maximal tori such that $x \in \AA(\bT_i,E)$. 
    \begin{itemize}
        \item  If the image of $\bT_i(E) \cap \bG(E)_{x,0}$ in $\bfG^E_x$ is $\bfT$, then $\bT_1$ and $\bT_2$ are $\bG(E)_{x,0^+}$-conjugate.
        \item  If for all $n \in \N$ there exists $k \in \bG(E)_{x,n}$ such that $\lsup{k}\bT_1 = \bT_2$, then $\bT_1 = \bT_2$.
        \item  If $g \in N_{\bG(E)}(\bT_1) \cap \bG(E)_{x,0^+}$, then $g \in \bT_1(E)_{0^+}$.
    \end{itemize} 
\end{lemma}

\begin{proof}
    Since $x$ belongs to $\AA(\bT_1,E) \cap \AA(\bT_2,E)$, there exists an $h \in \bG(E)_{x,0}$ such that $\lsup{h}\bT_1 = \bT_2$. Let $\bar{h}$ denote the image of $h$ in
$\bfG^E_x$. By hypothesis, $\lsup{\bar{h}}\bfT = \bfT$; thus, $\bar{h} \in N_{\bfG^E_x}(\bfT)$. Consequently, by looking
at the affine Bruhat decomposition, we see that there exist an $n \in N_{\bG(E)}(\bT_1)  \cap \bG(E)_{x,0}$  and a $g \in \bG(E)_{x,0^+}$ such that $h = gn$. We have $\bT_2 = \lsup{h}\bT_1 = \lsup{g}\bT_1$.  This establishes the first claim.

For the second claim, suppose $\bT_1 \neq \bT_2$.  Then there exists $t \in \bT_1(E)$ such that $t \not \in \bT_2(E)$.
Since $\bT_2(E)$ is closed, there exists $\ell \in \N$ such that $t \cdot \bG(E)_{x,\ell} \cap  \bT_2(E) = \emptyset$.
By hypothesis, for all $n \in \N$ there exists $k \in \bG(E)_{x,n}$ such that $\lsup{k}t \in \bT_2(E)$.  This implies that $t \in \bT_2(E)$ modulo $\bG(E)_{x,m}$ for all $m \in \N$.  This contradicts the fact that  $t \cdot \bG(E)_{x,\ell} \cap  \bT_2(E) = \emptyset$.

For the final claim we note that since $g \in N_{\bG(E)}(\bT_1)$, we have that  $\AA(\bT_1,E)$ is stabilized by $g$.    If C is an alcove in $\AA(\bT_1, E)$ such that $F \subset \bar{C}$, then since $g \in \bG(E)_{x,0^+}$ we have that 
$g$ fixes $C$ pointwise.  Therefore, since $g \AA(\bT_1,E) = \AA(\bT_1,E)$ and $g$ fixes $C$ pointwise, we conclude that $g$ fixes $\AA(\bT_1, E)$ pointwise. Thus, we conclude
that $g \in \bT_1(E) \cap \bG(E)_{x,0^+} = \bT_1(E)_{0^+}$.
\end{proof}

\subsubsection{Recasting part~(\ref{it:rationa}) of Theorem~\ref{thm:main}}

The point $x_c$ of Theorem~\ref{thm:main}  determines a facet $F_c$ in  $\BB(\bG,K_c)$.   Let $\bfG_{c}$ denote the connected reductive $\ffc$-group attached to $F_c$.   We identify $\bfG_c$ with $\bG(K_c)_{F_c,0}/\bG(K_c)_{F_c,0^+}$, and we have  $\bfG_c^{\sigma} = \bfG_{x_c}$.  (Recall that $\bfG_c^{\sigma}$ denotes the connected component of $ \Fix_{\bfG_c}(\sigma)$.)   We let $\bshA$ denote the maximal torus in $\bfG_c$ corresponding to $\Ebtorus$.  That is,  the group of $\ffc$-points of  $\bshA$ coincides with  the image of $\Ebtorus(K_c) \cap \bG(K_c)_{F_c,0} $ in $\bfG_c$.

If $\bfT$ is a maximal $\ffc$-torus in $\bfG_c$ then there exists $g \in \bfG_c$ such that $\bfT = \lsup{g}\bshA$.  If $\bfT$ is $\sigma$-stable, then $g\inv \sigma(g) \in N_{\bfG_c}(\bshA)$ determines a $\sigma$-conjugacy class in $\absW$.   This gives a well-defined map from the set of $\sigma$-stable maximal tori in $\bfG_c$ to $\Wtamesim$.   Let $\dorbit_c$ denote the preimage of $c$ under this map, and let $\dorbit^{\Fr}_c$ denote the subset of $\dorbit_c$ consisting of tori that are $\Fr$-stable.   Note that  $\bfG_{x_c}$ ($= G_{x_c,0}/G_{x_c,0^+}$) acts on $\dorbit_c$  while   $\bfG_{x_c}^\Fr$  acts on $\dorbit_c$   and $\dorbit^{\Fr}_c$.

Recall from Definition~\ref{defn:6.3.12} that $\dorbit^k_{\bT^c,x_c}$ denotes the set of $\Fr$-stable tori in $\Ad(G_{x_c,0}) \bT^c$.  For $\bT \in \dorbit^k_{\bT^c,x_c}$ let $\varphi(\bT) \in \dorbit_c^{\Fr}$ be the torus  whose group of $\ffc$-points corresponds to the image of $\bG(K_c)_{x_c,0} \cap \bT(K_c)$ in $\bfG_c$.   In this way we  define a $G_{x_c,0}^\Fr$-equivariant map 
$\varphi \colon   \dorbit^k_{\bT^c,x_c} \rightarrow \dorbit_c^{\Fr}$.

\begin{lemma}  \label{lem:liftsthesame}
    If $\bfT \in \dorbit_c^{\Fr} $ and $\bT_1, \bT_2 \in \varphi\inv[\bfT]$, then $\bT_1$ and $\bT_2$ are $G^{\Fr}_{x_c,0^+}$-conjugate.
\end{lemma}

 With minor changes (e.g., changing $K$ to $K_c$ and $\Gamma$ to $\Gal(K_c/k)$), this proof is the same as the proof of~\cite[Lemma~2.2.2]{debacker:maximalunramified}.

\begin{proof}
From Lemma~\ref{lem:liftsthesameKc} there exists $g \in G(K_c)_{x_c,0^+}$ such that $\bT_2 = \lsup{g}\bT_1$.
For $\gamma \in \Gal(K_c/k)$, let $c_g(\gamma) := g\inv \gamma(g)  \in N_{\bG(K_c)}(\bT_1)$; $c_g$ is a one-cocycle. 
Since $F_c$ is $\Gal(K_c/k)$-stable and 
$g \in \bG(K_c)_{x_c,0^+}$, we have $c_g(\gamma) \in \bG(K_c)_{x_c,0^+}$. 
From Lemma~\ref{lem:liftsthesameKc} we conclude
that $c_g(\gamma) \in \bT_1(K_c)_{0^+}$.

Since $\cohom^1(\Gal(K_c/k), \bT_1(K_c)_{0^+})$
is trivial~\cite[Theorem 13.8.5 (1)]{kaletha-prasad:bruhat-tits}, there exists a $z \in \bT_1(K_c)_{0^+}$
 such
that $gz$ is fixed by $\Gal(K_c/k)$. We have $\lsup{gz}\bT_1 = \bT_2$ and, thanks to~\cite[Proposition~12.9.4]{kaletha-prasad:bruhat-tits}, $gz \in \bG(K_c)_{x_c,0^+}^{\Gal(K_c/k)} = G_{x_c,0^+}^{\Fr}$.
\end{proof}

\begin{lemma}   \label{lem:phiissurj}
    The $G_{x_c,0}^\Fr$-equivariant map 
$\varphi \colon   \dorbit^k_{\bT^c,x_c} \rightarrow \dorbit_c^{\Fr}$ is surjective.
\end{lemma}

\begin{proof}
Fix $\bfT \in \dorbit_c^{\Fr}$.   Since $\bfT \in \dorbit_c$, there exists $\ell \in \bG(K_c)_{x_c,0}$ such that $\bT  = \lsup{\ell}\bT_c$ is a lift of $\bfT$; that is, $\bfT$ corresponds to  the image of $\bT(K_c) \cap \bG(K_c)_{x_c,0}$ in $\bfG_c$.

Let $E \leq K_c$ be a  finite Galois extension of $k$ such that $\bT$ is $E$-split.   It will be  enough to show that there exists $\ell \in \bG(E)_{x_c,0^+}$ such that $\lsup{\ell}\bT$ is a $k$-torus.

Suppose that for all $\ell \in \bG(E)_{x_c,0^+}$ we have that  $\lsup{\ell}\bT$ is not a $k$-torus.
Then, thanks to Lemma~\ref{lem:liftsthesameKc}, for all  $\ell \in \bG(E)_{x_c,0^+}$ there exists $s_\ell \in \Q_{>0}$ such that
\begin{itemize}
    \item $\gamma(\lsup{\ell}\bT)$ is contained in the ${\bG(E)_{x_c,s_\ell}}$-orbit of $\lsup{\ell}\bT$ for all $\gamma \in \Gal(E/k)$ and
    \item there exists $\gamma \in \Gal(E/k)$ such that $\gamma(\lsup{\ell}\bT)$ is not contained in the ${\bG(E)_{x_c,s_\ell^+}}$-orbit of $\lsup{\ell}\bT$
\end{itemize}
 Since $\bG(E)_{x_c,0^+}$ is compact and the assignment $\ell \mapsto s_\ell$ is locally constant, there exists $h \in \bG(E)_{x_c,0^+}$ such that $s_h \geq s_\ell$ for all $\ell \in \bG(E)_{x_c,0^+}$. 
 Without loss of generality, we set $\bT = \lsup{h}\bT$.  We then have  $s_1 \geq s_\ell$ for all $\ell \in \bG(E)_{x_c,0^+}$.
 We will produce $j \in \bG(E)_{x_c,0^+}$ such that $s_j > s_1$, a contradiction.  

 If $\gamma,\rho \in \Gal(E/k)$, then $\gamma \rho (\bT) = \lsup{\ell_{\gamma \rho}}\bT = \gamma(\lsup{\ell_\rho}\bT) = \lsup{\gamma(\ell_\rho)\ell_\gamma} \bT$.  
Hence $f_{\bT}(\gamma,\rho) := \ell_\gamma \inv \cdot \gamma(\ell_\rho)\inv \cdot   \ell_{\gamma \rho}   $ is an element of $\bG(E)_{x_c,s_1} \cap N_{\bG}(\bT)$.    From Lemma~\ref{lem:liftsthesameKc} we conclude that $f_{\bT}(\gamma,\rho) \in \bT(E)_{s_1}$.  Since $\bG(E)_{x_c,s_1:s_1^+}$ is abelian, one checks that the induced function $\bar{f} \colon \Gal(E/k) \times \Gal(E/k) \rightarrow  \bT(E)_{s_1:s_1^+}$  is a  $2$-cocycle.   Since $\cohom^2(\Gal(E/k),\bT(E)_{s_1:s_1^+})$ is trivial~\cite[Theorem 13.8.5 (2)]{kaletha-prasad:bruhat-tits}, there exists  $\bar{\alpha} \colon \Gal(E/k)  \rightarrow  \bT(E)_{s_1:s_1^+}$, written $\gamma \mapsto \bar{\alpha}_\gamma$, such that $\bar{f}(\gamma,\sigma) = \gamma \bar{\alpha}_\sigma \cdot \bar{\alpha}_{\gamma \sigma}\inv \cdot \bar{\alpha}_\gamma $ for all $\gamma, \sigma \in \Gal(E/k)$.  If $\bar{\ell}_\gamma$ denotes the image of $\ell_\gamma$ in $\bT(E)_{s_1:s_1^+}$, then one checks that $\gamma \mapsto \bar{\ell}_\gamma \bar{\alpha}_\gamma$ defines a $1$-cocycle with values in $\bG(E)_{x_c,s_1:s_1^+}$.  Since  $\cohom^1(\Gal(E/k),\bG(E)_{s_1:s_1^+})$ is trivial~\cite[Theorem 13.8.5 (2)]{kaletha-prasad:bruhat-tits}, there exists $j \in \bG(E)_{x_c,s_1}$ such that 
the image of $\gamma(j)\inv j$ in  $\bG(E)_{x_c,s_1:s_1^+}$ is $\bar{\ell}_\gamma \bar{\alpha}_\gamma$ for all $\gamma \in \Gal(E/k)$.  If $\alpha_\gamma \in \bT(E)_{s_1}$ is any lift of $\bar{\alpha}_\gamma$ then there exists $x_\gamma \in \bG(E)_{x_c,s_1^+}$ such that $\ell_\gamma \alpha_\gamma = \gamma(j)\inv j x_\gamma$.  Thus, for all $\gamma \in \Gal(E/k)$ we have
$$\gamma(\lsup{j}\bT)  = \lsup{\gamma(j) \ell_\gamma} \bT = \lsup{j x_\gamma \alpha_\gamma   \inv } \bT = \lsup{jx_\gamma j\inv} ( \lsup{j} \bT).$$
 Since $jx_\gamma j\inv \in \bG(E)_{x_c,s_1^+}$, we conclude that $s_j$ is greater than $s_1$.
\end{proof}

\begin{cor}  \label{cor:recasting}
The  $\bfG_{x_c}^\Fr$-conjugacy classes in $\dorbit_c^{\Fr}$ are in bijective correspondence with the $G^{\Fr}$-conjugacy classes in $\dorbit^k_{\bT^c}$.
\end{cor}

\begin{proof}
From  Lemma~\ref{lem:phiissurj} and Lemma~\ref{lem:liftsthesame}
 the $G_{x_c,0}^\Fr$-equivariant map $\varphi \colon   \dorbit^k_{\bT^c,x_c} \rightarrow \dorbit_c^{\Fr}  $
descends to a bijective map from the set of 
$G_{x_c,0}^\Fr$-conjugacy classes in $\dorbit^k_{\bT^c,x_c}$
to the  set of 
$\bfG_{x_c}^\Fr$-conjugacy classes in 
$\dorbit_c^{\Fr}$.   The result now follows from 
Corollary~\ref{cor:rational_orbits_agree}.
\end{proof}

\subsection{A more concrete realization of \texorpdfstring{$W_T$, $T/T_{0^+}$, $T_F/T_{0^+}$, and $\bar{W}_T$}{bunch of group quotients} in the tame setting}

Fix $c \in \Wtamesim$, $w \in c$, $n \in \titsW$ with image $w$ in $\absW$.
Choose $h$, $\lambda = \lambda_n$, $g = g_n \in \bG(\nsplits)_{x_n,0}$ as in \S\ref{sec:markssec} and set $\bT = \lsup{g}\Ebtorus$.   Note that the image of $n_\sigma := g \inv \sigma(g)$ in $\absW$ is $w$.  Let $F \subset \BB(G)$ be the facet corresponding to $x_n \in \BB^{\red}(G)$.

Because it will help with the bookkeeping later, we will refer to $w$ as $w_\sigma$.

\begin{lemma}  \label{lem:idwithnormal}
We have
\begin{enumerate}
    \item  \label{lem:idwithnormalweyl} The group homomorphism that sends $m \in N_{G}(\bT)$ to $g\inv m g \in N_{\bG(\nsplits)}(\Ebtorus)$ induces an isomorphism
    $$\rho_g \colon N_G(\bT)/T \longrightarrow \absW^{w_\sigma \sigma}.$$
\item  \label{lem:idwithnormaltorus} The group homomorphism that sends $t \in T$ to $g\inv t g \in \Ebtorus(\nsplits)$ induces an isomorphism
$$\rho_g \colon T/{T_{0^+}} \longrightarrow (\Ebtorus(\nsplits)/\Ebtorus(\nsplits)_{0^+})^{w_\sigma \sigma}.$$

\item \label{lem:idwithparhoricnormaltorus}
The group homomorphism that sends $t \in T_F$ to $g\inv t g \in \Ebtorus(\nsplits)_0$ induces an injective map
$$\rho_g \colon T_F/{T_{0^+}} \longrightarrow
(\bshA)^{w_\sigma \sigma}.$$

\item  \label{lem:idwithenlargedWeyl} If $w_\sigma$ is $\sigma$-elliptic, then for $g$ as in Lemma~\ref{lem:thepoint2} the group homomorphism that sends $m \in N_{G_{F,0}}(T)$ to $g\inv m g \in  N_{\bG(\nsplits)}(\Ebtorus)$ induces an injective map
$$\rho_g \colon  N_{G_{F,0}}(\bT)/T_{0^+} \longrightarrow (N_{\Equotient_{x_0}} (\bshA))^{n \lambda(\xi) \sigma}.$$
\end{enumerate}
Moreover, if $\bG$ is simply connected, then $T_0 = T_{0^+}$, $\bar{W}_T = N_{G_{F,0}}(\bT)/T_{0^+}$, and the maps of both (\ref{lem:idwithparhoricnormaltorus}) and (\ref{lem:idwithenlargedWeyl}) are isomorphisms.
\end{lemma}

\begin{remark}  The maps in (\ref{lem:idwithparhoricnormaltorus}) and (\ref{lem:idwithenlargedWeyl}) both fail to be surjective when $\bG = \PGL_2$.   See also Lemma~\ref{lem:idwithnormal2Asharp}.
\end{remark}

\begin{proof}     Recall that ${n_{\sigma}} \in N_{\bG(\nsplits)}(\Ebtorus)$ has image $w_\sigma$ in $\absW$.

We begin by showing~(\ref{lem:idwithnormalweyl}).
Suppose $m \in N_{G}(\bT)$.
Since
$$ {n_{\sigma}} \sigma(g\inv  m g) {n_{\sigma}}\inv = ( g\inv \sigma(g))( \sigma(g \inv) m \sigma(g)) (\sigma(g\inv) g) =  g\inv  m  g,$$
the map sending $m \in N_{G}(\bT)$ to $g\inv m g \in N_{\bG(\nsplits)}(\Ebtorus)$ induces a homomorphism from $N_{G}(\bT)$ to $\absW^{w_\sigma \sigma}$.  Since the kernel of this homomorphism is $T$, the map descends to an injective map $\rho_g \colon W_\bT \rightarrow \absW^{w_\sigma \sigma}$.

We now show that $\rho_g$ is surjective.  Choose $w' \in \absW^{w\sigma}$.  Fix $n' \in N_{\bG(\nsplits)}(\Ebtorus)$ such that the image of $n'$ in $\absW$ is $w'$.  We have
$$\sigma(\lsup{g}n') = g (g\inv \sigma(g)) \sigma(n') (\sigma(g)\inv g) g\inv = g ( \lsup{n_\sigma} \sigma(n') ) g\inv = \lsup{g} (n'a) = \lsup{g}n' \lsup{g}a$$
for some $a \in \Ebtorus(\nsplits)$.
As $\lsup{g}a \in \bT(\nsplits)$ we conclude that the coset $(\lsup{g}n' )\bT(\nsplits)$ is $\sigma$-invariant.  Since $\coho^1(\sigma, \bT(\nsplits))$ is trivial, we conclude that every $\sigma$-invariant coset in $N_{\bG(\nsplits)} ( \bT) / \bT(\nsplits)$ has a representative in $N_G(\bT)/T$.  Hence, $\rho_g$ is surjective.

 We now show that~(\ref{lem:idwithnormaltorus}) holds.   Note that
\begin{equation*}
    \begin{split}
T_{0^+} &= T \cap G_{F,0^+} = \bT(\nsplits)^\sigma \cap ( \bG(\nsplits)_{F,0^+})^\sigma =( \bT(\nsplits) \cap \bG(\nsplits)_{F,0^+})^\sigma \\ &= (\lsup{g}(\Ebtorus(\nsplits)) \cap \bG(\nsplits)_{F,0^+})^\sigma = (\lsup{g}(\Ebtorus(\nsplits) \cap \bG(\nsplits)_{F,0^+}))^\sigma
= (\lsup{g} (\Ebtorus(\nsplits)_{0^+}))^\sigma.
\end{split}
\end{equation*}
Since $\coho^1(\sigma, \bT(\nsplits)_{0^+})$ is trivial, we  conclude that $T/T_{0^+}$ is isomorphic to
$$ (\lsup{g} (\Ebtorus(\nsplits) )/ \lsup{g} (\Ebtorus(\nsplits)_{0^+}))^\sigma =  (\lsup{{g}} (\Ebtorus(\nsplits)/ (\Ebtorus(\nsplits)_{0^+})))^\sigma.$$
Since $w_\sigma$ is the image of $g\inv \sigma(g)$ in $\absW$, we conclude that $(\lsup{{g}} (\Ebtorus(\nsplits)/ (\Ebtorus(\nsplits)_{0^+})))^\sigma$ is isomorphic, via their identification under the map $\Ad(g)$,  to   $(\Ebtorus(\nsplits)/ (\Ebtorus(\nsplits)_{0^+}))^{w_\sigma \sigma}$.

For (\ref{lem:idwithparhoricnormaltorus}), recall that $T_F = T \cap G_{F,0}$ and $g \in G_{F,0}$, so $g\inv T_F g \leq \Ebtorus(\nsplits)_0$.   Thus, from~(\ref{lem:idwithnormaltorus})  we have an injective map from $T_F/{T_{0^+}}$ to $(\Ebtorus(\nsplits)_0/\Ebtorus(\nsplits)_{0^+})^{w_\sigma \sigma}$.    Since $(\Ebtorus(\nsplits)_0/\Ebtorus(\nsplits)_{0^+})$ is isomorphic to $\bshA$, the result follows.

We now show that (\ref{lem:idwithenlargedWeyl}) is valid.   From Lemma~\ref{lem:thepoint2} we may assume that $g \inv \sigma(g) = z  \cdot \lsup{\lambda(\pi)}n $ for some $z \in \tilde{Z}_0$.
Suppose $m \in N_{G_{F,0}}(\bT)$.   Since $g \in \bG(\nsplits)_{F,0}$ we have $g\inv m g \in N_{\bG(\nsplits)_{F,0}}(\Ebtorus)$.  Since conjugation by $g$ carries topologically unipotent elements in $G_{F,0}$ to topologically unipotent elements in $\bG(\nsplits)_{F,0}$, we have a well-defined injective group homomorphism from $\bar{W}_T$ to $\Stab_{\bG(\nsplits)_{x_n,0}/ \bG(\nsplits)_{x_n,0^+}} (\Ebtorus)$ with image in $(\Stab_{\bG(\nsplits)_{x_n,0}/ \bG(\nsplits)_{x_n,0^+}} (\Ebtorus))^{{n_{\sigma}} \sigma}$.   Here the stabilizer in the quotient is to be interpreted as in Remark~\ref{rem:quotref}. 
Since $x_n = \lambda(\pi) \cdot x_0$, ${n_{\sigma}} = z \cdot \lsup{\lambda(\pi)}n$, and $\bG(\nsplits)_{x_0,0}/ \bG(\nsplits)_{x_0,0^+}$ is isomorphic to $\Equotient_{x_0}$ the result follows.

Finally, if $\bG$ is simply connected and $x \in \bG(\nsplits)_{x_0,0}$ has image in $(\bG(\nsplits)_{x_0,0}/\bG(\nsplits)_{x_0,0^+})^{n \lambda(\xi) \sigma}$, then  $g \lambda(\pi) x \lambda\inv(\pi) g\inv$ belongs to $\Stab_{\bG(K)}(F) \cdot \bG(\nsplits)_{F,0^+} = \bG(K)_{F,0} \cdot \bG(\nsplits)_{F,0^+}$.  It follows that the maps of both (\ref{lem:idwithparhoricnormaltorus}) and (\ref{lem:idwithenlargedWeyl}) are surjective, hence isomorphisms.
\end{proof}

\begin{remark}
If $y \in \bG(\nsplits)$ also has the property that $\lsup{y} \Ebtorus = \bT$, then there exists a $\tilde{n} \in N_{\bG(\nsplits)}(\Ebtorus)$ such that $g = y \tilde{n}$.  Thus, $\rho_y(N_G(\bT)/T) = \bar{w} \absW^{w_\sigma \sigma} \bar{w}\inv = \absW^{(\bar{w}w_\sigma \sigma(\bar{w})\inv )\sigma}$ where $\bar{w}$ denotes the image of $\tilde{n}$ in $\absW$.   Similarly,
$\rho_y(T/T_{0^+}) = (\Ebtorus(\nsplits)/ (\Ebtorus(\nsplits)_{0^+}))^{(\bar{w}w_\sigma \sigma(\bar{w})\inv )\sigma}$.
\end{remark}

\subsection{The action of \texorpdfstring{$\Fr$}{Fr}}
\label{subsec:Fraction}

Suppose $\lsup{G}\bT$ contains a $k$-torus.  It may happen that $\bT = \lsup{g}\bA$ is not defined over $k$.   However, if $\lsup{G}\bT$ contains a $\Fr$-stable element, then there exists $\ell \in G$ such that $\lsup{\ell g}\bA = \lsup{\ell}\bT$ is defined over $k$.   In fact, if $\bT$ is $K$-minisotropic, then from Lemma~\ref{lemma:KHR} and Lemma~\ref{lem:cardk} we can also assume $\ell \in G_{x_n,0}$.  Note that $(\ell g)\inv \sigma(\ell g) = g\inv \sigma(g) = n_\sigma$.

\begin{defn}   \label{defn:nfrob}   Suppose $\bT = \lsup{g}\bA$ is defined over $k$ and $g \in \bG(\nsplits)_{x_n,0}$.
Define $n_{\Fr} \in N_{\bG(K_n)_{x_n,0}}(\Ebtorus)$ by $n_{\Fr} := g\inv \Fr(g)$.
Let $w_{\Fr}$ denote the image of $n_{\Fr}$ in $\absW$.
\end{defn}

\begin{lemma}  \label{lem:nFraction}
We have that $n_\Fr \circ \Fr$ stabilizes $N_{\bG(K_n)_{x_n,0}}(\Ebtorus)^{n_\sigma  \sigma}$ while
$w_\Fr \circ \Fr$ stabilizes $\absW^{w_\sigma \sigma}$.
\end{lemma}

\begin{proof}
It will be enough to prove the first statement. Suppose $m \in N_{\bG(K_n)_{x_n,0}}(\Ebtorus)$ and $\Ad({n_\sigma}) \sigma(m) = m$.   We have
\begin{equation*}
\begin{split}
\Ad({n_\sigma})  \sigma [\Ad({n_\Fr}) \Fr(m)] &= \Ad [ (g\inv \sigma(g) )\sigma(g\inv \Fr(g)) ] \sigma (\Fr (m))\\
& =  \Ad [ (g\inv \sigma(\Fr(g)) ] \sigma (\Fr (m))\\
& =  \Ad ( n_\Fr ) \Fr [\Ad(g\inv \sigma^q(g))  \sigma^q  (m)]\\
    \end{split}
\end{equation*}
Since
$$\Ad(g\inv \sigma^q(g))  \sigma^q  (m) = \Ad(g\inv \sigma^{q-1}(g)) \sigma^{q-1}[ \Ad(g\inv \sigma(g)) \sigma(m)] =  \Ad(g\inv \sigma^{q-1}(g)) \sigma^{q-1}[ m],$$
after $q-1$ iterations we conclude that $\Ad(g\inv \sigma^q(g))  \sigma^q  (m) = m$.  Hence, we have shown
$$\Ad({n_\sigma})  \sigma [\Ad({n_\Fr}) \Fr(m)]  =  \Ad ( n_\Fr ) \Fr (m),$$
as required.
\end{proof}

\begin{lemma} \label{lem:FractionoverE}
  Via $\Ad(g\inv)$ (that is, $\rho_g$), the action of $\Fr$ on $W_T$ and $T/T_{0^+}$ corresponds to the action of $w_{\Fr} \circ \Fr$ on $\absW^{w \sigma}$ and $(\Ebtorus(\nsplits)/\Ebtorus(\nsplits)_{0^+})^{w \sigma}$.
\end{lemma}

\begin{proof}
We will verify this for $\absW^{w_\sigma \sigma}$; the proof for  $(\Ebtorus(\nsplits)/\Ebtorus(\nsplits)_{0^+})^{w_\sigma \sigma}$ is similar.    The map from $N_G(T)$ to $N_{\bG(\nsplits)}(\Ebtorus)/\Ebtorus(\nsplits)$ is given by $m \mapsto g\inv m g$.   Thus $\Fr(m)$ maps to
$$g\inv \Fr(m) g = g\inv \Fr(g g\inv) \Fr(m) \Fr(g g\inv) g = (g\inv \Fr(g)) \Fr(g\inv m g) (\Fr(g)\inv g).$$
The result follows from looking at the images of these elements in $\absW^{w_\sigma \sigma}$.
\end{proof}

\begin{lemma}   \label{lem:acomputationofFr}
We have
$$\text{$\Fr(N_q(n_\sigma)) = n\inv_\Fr n_\sigma \sigma(n_\Fr)$ \, \, and \, \,   $\Fr(N_q(w_\sigma)) = w\inv_\Fr w_\sigma \sigma(w_\Fr)$.}$$
\end{lemma}

\begin{proof}
For the first equality we have:
\begin{equation*}
\begin{split}
 \Fr(N_q(n_\sigma)) &= \Fr(n_\sigma \cdot  \sigma(n_\sigma) \cdots \sigma^{q-1}(n_\sigma)) = \Fr(g \inv \sigma(g)  \cdot  \sigma(g \inv \sigma(g) ) \cdots \sigma^{q-1}(g \inv \sigma(g)) )\\
 &= \Fr(g\inv \sigma^q(g))  = \Fr( g\inv (\Fr \inv \circ \sigma \circ \Fr) (g)) = \Fr(g\inv) \sigma(\Fr(g)) \\
 &= \Fr(g\inv) (g g\inv) \sigma( (g g\inv) \Fr(g)) = (\Fr(g\inv) g) ( g\inv \sigma( g)) \sigma (g\inv \Fr(g))\\
 &= n\inv_\Fr n_\sigma \sigma(n_\Fr).
\end{split}
\end{equation*}
The second equality follows from the first.
\end{proof}

\begin{corollary}  \label{cor:wFrdetermined}  Recall that $n_\Fr \in N_{\bG(K_n)_{F,0}}(\Ebtorus)$.
The equality
$$\Fr(N_q(n_\sigma)) = n\inv_\Fr n_\sigma \sigma(n_\Fr)$$
of Lemma~\ref{lem:acomputationofFr}
uniquely determines the element $n_\Fr$  up to left multiplication by an element of $(N_{\bG(K_n)_{F,0}}(\Ebtorus))^{n_\sigma \circ \sigma}$. Similarly, the equality
$$\Fr(N_q(w_\sigma)) = w\inv_\Fr w_\sigma \sigma(w_\Fr)$$
uniquely determines the element $w_\Fr$  up to left multiplication by an element of $\absW^{w_\sigma \sigma}$.
\end{corollary}

\begin{proof}  It will be enough to establish the first statement. Suppose $y \in N_{\bG(K_n)_{F,0}}(\Ebtorus)$ satisfies
$$\Fr(N_q(n_\sigma)) = y\inv n_\sigma \sigma(y).$$
We then have
$$n_\sigma \sigma (n_{\Fr}y\inv) n_\sigma\inv  = n_\Fr y\inv$$
and the result follows.
\end{proof}

\begin{lemma}  \label{lem:nottransportofstructureB}
Suppose $c$ is $\sigma$-elliptic.  Suppose $\bT' = \lsup{h}\bT \in\CTk$ with $h \in G$.  Set $m_h = h\inv \Fr(h) \in N_G(\bT)$ and let $w_h$ denote the image of $g\inv {m}_h g$ in $\absW$.   Via the map $\Ad(g\inv) \circ \Ad(h\inv)$ ( i.e., the composition $\rho_g \circ \varphi_h$ where $\varphi_h$ is defined in Lemma~\ref{lem:nottransportofstructure}) the action of $\Fr$ on $W_T$ and $T/T_{0^+}$ corresponds to the action of $(w_h w_{\Fr}) \circ \Fr$ on $\absW^{w_\sigma \sigma}$ and $(\Ebtorus(\nsplits)/\Ebtorus(\nsplits)_{0^+})^{w_\sigma \sigma}$.
\end{lemma}

\begin{proof}
This follows from Lemmas~\ref{lem:nottransportofstructure} and~\ref{lem:FractionoverE}.
\end{proof}

\begin{rem}
If $\bG$ is $K$-split, then from Lemmas~\ref{lem:FractionoverE} and~\ref{lem:acomputationofFr} we have  $\Fr(w_\sigma^q) = w\inv_\Fr w_\sigma w_\Fr$ and $(W_T)_\simFr$ is in bijection with  the set of $ w_\Fr \circ \Fr$-conjugacy classes in $\absW^{w_\sigma}$.  
\end{rem}

\begin{rem}  Note that from Corollary~\ref{cor:wFrdetermined} the element $w_\Fr$ will only be determined (in $\absW$) up to left multiplication by an element of $\absW^{w_\sigma}$, but from Lemma~\ref{lem:nottransportofstructureB} this is exactly what must happen as different choices of base point for $\lsup{G}\bT$ will result in different realizations of the various actions.
\end{rem}

\begin{lemma}  \label{lem:idwithnormal2Asharp}
Suppose $\bG$ is semisimple and $c$ is $\sigma$-elliptic.
The group homomorphism that sends $t \in T$ to $g\inv t g \in \Ebtorus(\nsplits)$ induces an isomorphism $\rho_g \colon T/{T_{0^+}} \rightarrow \bshA{}^{w_\sigma \sigma}$.
\end{lemma}

\begin{remark} \label{rem:simplyconnectedreduction}
If $\bG$ is simply connected, then since $\nsplits/K$ is tame we have $\bG(\nsplits)_{F,0}^\sigma = G_{F,0}$.     Thus, in this case,
$\bfT_F  = T/{T_{0^+}}$ and so $\bfT_F$ is isomorphic to $\bshA{}^{w_\sigma \sigma}$ in agreement with Lemma~\ref{lem:idwithnormal}.
\end{remark}

\begin{proof}
Since $c$ is $\sigma$-elliptic, we have that $T$ is a bounded subgroup of $G$.  It follows that $T$ is bounded in $\bT(\nsplits)$, and so $T \subset \bT(R_{\nsplits})$ where $R_{\nsplits}$ denotes the ring of integers of $\nsplits$.   Since $g \in \bG(\nsplits)_{F,0}$, we have
$$T = \bT(R_{\nsplits})^{\sigma} = (\bT(\nsplits) \cap \bG(\nsplits)_{F,0})^\sigma = (\lsup{g}(\Ebtorus(\nsplits)) \cap \bG(\nsplits)_{F,0})^\sigma = (\lsup{g} (\Ebtorus(R_{\nsplits})))^\sigma.$$
From the proof of Lemma~\ref{lem:idwithnormal} we have
$T_{0^+} =  (\lsup{g} (\Ebtorus(\nsplits)_{0^+}))^\sigma$.
Identify $\bshA$ with the image of $\Ebtorus(R_{\nsplits})/ \Ebtorus(\nsplits)_{0^+}$ in  $\bG(\nsplits)_{F,0}/\bG(\nsplits)_{F,0^+}$.    Since $\coho^1(\sigma, T_{0^+})$ is trivial, we  conclude that $T/T_{0^+}$ is isomorphic to
$$ (\lsup{g} (\Ebtorus(R_{\nsplits}) / \lsup{g} (\Ebtorus(\nsplits)_{0^+}))^\sigma =  (\lsup{\bar{g}} (\Ebtorus(R_{\nsplits})/ (\Ebtorus(\nsplits)_{0^+})))^\sigma =   (\lsup{\bar{g}} \bshA)^\sigma$$
where $\bar{g}$ denotes the image of $g$ in $ \bG(\nsplits)_{F,0}/\bG(\nsplits)_{F,0^+}$.

Since $w_\sigma$ denotes  the image of  $g \inv \sigma(g) \in N_{\bG(\nsplits)}(\Ebtorus)$ in $\absW$,  we have that $(\lsup{\bar{g}} \bshA)^\sigma$ is isomorphic to  $\bshA{}^{w_\sigma \sigma}$ via their identification under the map $\Ad(\bar{g})$.
\end{proof}

\begin{cor}  \label{cor:scellipticsummation}
 Suppose $\bG$ is simply connected and $w_\sigma$ is $\sigma$-elliptic.   The set of $k$-stable classes in $\CTk$ is indexed by $(\absW^{w_\sigma \sigma})_{\sim_{w_\Fr \circ \Fr}}$.  If $\bT'$ is a $k$-torus in the $w_\Fr \circ \Fr$-class corresponding to $w' \in \absW^{w_\sigma \sigma}$,  then, up to $G^{\Fr}$-conjugacy, the set of  $k$-embeddings of $\bT'$ into $\bG$ is indexed by $\bshA{}^{w_\sigma \sigma}_{w' w_\Fr \circ \Fr}$.
\end{cor}

\begin{proof}
    This  follows from Lemmas~\ref{lem:K-stableclasses2orig}, \ref{lem:K-stableclasses3}, \ref{lem:idwithnormal}, and \ref{lem:FractionoverE}.
\end{proof}

\subsection{The action of \texorpdfstring{$\Fr$ when $\bG$ is $k$-quasi-split and $n$ is $\sigma$-elliptic}{Fr when G is k-quasi-split and n is sigma elliptic}}
In this subsection we refine the results of subsection~\ref{subsec:Fraction} under the additional assumptions that $\bG$ is $k$-quasi-split and $n$ is $\sigma$-elliptic.

Since $\bG$ is $k$-quasi-split, we may assume that $x_0$ is Frobenius fixed, and so, since $x_n$ is $\Fr$-invariant, we have that $\Fr(\lambda) = \lambda$.   In this section we also assume that $\Fr$ is chosen such that $\Fr(\pi) = \pi$.

Since $n$ is $\sigma$-elliptic, from Lemma~\ref{lem:thepoint2} we may assume that $n_\sigma = z \cdot \lsup{\lambda(\pi)}n$.
 for some $z \in \tilde{Z}_0$.

\begin{defn}   \label{defn:nfrob2}
 Define $n_F \in N_{\bG(\nsplits)_{x_0,0}}(\Ebtorus)$ by $n_F = \lsup{\lambda\inv(\pi)}n_{\Fr}$.
\end{defn}

\begin{cor}   If $\bG$ is $k$-quasi-split and $n$ is $\sigma$-elliptic, then
we have that $n_F \circ \Fr$ stabilizes
$N_{\bG(\nsplits)_{x_0,0}}(\Ebtorus)^{n \lambda(\xi)\sigma}$.
\end{cor}

\begin{proof}
 Suppose $x \in N_{\bG(\nsplits)_{x_0,0}}(\Ebtorus)$ such that $\lsup{n \lambda(\xi)}\sigma(x) = x$.  Then
 $$\lsup{\lambda(\pi)} x = \lsup{\lambda(\pi) n \lambda(\xi)} \sigma(x) = \lsup{n_\sigma} \sigma (\lsup{\lambda(\pi)} x).$$
 Thus, $\lsup{\lambda(\pi)}x \in (N_{\bG(\nsplits)_{x_n,0}}(\Ebtorus))^{n_\sigma \sigma}$. \,
Since $\lsup{\lambda(\pi)} n_F \circ \Fr = n_{\Fr} \circ \Fr$, the result follows by unwinding definitions and using Lemma~\ref{lem:nFraction}.
\end{proof}

\begin{cor}   \label{cor:acomputationofFr2}
 If $\bG$ is $k$-quasi-split and $n$ is $\sigma$-elliptic, then we have
$$\text{ $\Fr[N_q(n \lambda(\xi) z)] = n_F\inv ( n \lambda(\xi)z) \sigma(n_F)$.}$$
\end{cor}

\begin{proof}
We have
\begin{equation*}
\begin{split}
    \Fr[N_q(n & \lambda(\xi)z )] =
   \Fr[N_q(\lsup{\lambda\inv(\pi)} n_\sigma  \lambda(\xi))]
   =  \Fr[N_q(\lambda\inv(\pi) n_\sigma  \sigma(\lambda(\pi)))]\\
   & = \Fr[( \lambda\inv(\pi) n_\sigma \sigma(\lambda(\pi)) ) \cdot \sigma(   \lambda\inv(\pi) n_\sigma \sigma(\lambda(\pi))) \cdot \sigma^2(  \lambda\inv(\pi) n_\sigma \sigma(\lambda(\pi))) \cdot \cdots \cdot \sigma^{q-1}(  \lambda\inv(\pi) n_\sigma \sigma(\lambda(\pi))]\\
     & =\lambda\inv(\pi)  \Fr[ N_q( n_\sigma)  \sigma^q(\lambda(\pi))] = \lambda\inv(\pi)  \Fr[ N_q( n_\sigma) ] \lambda(\pi)  \cdot  \Fr[\lambda\inv(\pi) \sigma^q(\lambda(\pi))] \\
     &= \lambda\inv(\pi)  \Fr[ N_q( n_\sigma) ] \lambda(\pi)  \cdot  \lambda(\Fr(\xi^q))
      \end{split}
\end{equation*}
     From Lemma~\ref{lem:acomputationofFr}, this becomes
 \begin{equation*}
\begin{split}
\lambda\inv(\pi) &n_{\Fr}\inv n_\sigma \sigma(n_{\Fr}) \lambda(\pi)  \cdot  \lambda(\Fr(\Fr\inv(\xi)))
          = (n_F\inv n z ) \cdot \lambda\inv(\pi) \sigma(n_{\Fr}) \lambda(\pi)  \cdot  \lambda(\xi)\\
         &  = (n_F\inv n z ) \cdot \lambda\inv(\pi) \sigma( \lambda (\pi)) \cdot \sigma( \lambda\inv(\pi) n_{\Fr} \lambda(\pi)) \cdot \sigma(\lambda\inv(\pi)) \lambda(\pi)  \cdot  \lambda(\xi)\\
          &  = (n_F\inv n z \lambda(\xi) ) \sigma(n_F)  \cdot \lambda\inv(\xi)   \cdot  \lambda(\xi)\\
    &=    n_F\inv ( n \lambda(\xi) z) \sigma(n_F).
   \end{split}
\end{equation*}
\end{proof}

\begin{corollary}  \label{cor:wFrdetermined2}   If $\bG$ is $k$-quasi-split and $n$ is $\sigma$-elliptic, then the equality
$$\Fr(N_q(n \lambda(\xi) z)) = n\inv_F n \lambda(\xi) z \sigma(n_F)$$
uniquely determines the element $n_F$  up to left multiplication by an element of $(N_{\bG(K_n)_{x_0,0}}(\Ebtorus))^{n \lambda(\xi) \circ \sigma}$.
\end{corollary}

\begin{proof}
    The proof is very similar to that of  Corollary~\ref{cor:wFrdetermined}.
\end{proof}

\begin{cor}  \label{cor:scellipticsummation2}
 Suppose $\bG$ is $k$-quasi-split and simply connected.  Suppose also that  $n$ is $\sigma$-elliptic.     The set of $\bG(k)$-classes in $\CTk$ is indexed by $((N_{\Equotient_{x_0}}(\bshA))^{n \lambda(\xi) \sigma})_{\sim_{n_F \circ \Fr}}$.
\end{cor}

\begin{proof}
    This  follows from Lemma~\ref{lem:K-stableclasses2},  Lemma~\ref{lem:idwithnormal}, and  tracing through how $\Fr$ acts on  $(N_{\Equotient_{x_0}}(\bshA))^{n \lambda(\xi) \sigma}$.
\end{proof}

\subsection{Example: the Coxeter conjugacy class for \texorpdfstring{$K$-split groups}{K-split groups}}   \label{ex:coxeter}

Suppose $\bG$ is $K$-split. Since $\bG$ is $K$-split, we have that $\sigma$ acts trivially on $\absW$ and  $\Pi = \Esimple$. Suppose $w_\sigma$ is a Coxeter element $\prod_{\alpha \in  \Esimple} w_\alpha$ in $\absW$; here $w_\alpha$ is the simple reflection in $\absW$ corresponding to $\alpha$.
We have that $\absW^{w_\sigma} = \langle w_\sigma \rangle$.   Let $\ell$ denote the order of $w_\sigma$ and suppose throughout this section that $(p,\ell) = 1$.

Since $(p,\ell) = 1$, we know that $w_\sigma^q$ is $\absW$-conjugate to $w_\sigma$.  Since $\Fr$  preserves $\Eroot$, it must carry $\Esimple$ to another basis for $\Eroot$.   Thus $\Fr(w_\sigma)$ is again a Coxeter element of $\absW$.
Since the set of Coxeter elements form a single $\absW$-conjugacy class, we conclude that the Coxeter conjugacy class in $\absW$ is stable under $\Fr \circ N_q$.  Thus, from Corollary~\ref{cor:tameoverkexperiment} we may choose a $\Fr$-stable maximal torus $\bT$ of $\bG$ in $\varphi_{\sigma}(\lsup{\absW}w_\sigma)$.

From Lemma~\ref{lem:K-stableclasses3} we have that the elements of the set of  $k$-stable conjugacy classes in $\CTk$ are indexed by $(W_{T})_\simFr$.   Since $y = w_\sigma^q$ also generates $\langle w_\sigma \rangle = \absW^{w_\sigma}$, from  Lemma~\ref{lem:FractionoverE} and the fact that $W_T$ is abelian we have that $(W_{T})_\simFr$ is isomorphic to
 $\langle w_\sigma \rangle / \langle y^{-1} (w_\Fr \Fr(y) w_{\Fr}\inv) \rangle$.  From Lemma~\ref{lem:acomputationofFr} we have that
 $y^{-1} (w_\Fr \Fr(y)w_{\Fr}\inv)= (w_\sigma\inv)^{q-1}$.  Thus, the elements of the set of $k$-stable conjugacy classes in  $\CTk$ are indexed  by $\langle w_\sigma \rangle / \langle w_\sigma^{q-1}  \rangle \cong \Z/(\ell,q-1)$.

 The group $T_F/T_0$ may be identified with $\bfZ_{x_0}$, the center of $\bfG_{x_0}$.  Thus,  from Lemma~\ref{lem:K-stableclasses2orig}, we conclude that, up to $G^{\Fr}$-conjugacy, the group $(\bfZ_{x_0})_{\Fr}$ parameterizes the $k$-embeddings of $\bT$ into $\bG$.

Since $\bfZ_{x_0}$ is central in the reductive quotient $\bfG_{x_0}$, we conclude that $w_\sigma$ acts trivially  on $\bfZ_{x_0}$.    Thus, if $\bT'$ is a $k$-torus in one of the stable classes in $\CTk$, then, up to $G^{\Fr}$-conjugacy, the set of $k$-embeddings of $\bT'$ into $\bG$ is indexed by $(\bfZ_{x_0})_{w_h w_\sigma \circ \Fr} \cong (\bfZ_{x_0})_{\Fr}$.
 If we also assume that $\bG$ is $k$-split, then  $(\bfZ_{x_0})_{\Fr} = \bfZ_{x_0} / \{ x^{q-1} \, | \, x \in \bfZ_{x_0} \} $.

\subsection{Example: \texorpdfstring{$\SL_n$ and unramified $\SU_n$}{SLn and unramified SUn}}  \label{ex:sltwo}
Suppose $\bG_+$ is the  $k$-group $\SL_n$ and $\bG_-$ is the $k$-group unramified $\SU_n$.   Suppose $p$ does not divide $n$, and  let $w_\sigma$ be a Coxeter element in $\absW$.  Note that the Coxeter class is the only $\sigma$-elliptic class in $\absW$.

From the discussion in \S\ref{ex:coxeter},  we can choose a $k$-torus $\bT$ in the $G_{\pm}$-orbit $\varphi_\sigma(\lsup{\absW}w_\sigma)$ and the elements of the set of $k$-stable classes in $\CTk$  are indexed by $\langle w_\sigma \rangle / \langle w_\sigma^{q-1}  \rangle$, which is  isomorphic to  $\Z/(n,q - 1)$.

Since $\bG_\pm$ is simply connected, from Remark~\ref{rem:simplyconnectedreduction} we have  $\bfT_F \simeq  \bshA{}^{w_\sigma \sigma}$, which, in this case, may be identified with $\mu_n$, the center of $\bG_{\pm}$.   The action of $w_\Fr$ on $\bshA{}^{w_\sigma \sigma}$ is trivial and the action of $\Fr$ on $\bshA{}^{w_\sigma \sigma}$ is given by $x \mapsto (x^{\pm 1})^q$.   Thus, $(\bshA{}^{w_\sigma \sigma})\sim_{w_\Fr \circ \Fr}$ is isomorphic to  $\Z/(n,q \mp 1)$.
We conclude that there are, up to $G_{\pm}^{\Fr}$-conjugacy, $(n,q \mp 1)$  $k$-embeddings of $\bT$ into $\bG$.

The set $\CTk$ breaks into $(n,q - 1)$ $k$-stable classes.  If  $\bT'$ is a $k$-torus in one of these classes,    then, up to $G_{\pm}^{\Fr}$-conjugacy, it can be embedded into $\bG_\pm$ in  $(n,q \mp 1)$ ways.

\subsection{Example: \texorpdfstring{$\Sp_4$}{Sp4}}  \label{ex:sp4}
Suppose $p > 2$.  We adopt the notation of Example~\ref{ex:tamesp4}.

As discussed in Example~\ref{ex:tamesp4} there are two $G$-conjugacy classes of  $K$-minisotropic maximal $K$-tori in $\Sp_4$, denoted $\dorbit_{C_2}$ and $\dorbit_{-1}$, corresponding to the two elliptic $\absW$-conjugacy classes  ${C_2}$ and ${-1}$.    Since $\Sp_4$ is $k$-split, from Corollary~\ref{cor:equivmap} each of  $\dorbit_{C_2}$ and $\dorbit_{-1}$ contains tori that are defined over $k$.

A Coxeter element of $\Sp_4$ has order $4$ and the center of $\Sp_4$ is isomorphic to $\mu_2$.  Therefore, $\dorbit^k_{C_2}$ decomposes into $(q-1,4)$ $k$-stable classes.  If $\bT$ is a torus in one of these classes, then, up to $\Sp_4(k)$-conjugacy, it $k$-embeds  into $\Sp_4$ in two ways.  Since $\bar{W}_T$ is isomorphic to $\mu_8$, the number of $\Sp_4(k)$-conjugacy classes in $\dorbit^k_{C_2}$ is $(q-1,8)$.

The set $\dorbit^k_{-1}$ breaks into five $k$-stable orbits, corresponding to the five conjugacy classes in $\absW^{-1} = \absW$.   Note that we can take $w_\Fr$ to be the trivial element in $\absW$.   A  conjugacy class $c$ in $\absW$ will be called $\alpha$-even if for some (hence any) element $w$ in $c$ we have  that $w_\alpha$ appears an even number of times in some (hence any) expression for $w$ in terms of the simple reflections $w_\alpha$ and $w_\beta$.  If $\bT'$ belongs to a $k$-stable class that is indexed by an $\alpha$-even conjugacy class (of which there are three), then, up to $\Sp_4(k)$-conjugacy, $\bT'$ has four $k$-embeddings into $\Sp_4$.  If $\bT'$ belongs to a $k$-stable class that is not indexed by an $\alpha$-even conjugacy class (of which there are two), then, up to $\Sp_4(k)$-conjugacy, $\bT'$ has two $k$-embeddings into $\Sp_4$.

For some choice of $\bT \in \dorbit^k_{-1}$, the group $\tilde{C}_2 := N_{G_{x_T,0}}(\bT)/T_0$ has $32$ elements and is isomorphic to the group, under matrix multiplication, of two-by-two matrices whose elements look like
$$\begin{bmatrix}
a & 0 \\
0 & b
\end{bmatrix}
\, \text{ or } \,
\begin{bmatrix}
0 & c \\
d & 0
\end{bmatrix}
$$
with $a,b,c,d \in \{1,i,-1,-i\}$ and $\Fr(i)=\pm i$ depending on whether or not $-1$ is a square in $\ff^\times$.  One calculates that
if $4$ does not divide $q-1$, then  $\dorbit^k_{-1}$ breaks into six $\Sp_4(k)$-conjugacy classes.
If $4$ does      divide $q-1$, then $\dorbit^k_{-1}$ breaks into fourteen $\Sp_4(k)$-conjugacy classes; more specifically, in this case two of the stable classes break into two  $\Sp_4(k)$-orbits each, two break into three $\Sp_4(k)$-orbits each, and one breaks into four $\Sp_4(k)$-orbits.

\subsection{Example: \texorpdfstring{$\PSp_4$}{PSp4}}   \label{ex:psp4}
Suppose $p > 2$.  We adopt the notation of Example~\ref{ex:sp4}.

There are two $G$-conjugacy classes of  $K$-minisotropic maximal $K$-tori in $\PSp_4$, denoted $\dorbit_{C_2}$ and $\dorbit_{-1}$, corresponding to the two elliptic $\absW$-conjugacy classes  ${C_2}$ and ${-1}$.    Since $\PSp_4$ is $k$-split, from Corollary~\ref{cor:equivmap} each of  $\dorbit_{C_2}$ and $\dorbit_{-1}$ contains tori that are defined over $k$.

A Coxeter element of $\PSp_4$ has order $4$ and the center of $\PSp_4$ is trivial.  Therefore, $\dorbit^k_{C_2}$ decomposes into $(q-1,4)$ $k$-stable classes.  If $\bT$ is a torus in one of these classes, then, up to $\PSp_4(k)$-conjugacy, it $k$-embeds  into $\PSp_4$ in one way; hence, the number of $\PSp_4(k)$-conjugacy classes in $\dorbit^k_{C_2}$ is also  $(q-1,4)$.

As in the case of  $\Sp_4$, the set $\dorbit^k_{-1}$ breaks into five $k$-stable orbits indexed by the $\absW$-conjugacy classes in $\absW$.       If $\bT'$ belongs to a $k$-stable class that is indexed by an $\alpha$-even conjugacy class, then, up to $\PSp_4(k)$-conjugacy, $\bT'$ has two $k$-embeddings into $\PSp_4$.  If $\bT'$ belongs to a $k$-stable class that is not indexed by an $\alpha$-even conjugacy class, then, up to $\PSp_4(k)$-conjugacy, $\bT'$ has one $k$-embedding into $\PSp_4$.

For some choice of $\bT \in \dorbit^k_{-1}$, the group $N_{G_{x_T,0}}(\bT)/T_0$ has $16$ elements and is isomorphic to the $\tilde{C}_2/\mu_2$.  One calculates that
if $4$ does not divide $q-1$, then  $\dorbit^k_{-1}$ breaks into six $\PSp_4(k)$-conjugacy classes.
If $4$ does  divide $q-1$, then $\dorbit^k_{-1}$ breaks into ten $\PSp_4(k)$-conjugacy classes; more specifically, in this case each stable classes breaks into two  $\PSp_4(k)$-orbits each.

\subsection{Example: \texorpdfstring{$\Gtwo$}{G2}}   \label{ex:gtwo}
Suppose $p > 3$.  We adopt the notation of Example~\ref{ex:tameG2}.  As discussed in Example~\ref{ex:tameG2} there are three $G$-conjugacy classes of  $K$-minisotropic maximal  $K$-tori in $\Gtwo$,  denoted $\dorbit_{G_2}$, $\dorbit_{A_2}$, and $\dorbit_{-1}$ corresponding to the three elliptic $\absW$-conjugacy classes $G_2$, $A_2$, and $-1$.
Since $\Gtwo$ is $k$-split, from Corollary~\ref{cor:equivmap} each of   $\dorbit_{G_2}$, $\dorbit_{A_2}$, and $\dorbit_{-1}$ contains tori that are defined over $k$.

A Coxeter element of $\Gtwo$ has order $6$ and the center of $\Gtwo$ is trivial.  Therefore, $\dorbit^k_{G_2}$ decomposes into $(q-1,6)$ $k$-stable classes.  Since $\Gtwo$ has trivial center, if $\bT'$ is a torus in one of these classes, then, up to $\Gtwo(k)$-conjugacy, it $k$-embeds  into $\Gtwo$ in one way.   This shows that the number of $\Gtwo(k)$-conjugacy classes in $\dorbit^k_{G_2}$  is $(q-1,6)$; this can also be verified by calculating $\bar{W}_T$.

Suppose $\bT \in \dorbit_{A_2}^k$.  The group $\bar{W}_T$ is isomorphic to $\mu_6$,  and $\bfT_F$ is isomorphic to $\mu_3$.  One calculates that $(\bar{W}_T)_{\Fr} =(\bar{W}_T)_{\simFr}$ is $\mu_6(\ff)$ -- that is, there are six $k$-stable classes in $\dorbit_{A_2}^k$ when the cubic roots of unity belong to $\ff$ and two otherwise.  For half of the $k$-stable classes in $\dorbit_{A_2}^k$ a torus $\bT'$ in the class will embed, up to $\Gtwo(k)$-conjugacy, into $\Gtwo$ in three ways.   For the other half of the classes a torus $\bT'$ in the class will embed, up to $\Gtwo(k)$-conjugacy, into $\Gtwo$ in one way.    Finally, there are twelve $G^\Fr$-conjugacy classes of tori in $\dorbit_{A_2}^k$ when $3 = \dabs{\mu_3(\ff)}$ and three otherwise.

Alternatively, if $w_\sigma = w_\alpha w_\beta w_\alpha w_\beta$, then $W_T \cong \absW^{w_\sigma \sigma} = \absW^{w_\sigma}= \langle w_\alpha w_\beta \rangle$.
Up to left multiplication by an element of $\absW^{w_\sigma}$, from Corollary~\ref{cor:wFrdetermined} we have $w_\Fr$ is trivial if $q$ is congruent to $1$ modulo $6$ and $w_\Fr$ is $w_\alpha$ if $q$ is congruent to $-1$ modulo $6$.  (Note that $\mu_3(\ff) = 1$ if and only if $q$ is congruent to $-1$ mod $6$.)  When $w_\Fr = 1$, there are six $k$-stable classes in $\dorbit_{A_2}^k$ and a torus $\bT'$ indexed by  $(w_\alpha w_\beta)^j \in \langle w_\alpha w_\beta \rangle $ will, up to $\Gtwo(k)$-conjugacy, $k$-embed into $\Gtwo$ in three ways when $j$ is even and in one way when $j$ is odd.  When $w_\Fr$ is $w_\alpha$, then there are two $K$-stable classes in $\dorbit_{A_2}^k$ and a torus $\bT'$ indexed by $\bar{w}  \in \langle w_\alpha w_\beta \rangle / \langle ( w_\alpha w_\beta)^2 \rangle  $ will, up to $\Gtwo(k)$-conjugacy, $k$-embed into $\Gtwo$ in three ways if $\bar{w}$ is not trivial and in one way if $\bar{w}$ is trivial.

If $\dabs{\mu_3(\ff)} = 3$, then $\dorbit_{A_2}^k$ breaks into twelve $\Gtwo(k)$-conjugacy classes.  If $\dabs{\mu_3(\ff)} = 1$, then $\dorbit_{A_2}^k$ breaks into three $\Gtwo(k)$-conjugacy classes.

For some choice of $\bT \in \dorbit^k_{-1}$, the group $\bar{W}_T$ is isomorphic to the Weyl group of $\Gtwo$ and $\bfT_F$ is isomorphic to  $\mu_2 \times \mu_2$.    Thus
$\dorbit^k_{-1}$ decomposes into six $k$-stable classes -- one for each $w_\Fr$-conjugacy class in $\absW^{-1} = \absW$.
 If $\bT'$ belongs to a $k$-stable class that is indexed by a  $w_\Fr$-conjugacy class with one element then, up to $\Gtwo(k)$-conjugacy, $\bT'$ has four $k$-embeddings into $\bG$.
 If $\bT'$ belongs to a $k$-stable class that is indexed by a  $w_\Fr$-conjugacy class with three elements then, up to $\Gtwo(k)$-conjugacy, $\bT'$ has two $k$-embeddings into $\bG$.
  If $\bT'$ belongs to a $k$-stable class that is indexed by a  $w_\Fr$-conjugacy class with two elements then, up to $\Gtwo(k)$-conjugacy, $\bT'$ has one $k$-embedding into $\bG$.

 The group $N_{G_{x_T,0}}(\bT)/T_0$ has $48$ elements and is  isomorphic to a Tit's group of $\Gtwo$ with generators $n_\alpha$ and $n_\beta$ of order four where $\Fr(n_\alpha) = n_\alpha$ and $\Fr(n_\beta) = n_\beta^q$.
 Independent of how $\Fr$ acts, $\dorbit_{-1}^l$ breaks into ten $\Gtwo(k)$-conjugacy classes. The $k$-stable classes indexed by a $w_\Fr$-class with an odd number of elements each break into two $\Gtwo(k)$-conjugacy classes.

\subsection{Example: ramified \texorpdfstring{$\SU_3$}{SU3}}   Suppose $p > 3$.
 We adopt the notation of Example~\ref{ex:ramsu3}.  As discussed in Example~\ref{ex:ramsu3} there are two $G$-conjugacy classes of  $K$-minisotropic maximal $K$-tori in ramified $\SU_3$,
 denoted $\dorbit_{c_{w_\alpha}}$ and $\dorbit_{c_{w_0}}$, corresponding to the twisted Coxeter element $w_\alpha \sigma$ and the element $-1 = w_0 \sigma$.   Since  $c_{w_\alpha} = (\Fr \circ N_q)(c_{w_\alpha})$ and $c_{w_0} = (\Fr \circ N_q)(c_{w_0})$, from Corollary~\ref{cor:tameoverkexperiment} we conclude that each of $\dorbit_{c_{w_\alpha}}$ and $\dorbit_{c_{w_0}}$ contains tori that are defined over $k$.

Suppose $\bT \in \dorbit^k_{c_{w_\alpha}}$.  Then ${W}_T$ is isomorphic to $\mu_3$.  Therefore, for the twisted Coxeter class, we have that $\dorbit^k_{c_{w_\alpha}}$ decomposes into $\dabs{(\mu_3)_{\Fr}} = \dabs{\mu_3(\ff)}$ $k$-stable classes. Since $\bshA{}^{w_\alpha  \sigma} \cong \bar{\sfT}_F$ is trivial, for $\bT'$ in a given $k$-stable class there is, up to $\SU_3(k)$-conjugacy, only one  $k$-embedding of $\bT'$ into $\SU_3$.

Suppose $\bT \in \dorbit^k_{c_{w_0}}$.  Then ${W}_T$ is isomorphic to $S_3$, the symmetric group on three letters.   Since $\absW^{-1} = \absW$, we can assume $w_\Fr = 1$.  Thus
$\dorbit^k_{c_{w_\alpha}}$ decomposes into three $k$-stable classes -- one for each $\Fr$-conjugacy class in $\absW^{-1} = \absW$.
 We have $\bshA{}^{-1}$ is $\mu_2 \times \mu_2$.
 If $\bT'$ belongs to a $k$-stable class that is indexed by a  $\Fr$-conjugacy class with one element then, up to $\SU_3(k)$-conjugacy, $\bT'$ has four $k$-embeddings into $\SU_3$.
 If $\bT'$ belongs to a $k$-stable class that is indexed by a $\Fr$-conjugacy class with three elements then, up to $\SU_3(k)$-conjugacy, $\bT'$ has two $k$-embeddings into $\SU_3$.
  If $\bT'$ belongs to a $k$-stable class that is indexed by a  $\Fr$-conjugacy class with two elements then, up to $\SU_3(k)$-conjugacy, $\bT'$ has one $k$-embedding into $\SU_3$.

\section{\texorpdfstring{$K$-minisotropic}{K-minisotropic} tori in isogenous groups} \label{sec:isogenous}

Motivated by the examples of $\Sp_4$ and $\PSp_4$ discussed in Examples~\ref{ex:sp4} and~\ref{ex:psp4} we show that information about $K$-minisotropic maximal tori can, in some cases, be derived from information about the analogous tori in isogenous groups.

\subsection{On the surjectivity of isogenies at the level of parahoric groups for tori}

Recall that $L$ is the completion of $K$ and $R_L$ denotes its ring of integers.   Let $\bT$ and $\bT'$ be tori over $K$, hence they are tori over $L$. We denote by $\CT$ and $\CT'$ the connected N\'{e}ron models of $\bT$ and $\bT'$, such that $\bar{T}_0 := \CT(R_L)\subset \bar{T} := \bT(L)$ is the parahoric subgroup of $\bar{T}$ and  $\bar{T}'_0 := \CT'(R_L)\subset \bar{T}' := \bT'(L)$ is the parahoric subgroup of $\bar{T}'$.

Suppose $\rho : \bT \ra \bT'$ is an isogeny.
Since $\CT$ and $\CT'$ are  (lft-)N\'{e}ron models, $\rho$ extends uniquely to a  morphism $\rho:\CT\ra\CT'$~\cite[Proposition~6]{bosch-lutkebohmert-raynaud:neron}.
From~\cite[Proposition~B.9.1]{kaletha-prasad:bruhat-tits}, we have
$\bar{T}_{0^+} := \CT(R_L)_{0^+}= \ker(\CT(R_L)\ra\CT(\ffc))$ and $\bar{T}'_{0^+} := \CT'(R_L)_{0^+}= \ker(\CT'(R_L)\ra\CT'(\ffc))$.

\begin{proposition} \label{prop:parahoric} If $\rho:\bT\ra \bT'$ is an isogeny whose order is prime-to-$p$, then $\rho[\bar{T}_0]=\bar{T}'_0$, $\rho[\bar{T}_{0^+}] = \bar{T}'_{0^+}$, and  $\rho$ is surjective on special fibers.
\end{proposition}

\begin{proof}[Proof (Cheng-Chiang Tsai).] We first claim that $C:=\coker(\bar{T}_0 \xra{\rho} \bar{T}_0')$ is finite and prime-to-$p$. Indeed, we have a Snake Lemma diagram
\[
\begin{tikzcd}
&&&\ker(\rho_*)_I \arrow{d}\\
1\arrow{r}& \bar{T}_0\arrow{r}\arrow{d}& \bar{T}\arrow{r}{\kappa_{\bar{T}}}\arrow{d}&\X_*(\bT)_{I}\arrow{d} \arrow{r}&1\\
1\arrow{r}&\bar{T}'_0\arrow{r}\arrow{d}&\bar{T}' \arrow{r}{\kappa_{\bar{T}'}} \arrow{d}&\X_*(\bT')_{I}\arrow{r}&1\\
&C\arrow{r}&H^1(L,\ker(\rho))
\end{tikzcd}
\]
where $\ker(\rho_*)$ and $H^1(L,\ker(\rho))$ are both  finite and prime-to-$p$, hence the claim.

Next, we claim that $\rho:\CT\ra\CT'$ is surjective on special fibers. Suppose  on the contrary that it is not.  Since $\ffc$ is infinite, we then have
 $\coker(\CT(\ffc)\ra\CT'(\ffc))$ is infinite. Since $\bar{T}_0'=\CT'(R_L)$ surjects onto $\CT'(\ffc)$, we have that $C$ is also infinite, a contradiction.

Looking again at a Snake Lemma diagram
\[
\begin{tikzcd}
&&&\ker(\CT(\ffc)\ra\CT'(\ffc)) \arrow{d}\\
1\arrow{r}&\ker(\CT(R_L)\ra\CT(\ffc))\arrow{r}\arrow{d}{\rho^{\flat}}&\CT(R_L)\arrow{r}\arrow{d}&\CT(\ffc)\arrow{r}\arrow{d}&1\\
1\arrow{r}&\ker(\CT'(R_L)\ra\CT'(\ffc))\arrow{d}\arrow{r}&\CT'(R_L)\arrow{d}\arrow{r}&\CT'(\ffc)\arrow{r}&1 \\
& \coker(f^{\flat}) \arrow{r} &C&
\end{tikzcd}
\]
we need to show that $\rho^{\flat}$ is surjective.  Since we already know that $\CT(\ffc)$ surjects onto $\CT'(\ffc)$ and that both $C$ and $\ker(\CT(\ffc)\ra\CT'(\ffc))$ are finite and prime-to-$p$, it suffices to show that $\coker(f^{\flat})$ is pro-$p$. This follows from the fact that $\ker(\CT'(R_L)\ra\CT'(\ffc))$ is pro-$p$~\cite[Proposition~A.4.23]{kaletha-prasad:bruhat-tits}.
\end{proof}

\begin{corollary}  \label{cor:goodatKpoints} Suppose $\bT$ and $\bT'$ are tori that are defined over $K$. If $\rho:\bT\ra \bT'$ is an isogeny whose order is prime-to-$p$,  then
$\rho[T_0]= T'_0$ and  $\rho[{T}_{0^+}] = {T}'_{0^+}$.
\end{corollary}

\begin{proof}[Proof (Cheng-Chiang Tsai).]
    Suppose $\gamma' \in T'_0$.   From Proposition~\ref{prop:parahoric} there exists $\gamma \in \bar{T}_0 = \CT(R_L)$ such that $\rho(\gamma) = \gamma'$.  The fiber of  $\rho$ over $\gamma'$ is a finite scheme over $K$, and so any element of the fiber must be defined over a finite extension, call it $K'$, \, of $K$.  Thus, $\gamma \in \bT(K') \cap \bT(L) = \bT(K)$.  We conclude that $\gamma \in T_0$.  A similar argument shows that  $\rho[{T}_{0^+}] = {T}'_{0^+}$.
\end{proof}

\subsection{On the surjectivity of isogenies at the level of parahoric subgroups}

Suppose $\bH$ and $\bL$ are reductive $K$-groups and $\rho \colon \bH \rightarrow \bL$ is a $K$-isogeny.
Note that $\rho$ carries $H$ into $L$ and,  for $x \in \BB(H) = \BB(L)$,
it carries $\Stab_H(x)$ into $\Stab_L(x)$.   Thus, from~\cite[Corollary 3.3.1]{kaletha:regular} the map $\rho$ carries $H_{x,0}$ into $L_{x,0}$.
We show that, under mild conditions, the latter map is surjective.

A version of the lemma below, but for $k$ rather than $K$, occurs in~\cite[Lemma~3.3.2]{kaletha:regular}.

\begin{lemma}   \label{lem:surjpro} If $p$ does not divide the order of $\ker(\rho)$, then the map $\res_{H_{x,0^+}} \rho \colon H_{x,0^+} \rightarrow L_{x,0^+}$ is surjective.  Similarly, the  map
$\res_{H_{x,0}} \rho  \colon  H_{x,0} \rightarrow L_{x,0}$
is surjective.
\end{lemma}

\begin{proof}
Let $\bA_H$ denote a maximal $K$-split torus in $\bH$ such that $x$ belongs to the apartment of $A_H$. Let $\bA_L$ denote the corresponding maximal $K$-split torus of $\bL$.
Let $\Ebtorus_H$ denote the centralizer of $\bA_H$ in $\bH$, and let $\Ebtorus_L$ denote the centralizer of $\bA_L$ in $\bL$.   Since $\bH$ and $\bL$ are $K$-quasi-split, both $\Ebtorus_H$ and $\Ebtorus_L$ are maximal $K$-tori.  Moreover, $\rho(\Ebtorus_H) = \Ebtorus_L$.

Since  the restriction of $\rho$ to the group generated by unipotent elements in  $H_{x,0^+}$ is an isomorphism onto the group generated by    unipotent elements in $L_{x,0^+}$, it is enough to show that  the restriction of ${\rho}$ to the pro-unipotent radical of $\EbtorusKrat_H$  surjects onto the pro-unipotent radical of $\EbtorusKrat_L$.   This follows from Corollary~\ref{cor:goodatKpoints}.

Similarly,  since the restriction of $\rho$ to the group generated by unipotent elements in  $H_{x,0}$ is an isomorphism onto the group generated by    unipotent elements in $L_{x,0}$, it is enough to show that  the restriction of ${\rho}$ to the parahoric subgroup of $\EbtorusKrat_H$  surjects onto the parahoric subgroup of $\EbtorusKrat_L$.   This follows from Corollary~\ref{cor:goodatKpoints}.
\end{proof}

\subsection{\texorpdfstring{$k$-embeddings  of maximal $K$-anisotropic $k$-tori in isogenous semisimple groups}{k-embeddings of maximal K-anisotropic k-tori in isogenous semisimple groups}}  \label{sec:indexingisog}

Suppose now that $\bH$ and $\bL$ are semisimple $k$-groups and $\rho \colon \bH \rightarrow \bL$ is a $k$-isogeny.

Let $\bT^H$ be a maximal $K$-minisotropic $k$-torus in $\bH$, and let $\bT^L$ be the corresponding maximal $K$-minisotropic $k$-torus in $\bL$.  Let $x_T$ denote the point in $\BB(L) = \BB(H)$ identified by $\bT$ and let $F$ denote the facet to which $x_T$ belongs.  From Remark~\ref{rem:neededlater}
it follows that $(\res_H\rho)\inv[T^L_F]$, the preimage of $T^L_F = T^L \cap L_{x_T,0}$ under $\res_H\rho$, is $T^H_F$ and $(\res_H\rho)\inv[N_{L_{x_T,0}}(T^L)] = N_{H_{x_T,0}}(T^H)$.

Set $Z_F^H = Z^H \cap H_{F,0}$ where $Z^H$ denotes the center of $H$.  Similarly, since $\ker(\rho) \leq \bZ^H$ it makes sense to define $\ker(\rho)_F = \ker(\res_H\rho) \cap H_{F,0}$.

\begin{lemma}  \label{lem:kernelpara}
If $T^L_0$ denotes the parahoric subgroup of $T^L$, then $(\res_H\rho)\inv[T^L_0] \leq T^H_F$ and
$$(\res_H\rho)\inv[T^L_0] = \ker(\rho)_F T^H_0.$$
\end{lemma}

\begin{proof}
Since $T^L_0 \leq T^L_F$ and $(\res_H\rho)\inv[T^L_F] = T^H_F$, we have $(\res_H\rho)\inv[T^L_0] \leq T^H_F$.

Since $L$ is semisimple, we have that $T^L_0 = T^L_{0^+}.$  Thus, if
$t \in (\res_H\rho)\inv[T^L_0]$, then $\bar{t}$, the image of $t$ in $L$, is topologically unipotent and so $t$ must be topologically unipotent mod $\ker(\res_H \rho)$.   Since $(\res_H\rho)\inv[T^L_0] \leq T^H_F$, we conclude that $t$ must be topologically unipotent mod $\ker(\rho)_F$.
\end{proof}

\begin{defn}
Define $W^H_{T} = N_H(\bT^H)/T^H = N_{H_{x_T,0}}(T^H)/T^H$, $\bar{W}^H_{T} = N_{H_{x_T,0}}(T^H)/T^H_{0}$, and $\bar{\sfT}^{H} = T^H_F/T^H_0$.   Define $W^L_{T}$, $\bar{W}^L_{T}$, and $\bar{\sfT}^{L}$ similarly.
\end{defn}

\begin{lemma}  \label{lem:mapssurj}
The map $\rho$ induces an isomorphism $W^H_{T} \cong W^L_{T}$.  Moreover,
if $p$ does not divide the order of $\ker(\rho)$, then $\rho$ induces surjections $N_{H_{x_T,0}}(T^H) \twoheadrightarrow N_{L_{x_T,0}}(T^L)$, $T^H_F \twoheadrightarrow T^L_F$,  $\bar{W}^H_{T} \twoheadrightarrow \bar{W}^L_{T}$, and $\bar{\sfT}^{H} \twoheadrightarrow \bar{\sfT}^{L}$.  The kernel of each of these surjections is $\ker(\rho)_F$ or $\ker(\rho)_F T^H_0 /T^H_0$ as appropriate.
\end{lemma}

\begin{proof}
Since $\ker(\res_H\rho) \leq T^H$, we have $W^H_{T} \cong W^L_{T}$.
The remainder of the statements follow from  Corollary~\ref{cor:goodatKpoints} and Lemmas~\ref{lem:surjpro} and~\ref{lem:kernelpara}.
\end{proof}

\begin{remark}
If $p$ does not divide the order of $\ker(\rho)$, then 
we may and do identify $\ker(\rho)_F$ with its image in $\bar{\sfT}^H$.
\end{remark}

\begin{corollary}
If $p$ does not divide the order of $\ker(\rho)$,  then
$$(T^L_F)_{\Fr} \cong (T^H_F)_{\Fr}/ (\ker(\rho)_F)_\Fr \text{\, \,  and \, \,}  (\bar{\sfT}^{L})_{\Fr}   \cong   (\bar{\sfT}^{H})_{\Fr} / (\ker(\rho)_F)_\Fr.$$
\end{corollary}

\begin{proof}   From Lemma~\ref{lem:mapssurj} we know that
$$1 \rightarrow \ker(\rho)_F \longrightarrow T^H_F \longrightarrow T^L_F \longrightarrow 1$$
is exact.
Since taking coinvariants is right exact, we have
$(T^L_F)_{\Fr} \cong (T^H_F)_{\Fr}/ (\ker(\rho)_F)_\Fr$.
Similarly,  from Lemma~\ref{lem:mapssurj} we know that
$$1 \rightarrow \ker(\rho)_F T^H_0/T^H_0 \longrightarrow \bar{\sfT}^{H}  \longrightarrow \bar{\sfT}^{L} \longrightarrow 1$$
is exact.   Since taking coinvariants is right exact, from Lemma~\ref{lem:kernelpara} we conclude
 $(\bar{\sfT}^{L})_{\Fr}   \cong   (\bar{\sfT}^{H})_{\Fr} / (\ker(\rho)_F)_\Fr$.
\end{proof}

\begin{corollary}
If $p$ does not divide the order of $\ker(\rho)$,  then
\[\pushQED{\qed}  N_{L_{x_T,0}}(T^L)_{\sim \Fr} \cong (N_{H_{x_T,0}}(T^H)/\ker(\rho)_F)_{\sim \Fr}  \text{\, \,  and \, \,}  (\bar{W}^L_{T})_{\sim \Fr}  \cong (\bar{W}^H_{T}/\ker(\rho)_F)/_{\sim \Fr}.\qedhere  \popQED
\]
\end{corollary}

\begin{remark}
If $z \in \ker(\rho)_F$ and $c \in (\bar{W}^H_{T})_{\sim \Fr}$, then, since $z$ is central in $H$, we have $cz \in (\bar{W}^H_{T})_{\sim \Fr}$.
For $c_1$ and $c_2$ in $(\bar{W}^H_{T})_{\sim \Fr}$ we write $c_1 \stackrel{\rho}{\sim} c_2$ provided that there exists $z \in \ker(\rho)_F$ such that $c_1 = c_2z$.  Then  $(\bar{W}^H_{T}/\ker(\rho)_F)/_{\sim \Fr}$ may be identified with the set of $\stackrel{\rho}{\sim}$-equivalence classes in $(\bar{W}^H_{T})_{\sim \Fr}$.   A similar interpretation holds for $(N_{H_{x_T,0}}(T^H)/\ker(\rho)_F)_{\sim \Fr}$.
\end{remark}

\begin{remark}
Suppose $\bG'$ is the derived group of $\bG$.  Recall that  $\bG_{\scon}$ denotes the simply connected cover of $\bG'$ while $\bG_{\ad}$ denotes the adjoint group of $\bG$.     Let $\bT'$ denote a $K$-minisotropic maximal $k$-torus in $\bG'$.

The results of this Section (that is, Section~\ref{sec:indexingisog}) show that if $p$ does not divide the order of  the kernel of the isogeny $\bG_{sc} \rightarrow \bG'$, then  the stable classes in the $\bG'$-orbit of $\bT'$,  the rational classes in the $\bG'$-orbit of $\bT'$, and the $k$-embeddings, up to rational conjugacy, of $\bT'$ into $\bG'$ can be calculated in a straightforward manner if we know how  to parameterize these things for the corresponding torus in $\bG_{\scon}$.
\end{remark}

\subsection{A generalization of Example~\ref{ex:coxeter}}
If $\bG$ is $K$-split and $\bT$ is a tame $K$-minisotropic maximal $k$-torus in $\bG$ corresponding to the Coxeter element of $\absW$, then from Example~\ref{ex:coxeter}  the set of $k$-embeddings of $\bT$ into $\bG$ form a single $G_{\ad}^\Fr$-orbit.  This phenomenon can be generalized as follows.

\begin{lemma} \label{lem:ratadjointtrans}

Suppose $\bT$ is a maximal $K$-minisotropic $k$-torus in $\bG$. Let $x_T$ denote the point in $\BB(G)$ identified by $\bT$ and let $F$ denote the facet to which $x_T$ belongs.
If the images of $T \cap G_{F,0}$ and $Z \cap G_{F,0}$ in $\sfG_F$ agree, then $G_{\ad}^{\Fr}$ acts transitively on the set of $k$-embeddings of $\bT$ into $\bG$.   In fact, up to $G^{\Fr}$-conjugacy, the parahoric subgroup $((G_{\ad})_{F,0})^{\Fr}$ acts transitively on the set of $k$-embeddings of $\bT$ into $\bG$.
\end{lemma}

\begin{proof}
Let $\bT_{\ad}$ denote the $k$-torus in $\bG_{\ad}$ corresponding to $\bT$.

The $k$-map $\bG \rightarrow \bG_{\ad}$ induces maps $G_{F,0} \rightarrow (G_{\ad})_{F,0}$ and $G_{F,0^+} \rightarrow (G_{\ad})_{F,0^+}$.   Since  the images of $T \cap G_{F,0}$ and $Z \cap G_{F,0}$ in $\sfG_F$ agree, we conclude that the image of $T_{\ad} \cap (G_{\ad})_{F,0}$ in $(\sfG_{\ad})_F$ is trivial.

Suppose $\gamma \in T^{\Fr}$ is strongly regular semisimple and $g \in G$ such that $\lsup{g}\gamma$ is Frobenius fixed.   Since $\lsup{g}\gamma$ is Frobenius fixed, we conclude that $g \cdot x_T$ is Frobenius fixed.  Hence, from Lemma~\ref{lem:FrpointsofGorbit} there exists $\ell \in G^{\Fr}$ such that $\ell g \cdot x_T = x_T$.   So, without loss of generality we may and do  assume $g \in \Stab_{G}(x_T)$.   It will be enough to show that there exists $h \in ((G_{\ad})_{F,0})^{\Fr}$ such that $\lsup{g}\gamma = \lsup{h} \gamma$.

Let $\bar{g}$ denote the image of $g$ in $G_{\ad}$.  Note that $\bar{g} \in \Stab_{G_{\ad}}(x_T)$. From Lemma~\ref{lemma:KHR}  there exist $t \in T_{\ad}$ and $k \in (G_{\ad})_{F,0}$ such that $ \bar{g} = kt$.   Thus, $\lsup{g}\gamma = \lsup{k}\gamma$.   But then $k\inv \Fr(k) \in T_{\ad} \cap (G_{\ad})_{F,0} \leq T_{\ad} \cap (G_{\ad})_{F,0^+} = (T_{\ad})_0$.  Since $\cohom^1(\Fr, (T_{\ad})_{0})$ is trivial, there exists $y \in (T_{\ad})_{0}$ such that $k\inv \Fr(k) = y\inv \Fr(y)$.   But then $h = ky\inv \in ((G_{\ad})_{F,0})^\Fr \leq G_{\ad}^\Fr$ and $\lsup{h}\gamma = \lsup{k} \gamma = \lsup{g}\gamma$.
\end{proof}

\section{\texorpdfstring{$K$-minisotropic}{K-minisotropic} Coxeter \texorpdfstring{tori}{tori}}
\label{sec:Kminisotropicoxeter}

The  existence of a maximal $K$-minisotropic $k$-torus in $\bG$ was established by Adler  in the Appendix to~\cite{debacker:unramified}.   We sharpen this result by showing that there exists a maximal $K$-minisotropic $k$-torus  such that $x_T$ is in the interior of $C'$.

\begin{lemma}  \label{lem:coxetertorus}  Recall that $Z_C = Z \cap G_{C,0}$.  There exists a maximal $K$-minisotropic $k$-torus in $\bG$ such that (a) $x_T \in C'$ and (b)  $T_C = Z_C T_{0^+}$. 
\end{lemma}

\begin{proof} 
  If $\psi$ is an affine root of $\bA$ with respect to $\bG$, $\bA$, $K$, and our valuation $\nu$, then we  let $U_\psi$ denote the  subgroup of $U_{\dot{\psi}}$ corresponding to $\psi$.  Recall that for $z \in \AA(\bA)$, we have $U_\psi \leq G_{z,0}$ if and only if $\psi(z) \geq 0$.

Write $\Phi$ as the finite disjoint union $\bigsqcup_{i=1}^m \Phi_i$ with each $\Phi_i$ irreducible.  This decomposition produces a decomposition $\Delta = \bigsqcup_{i=1}^m \Delta_i$ of the affine simple roots and a decomposition $\Psi = \bigsqcup_{i=1}^m \Psi_i$ of the affine roots.  Similarly, we also have a decomposition  $\BB^{\red}(G) = \BB(G_{\scon}) = \prod_{i=1}^m \BB(G^i_{\scon})$ where $G^i_{\scon}$ is a simply connected group with root system $\Phi_i$, and we write $C' = \prod_{i=1}^m C'_i$ for the corresponding decomposition of $C'$.    For each $1 \leq i \leq m$ there exist  unique  $r_i \in \R_{>0}$ and a unique $x_i \in C'_i$ such that $r_i = \psi(x_i)$ for all $\psi \in \Delta_i$.    Let $x = (x_1, x_2, \ldots , x_m) \in C'$ and $\vec{r} = (r_1, r_2 , \ldots , r_m) \in \R_{>0}^m$.   Let $G_{x,\vec{r}}$ be the group $\langle A_{0^+}, U_\psi \, | \, \text{$\psi \in \Psi_i$  implies $\psi (x_i) \geq r_i$} \rangle$ and let $G_{x,\vec{r}^+}$ be the group $\langle A_{0^+}, U_\psi \, | \, \text{$\psi \in \Psi_i$  implies $\psi (x_i) > r_i$}\rangle$.

Since $\Fr$ preserves $C$, it acts on $\Delta$.
For $\psi \in \Delta$,  choose $u_\psi \in U_{\psi} \setminus U_{\psi^+}$  such that $\Fr (u_\psi) = u_{\Fr (\psi)}$. 
Set $u = \prod u_\psi$.  

Suppose $y = (y_1, y_2, \ldots , y_m) \in \BB^{\red}(G)$. If for  each $1 \leq i \leq m$ we have $\psi(y_i) = r_i$ for all $\psi \in \Delta_i$, then $x = y$.  That is, $u \in G_{y,\vec{r}} \setminus G_{y,\vec{r}^+}$ if and only if $x = y$.  In fact, this is true for any element of $u G_{x,\vec{r}^+}$.

It follows that $C_G(\gamma) \leq \Stab_G(C)$ for all $\gamma \in u G_{x,\vec{r}^+}$.    Since $u\inv \Fr(u) \in G_{x,\vec{r}^+}$ and $\cohom^1(\Fr, G_{x,\vec{r}^+})$ is trivial, we conclude that  $u G_{x,\vec{r}^+} \cap G^{\Fr} \neq \emptyset$.  Thus, we may choose a strongly regular $\gamma \in uG_{x,\vec{r}^+} \cap G^{\Fr}$. Since $C_G(\gamma)$ is a bounded subgroup of $G$,  we conclude that $\bT := C_\bG(\gamma)$ is a  maximal $K$-anisotropic $k$-torus of $\bG$.  Note that $x_T = x$.

 Suppose $t \in T_F = G_{x_T,0} \cap T$, and let $\bar{t}$ denote the image of $t$ in  $\bfG_{x_T}$.  Since $x_T$ is in $C'$, the reductive quotient $\bfG_{x_T}$ is $\bfA$.  Since  $t$ commutes with $\gamma \in u G_{x_T,\vec{r}^+}$ and $\bar{t} \in \bfA$, we have  $\alpha(\bar{t}) = 1$ for all $\alpha \in \Delta$.  Thus, $T_C = Z_C T_{0^+}$.
\end{proof}

The point $x_T$ for $\bT = C_{\bG}(\gamma)$ with strongly regular $\gamma$ as in the proof of Lemma~\ref{lem:coxetertorus} is,
in the $K$-split tame setting, given by the Kac coordinates of the Coxeter conjugacy class in $\absW$.  For this reason, we make the following definition.

\begin{defn}   \label{defn:coxetertorus}
A torus that arises via the construction in the proof of Lemma~\ref{lem:coxetertorus} is called a \emph{$K$-minisotropic Coxeter $k$-torus}.   If we don't require that it be defined over $k$ (i.e., we don't require that $\Fr(u) = u$ or $\Fr(\gamma) = \gamma$), then we will call it a \emph{$K$-minisotropic Coxeter $K$-torus} or just a \emph{$K$-minisotropic Coxeter torus}.
\end{defn}

$K$-minisotropic Coxeter $k$-tori have many desirable properties.  We enumerate some of these properties here.

\begin{cor}
    Suppose $\bT$ is a $K$-minisotropic Coxeter $k$-torus with $x_T \in C'$.  \begin{itemize}
        \item  If a facet $H$ of $\BB(G)$ is contained in the closure of $C$, then the representatives for $\Stab_{G}(H)/G_{H,0}$ can be chosen in $T$.  If $H$ is also $\Fr$-stable, then the representatives for $\Stab_{G^{\Fr}}(H)/G^{\Fr}_{H,0}$ can be chosen in $T^{\Fr}$.  
        \item The group $G_{\ad}^{\Fr}$ acts transitively on the set of $k$-embeddings of $\bT$ into $\bG$.
        \item Every element of $\Stab_{G_{\ad}}(C)$ is cohomologous to an element of $T_{\ad}$.  Here $\bT_{\ad}$ is the $k$-torus in $\bG_{\ad}$ corresponding to $\bT$.
    \end{itemize} 
\end{cor}

\begin{proof}
The first point follows from Corollaries~\ref{cor:propertyofcoxeter1} and~\ref{cor:propertyofcoxeter}, the second follows from 
Lemma~\ref{lem:ratadjointtrans}, and  the final point follows from  Lemma~\ref{lemma:KHR} and Remark~\ref{rem:oncohomology} 
\end{proof}

Recall that $\eta \colon \bG_{\scon} \rightarrow \bG$ is the composition $\bG_{\scon} \rightarrow \bG' \rightarrow \bG$.

\begin{cor}
If $\bT$ is a $K$-minisotropic Coxeter $K$-torus and $x_T \in C'$, then $T_C = \eta[Z_{\scon}] T_0$. 
\end{cor}

\begin{proof}
We already know from Corollary~\ref{cor:2.3.4} that $T_C = \eta[T_{\scon}] T_0$.  Since $\bT_{\scon}$ is a $K$-minisotropic Coxeter $K$-torus in $\bG_{\scon}$, from Lemma~\ref{lem:coxetertorus} we have $T_{\scon} = (T_{\scon})_C = Z_{\scon} \cdot (T_{\scon})_{0^+}$ where   $Z_{\scon}$ denotes the center of $G_{\scon}$.   Since $\eta[(T_{\scon})_{0^+}] \subset T_0$, the claim follows.
\end{proof}

\begin{cor}    \label{cor:AcapT}
If $\bT$ is a $K$-minisotropic Coxeter $K$-torus in $\bG$, then $A \cap T = A \cap Z$.
\end{cor}

\begin{remark}  This result is not true for arbitrary  $K$-minisotropic tori.   For example, if $p > 2$ and $\bT$ is a $K$-minisotropic torus in $\Sp_4$ of type  $-1 \in A_1 \times A_1$ as in Example~\ref{ex:tamesp4}, then $A \cap T$ has order $4$, not $2$.  
\end{remark}

\begin{proof}  Without loss of generality, we may and do assume that $\bG$ is absolutely almost simple.

    Choose $g \in A \cap T$.  Let $\bH = C_{\bG}(g)^{\circ}$.   If $\bH = \bG$, then $g \in Z$.   
    
    Suppose $\bH \neq \bG$.   Note that $ \Ebtorus \leq \bH$ and $\bT \leq \bH$.  
    
    If $\bH$ is a Levi $K$-subgroup, then since $g \in A$ we must have that $\bH$ is the Levi of a parabolic $K$-subgroup.  However, since $\bT$ is elliptic in $\bG$, it cannot then belong to $\bH$.
    
    Thus, $\bH$ must be a generalized Levi $K$-subgroup of $\bG$ that is not the Levi of a parabolic $K$-subgroup.  Consequently, the ranks of the derived groups of $\bH$ and $\bG$ must be the same and so we can identify $\BB^{\red}(H)$ in $\BB^{\red}(G)$.  
    Since $\bT \leq \bH$, by uniqueness we have $x_T \in \BB^{\red}(H) \subset \BB^{\red}(G)$.  Choose $r > 0$ so that $G_{x_T,0^+} = G_{x_T,r} \neq G_{x_T,r^+}$.  Since $x_T \in \BB(H)$, there exists $h \in H$ such that $hx_T \in \AA'(A)$.   Let $C_1'$ be the $G$-alcove in $\AA'(A)$ that contains $hx_T$, let $x_1$ be an absolutely special vertex in $\bar{C}'_1$, and let $\bB_1$ be the Borel subgroup determined by $\bA$, $\bar{C}_1'$, and $x_1$.
    Since $\lsup{h}\bT$ is a $K$-minisotropic Coxeter $K$-torus in $\bG$, it contains an element $u$ that, modulo $G_{hx_T,r^+}$,  looks like $\prod u_\psi$ as $\psi$ runs over ${\Delta}' = \Delta(\bG, \bA, K, \nu, C_1')$ and  $u_\psi \in U_{\psi} \setminus U_{\psi^+}$.  
    Since  $H_{y,s} = H \cap G_{y,s}$ for all $y \in \AA'(A)$ and $s \in \R_{>0}$, we have that, modulo $H_{hx_T,r^+}$, the element $u \in H$ looks like $\prod u_\psi$  where $\psi$ runs over ${\Delta}'$ and  $u_\psi \in U_{\psi} \setminus U_{\psi^+}$.  We conclude that for each $\psi \in {\Delta}'$ we have $\dot{\psi} \in \Phi(\bH,\bA)$.   In particular, $\Delta(\bG,\bB',\bA)$ is a subset of $\Phi(\bH,\bA)$, hence $\Phi(\bH,\bA) = \Phi(\bG,\bA)$.  This contradicts our assumption that  $\bH \neq \bG$.
\end{proof}

\begin{remark}  
The proof of  Corollary~\ref{cor:AcapT} shows that if $\bH$ is a connected reductive $K$-subgroup of $\bG$ that contains both a $K$-minisotropic Coxeter torus and the centralizer of a maximal $K$-split torus of $\bG$, then $\bH$ must be $\bG$.  In particular, the intersection of  a $K$-minisotropic Coxeter torus and  \emph{any} maximal $K$-split torus of $\bG$ must be central.
\end{remark}

As Lemma~\ref{lem:lieintersect3} shows, the Lie algebra version of Corollary~\ref{cor:AcapT} holds for any $K$-minisotropic Cartan subalgebra.   Thus, in characteristic zero we can always conclude that if $\bG$ is semisimple and  $\bT$ is any $K$-minisotropic torus $\bG$, then $A \cap T$ is finite.

\begin{lemma} \label{lem:lieintersect3}
Suppose $k$ is any field of characteristic zero and $\mathfrak{g}$ is the Lie algebra of a reductive $k$-quasi-split group.   Suppose $\mathfrak{b}$ is a Borel $k$-subalgebra of $\mathfrak{g}$ and $\mathfrak{h}$ is a Cartan $k$-subalgebra  of $\mathfrak{b}$.  If $\mathfrak{t}$ is a $k$-elliptic Cartan subalgebra of $\mathfrak{g}$, then $\mathfrak{t} \cap \mathfrak{h}^k$ is $\mathfrak{z} \cap \mathfrak{h}^k$, where $\mathfrak{z}$ denotes the center of $\mathfrak{g}$.
\end{lemma}

Here $\mathfrak{h}^k$ denotes the maximal $k$-split toral subalgebra of $\mathfrak{h}$.

\begin{proof}  Suppose $Y \in \mathfrak{t} \cap \mathfrak{h}^k$.
Then $C_{\mathfrak{g}}(Y)$ is a Levi $k$-subalgebra of a parabolic $k$-subalgebra of $\mathfrak{g}$.  Since $\mathfrak{h} \leq C_{\mathfrak{g}}(Y)$ and $\mathfrak{h}$ is $k$-elliptic, we conclude that $C_{\mathfrak{g}}(Y) = \mathfrak{g}$.  That is, $Y \in \mathfrak{z}$.
\end{proof}

\begin{remark}  In characteristic zero, the proof of Lemma~\ref{lem:coxetertorus} can be modified to work on the  Lie algebra by choosing, for $\psi \in \Delta$, $X_\psi \in \gg_{\psi} \setminus \gg_{\psi^+}$  such that $\Fr (X_\psi) = X_{\Fr (\psi)}$.    Here the subgroups $\gg_{\psi}$ and $\gg_{\psi^+}$ of the root group $\gg_{\dot{\psi}}$ are the Lie algebra analogues of $U_{\psi}$ and $U_{\psi^+}$.   We then replace $\gamma$ with a regular semisimple $Y \in X+ \gg^{\Fr}_{x,r^+}$ where $X =  \sum X_\psi \in \gg^{\Fr}$.
\end{remark}

\begin{rem}
In the tame setting the $G^{\Fr}$-conjugacy class of $C_{\bG}(\gamma)$ is independent of the choice of strongly regular $\gamma \in uG^{\Fr}_{x,r^+}$.  Similarly, in the tame setting in characteristic zero the $G^{\Fr}$-conjugacy class of the Cartan subalgebra $C_{\bgg}(Y)$ is independent of the choice of regular semisimple $Y \in X + \gg^{\Fr}_{x,r^+}$.   This follows from~\cite[Lemma~2.3.2]{adler:refined}.
\end{rem}

\appendix

\section{Two questions about parahoric subgroups of \texorpdfstring{$K$-minisotropic tori}{K-minisotropic tori}}  \label{sec:appendixtwo}

\begin{center}
Mitya Boyarchenko, Stephen DeBacker, Anna Spice, Loren Spice, and  Cheng-Chiang Tsai
\end{center}

We continue to use the notation developed in the main body of this paper.  In particular, $\bG$ is a connected reductive $k$-group, $G$ is the group of $K$-points of $\bG$, $C$ is an alcove in the building of $G$, and $G_{C,0}$ is the parahoric subgroup of $G$ corresponding to $C$.   If $\bT$ is a $K$-torus, then  $\BB(T)$ is the alcove of $\BB(T)$, and we denote by $T_0$ the parahoric subgroup of $T$.

 We explore two natural questions that arise from results in this paper:
\begin{enumerate}
    \item \label{item:splitting} Does the sequence
\begin{equation} \label{equ:exactomega}
1 \longrightarrow G_{C,0} \longrightarrow \Stab_{G}(C) \longrightarrow \Stab_G(C)/G_{C,0} \longrightarrow 1
\end{equation}
discussed in Section~\ref{subsec:omega} always split?
\item If $\bT$ is a maximal  $K$-anisotropic torus in the derived group of $\bG$, is the cardinality of  $T/T_0$  independent of the isogeny class of the derived group?
\end{enumerate}

\subsection{On the splitting of~\ref{equ:exactomega}}

It is known that the sequence~\ref{equ:exactomega} splits whenever $\bG$ is $K$-split and either the derived group of $\bG$ is simply connected or $\bG$ is almost simple~\cite{adrian:remark}.
In this appendix we  show that~\ref{equ:exactomega} also
splits when both $\bG$ is adjoint and a $K$-minisotropic Coxeter $k$-torus (see Definition~\ref{defn:coxetertorus}) in $\bG$ splits over a tame extension of $K$.  We then show that~\ref{equ:exactomega}  does not, in general, split.

Suppose $\bT$ is a $K$-anisotropic maximal torus.  Note that  $T$ is bounded. 
Write $T = T_{p'} \times T_p$ where $T_{p'}$ is the subgroup of elements of order coprime to $p$ and $T_p$ is the pro-$p$-subgroup of $T$. If $T_p = T_0 = T_{0^+}$, which always happens when $\bT$ splits over a tame extension, then the  sequence
\begin{equation*} 
1 \longrightarrow T_0 \longrightarrow T \longrightarrow T/T_0 \longrightarrow 1
\end{equation*}
splits.

\begin{lemma}
    If $\bG$ is an adjoint group and $\bT$ is a $K$-minisotropic Coxeter $K$-torus in $\bG$ that splits over a tame extension, then sequence~\ref{equ:exactomega} splits.
\end{lemma}

\begin{proof}
 This follows from the discussion prior to the statement of the lemma together with  Corollary~\ref{cor:propertyofcoxeter} and Lemmas~\ref{lem:mapfromtorus} and~\ref{lem:coxetertorus}.
\end{proof}

To see that the sequence in~\ref{equ:exactomega} can fail to split in the absence of tameness, even in characteristic zero, consider the following example.  Let $k$ be $\Q_2$, let $i \in \bar{k}$ be a square root of $-1$, and let $E$ be the quadratic ramified extension $K[i]$.  Note that $1+i$ is a uniformizer in $E$.  Let $\bT$ be a $K$-torus such that $\bT(E) = E^\times$ and $T = N_{E/K}^1$, the kernel of the norm map from $E$ to $K$.  From~\cite[Section 7.3]{kottwitz:rational} we have the following commutative diagram.
\begin{center}
\begin{tikzcd}
E^\times = \bT(E) \arrow[r, "\kappa_{\bT(E)}"] \arrow[d, "N"]
&\X_*(\bT) \arrow[d,  "\alpha" ] \\
N_{E/K}^1 = T  \arrow[r, "\kappa_T"]
& X_{*}(\bT)_{\Gal(E/K)}
\end{tikzcd}
\end{center}
Here each of the horizontal maps is a surjective Kottwitz map, $N$ is  defined by $N(x) = x \tau({x})\inv$ where $\tau$ is the nontrivial element of $\Gal(E/K)$, and $\alpha$ is the obvious surjective map.  Since $N(1+i) = i$ and $\alpha (\kappa_{\bT(E)} (1 +  i)) \neq 0$ in $\X_*(\bT)_{\Gal(E/K)} \cong \Z/2 \Z$, we have $\kappa_T(i) \neq 0$.  Note that $\kappa_T(-1) = \kappa_T(i^2) = 2 \kappa_T(i) = 0$.  Since the kernel of $\kappa_T$ is $T_0$,  we conclude that $-1 \in T_0$ and $\bar{i}$, the image of $i$ in $T/T_0$, generates the group $T/T_0 \cong \Z/2 \Z$.  Since the only order two element of $T$ is $-1$, we conclude that the sequence in~\ref{equ:exactomega} does not split.

If we replace $\Q_2$ in the above discussion with $\mathbb{F}_2((t))$ where $\mathbb{F}_2$ is the field with two elements and let $E$ be a quadratic extension of  $\mathbb{F}_2((t))$, then $T/T_0$ has order two, but $T$ has no order two elements and so the sequence in~\ref{equ:exactomega} does not split.

\subsection{Isogeny and the cardinality of \texorpdfstring{$T/T_0$ for $K$-anisotropic $\bT$}{Tquotient for K-anisotropic T}}

If $\bT$ is a $k$-torus in $\SL_2$ corresponding to the norm one elements of a tamely ramified extension and $\bT'$ is the corresponding $k$-torus in $\PGL_2$, then $T/T_0$ and $T'/T_0'$ have the same cardinality.  This is not a general phenomenon, but it is not rare.

From the results of Haines and Rapoport  and of Kottwitz discussed in Section~\ref{sec:someconsequences}  we have $T/T_0 \cong \X_*(\bT)_I$ where $I$ is the inertial group $\Gal(\bar{k}/K)$.  We will study $\X_*(\bT)_I$.

\begin{lemma}
\label{lem:cyclic-norm}
Let \(\Lambda\) be an additive cyclic group of order \(n\), and
\(\tau\) an automorphism of \(\Lambda\).
Let \(m\) be a multiple of \(n\) such that
\(\tau^m\) is trivial,
and put
\(N_{\tau, m}(\lambda) \mathrel{:=} \sum_{i = 0}^{m - 1} \tau^i(\lambda)\)
for all \(\lambda \in \Lambda\).
Then \(N_{\tau, m}\) is identically \(0\) on \(\Lambda\).
\end{lemma}

\begin{proof}
By the Chinese Remainder Theorem, we may, and do, assume that \(n\) is a power of some prime \(\ell\), so that \(\Lambda\) is a cyclic module for the \(\ell\)-adic integers \(\Z_\ell\).  Since \(N_{\tau, m}\) factors through \(N_{\tau^{m/n}, n}\), we may, and do, assume that \(m\) is also a power of \(\ell\), as well as a multiple of \(n\) (or even that \(m\) equals \(n\), though we do not need this).

Choose an element \(k \in \Z_\ell\) such that \(\tau\) sends a generator \(\lambda\) of \(\Lambda\) to \(k\lambda\), hence acts as multiplication by \(k\) on all of \(\Lambda\).  Then \(N_{\tau, m}\) acts on \(\Lambda\) as multiplication by the \(\ell\)-adic integer \(\sum_{i = 0}^{m - 1} k^i\).

Since \(\tau^m\), which acts on \(\Lambda\) as multiplication by \(k^m\), is trivial, and the annihilator of \(\Lambda\) in \(\Z_\ell\) is \(n\Z_\ell\), we have that \(k^m\) belongs to \(1 + n\Z_\ell\).  Since \(m\) is a power of \(\ell\) and \(\Z_\ell/(1 + \ell\Z_\ell)\) has order relatively prime to \(\ell\), we conclude that \(k\) belongs to \(1 + \ell \Z_\ell\), and so that the \(\ell\)-adic valuation of \(k^m - 1\) is the sum of the \(\ell\)-adic valuations of \(m\) and \(k - 1\).
  That is, \(k^m - 1\) belongs to \(m(k - 1)\Z_\ell\), so \(\sum_{i = 0}^{m - 1} k^i = \frac{k^m - 1}{k - 1}\) belongs to \(m\Z_\ell \subseteq n\Z_\ell\), and hence multiplication  by \(\sum_{i = 0}^{m - 1} k^i \)  annihilates \(\Lambda\); that is, \(N_{\tau, m}\) is identically \(0\)  on \(\Lambda\).
\end{proof}

We now present a technical result, Proposition \ref{prop:elliptic-order}, on Weyl groups, whose proof seems to require case-by-case analysis.

The quantification in Proposition \ref{prop:elliptic-order} seems elaborate, but note that almost every irreducible root system has a cyclic fundamental group, so the result usually says that the order of \(w\) is divisible by the order of the fundamental group, and in the remaining case (of \(D_n\) with \(n\) even) just says that the order of \(w\) is divisible by \(2\).

\begin{proposition}\label{prop:elliptic-order}
 If \(\Phi\) is an irreducible root system and \(w\) is an elliptic element of \(W(\Phi)\), then the order of \(w\) is divisible by the order of every element of every subquotient of the fundamental group of \(\Phi\).
\end{proposition}

\begin{proof}
Put \(W = W(\Phi)\).  We indicate below the possibilities, according to the Borel--de Siebenthal algorithm, for a maximal proper, but non-parabolic, Weyl subgroup \(W'\) of \(W\).

For \(W\) of type \(A_n\), the maximal proper Weyl subgroups of \(W\) are all parabolic.  Since the Weyl groups of types \(B_n\) and \(C_n\) are the same, we do not distinguish them.

We put \(B_1 = C_1 = A_1\) and \(D_2 = A_1 + A_1\), but do not consider \(D_1\), so that any occurrence of a symbol \(D_m\) carries the implicit constraint \(m > 1\).

\centerline{
\begin{tabular}{c|c}
Type of \(W\) & Type of \(W'\) \\ \hline
\(B_n\) or \(C_n\) & \(C_r + C_{n - r}\), \(D_n\) \\
\(D_n\) & \(D_r + D_{n - r}\) \\
\(E_6\) & \(A_1 + A_5\), \(3A_2\) \\
\(E_7\) & \(A_1 + D_6\), \(A_2 + A_5\), \(A_1 + 2A_3\), \(A_7\) \\
\(E_8\) & \(A_1 + A_2 + A_5\), \(A_1 + A_7\), \(A_1 + E_7\), \(A_2 + E_6\), \(A_3 + D_5\), \(2A_4\), \(A_8\), \(D_8\) \\
\(F_4\) & \(A_1 + A_3\), \(A_1 + C_3\), \(2A_2\), \(B_4\) \\
\(G_2\) & \(A_1 + A_1\), \(A_2\).
\end{tabular}
}

Recall that \(W\) is the Weyl group of \(\Phi\).
For each entry in the chart above,
if \(W'\) is the Weyl group of \(\Phi'\) and \(\Phi''\) is an irreducible component of \(\Phi'\), then
in each case except \(D_n\) with \(n\) even, the fundamental group of \(\Phi\) is (cyclic and) isomorphic to a subgroup of the fundamental group of \(\Phi''\).  Even in this exceptional situation, we have that every cyclic subgroup of the fundamental group of \(\Phi\) is isomorphic to a subgroup of the fundamental group of \(\Phi''\).

Therefore, we may, and do, assume that \(w\) belongs to no proper Weyl subgroup of \(W\).  Then, using \cite[p.~8, Lemma 10]{carter:weyl} to enumerate the possibilities and \cite[p.~23, Table 3]{carter:weyl} to compute their orders, we find:

\centerline{
\begin{tabular}{c|c}
Type of \(W\) & order of \(w\) is divisible by \ldots \\ \hline
\(A_n\) & \(n + 1\) \\ 
\(B_n\) or \(C_n\) & \(2n\) \\ 
\(D_n\) & \(2\gcd(n - 1, 2)\) \\ 
\(E_6\) & \(3\) \\ 
\(E_7\) & \(2\) \\ 
\(E_8\) & \(1\) \\ 
\(F_4\) & \(12\) \\
\(G_2\) & \(6\).
\end{tabular}
}
The result follows.
\end{proof}

 \begin{lemma}\label{lem:Kottwitz-card}
 Suppose $k$ has characteristic zero and
$\bG$ is split and simple over \(K\).  Suppose $\rho \colon \bG \ra \bH$ is an isogeny. Let $\bT$ be a $K$-anisotropic maximal torus in $\bG$ and let $\bT'$ denote the corresponding torus in $\bH$ under $\rho$.   Let $E$ be the splitting field of $\bT$ over $K$, and suppose that $I_E := \Gal(E/K)$ is cyclic.
We have that  $\X_*(\bT)_{I_E}$ and $\X_*(\bT')_{I_E}$ have the same (finite) cardinality.
 \end{lemma}

 \begin{proof}
 Since \(k\) has characteristic zero, the isogeny is separable.
Since \(\bG\) is \(K\)-split and \(K\)-simple, its absolute root system is irreducible.

By assumption, \(I_E\) is cyclic, say of order \(m\).  Choose a generator \(\tau\).

 We have an exact sequence
\[
\begin{tikzcd}
0 \arrow{r}&\X_*(\bT) \arrow{r}{\rho_*} &\ \X_*(\bT') \arrow{r} &\Lambda \arrow{r}&0,
\end{tikzcd}
\]
and $\Lambda$ is a quotient of the (absolute) fundamental group \(\X_*(\bT')/\Z\Phi^\vee(\bH, \bT')\) of $\bH$, hence a subquotient of the fundamental group of the absolute root system \(\Phi(\bH, \bT')\). The short exact sequence above yields the long exact sequence in homology
\[
\begin{tikzcd}
\cohom_1(I_E, \X_*(\bT')) \arrow{r} & \cohom_1(I_E, \Lambda) \arrow{r} & \X_*(\bT)_{I_E} \arrow{r}  & \X_*(\bT')_{I_E} \arrow{r} &\Lambda_{I_E} \arrow{r}&0.
\end{tikzcd}
\]
Since $I_E$ is cyclic of order \(m\) with generator \(\tau\), for every $\Z$-module $A$ on which $I_E$ acts we have $\cohom_1(I_E,A) = A^{I_E}/N_{I_E}[A]$  where $N_{I_E}(a) = \sum_{j=0}^{(m-1)} \tau^j(a)$ for $a \in A$.   Since $\bT$, and so $\bT'$, is $K$-anisotropic, we conclude that
$\cohom_1(I_E, \X_*(\bT'))  = 0$.   Thus, the long exact sequence becomes
\[
\begin{tikzcd}
0 \arrow{r} & \Lambda^{I_E}/N_{I_E}[\Lambda] \arrow{r} & \X_*(\bT)_{I_E} \arrow{r}  & \X_*(\bT')_{I_E} \arrow{r} &\Lambda_{I_E} \arrow{r}&0.
\end{tikzcd}
\]
Since
\[
\begin{tikzcd}
0 \arrow{r} & \Lambda^{I_E} \arrow{r} & \Lambda \arrow{r}{1 - \tau} & \Lambda \arrow{r}  &  \Lambda_{I_E} \arrow{r}&0
\end{tikzcd}
\]
is exact, we conclude that $\dabs{\Lambda^{I_E}} = \dabs{\Lambda_{I_E} }$.
  Thus, the index of $\X_*(\bT)_{I_E}$ in $ \X_*(\bT')_{I_E}$  is $\dabs{N_{I_E}[\Lambda]}$.

 Since $\bG$, hence $\bH$, is $K$-split, it must be the case that there is some elliptic element $w \in \absW$ such that the action of $\Gal(E/K)$ on $\X_*(\bT')$ satisfies $\tau (\lambda) = w(\lambda)$ for all $\lambda \in \X_*(\bT')$.  Then $m$ is the order of $w$, so Proposition \ref{prop:elliptic-order} gives that the order of every element of $\Lambda$ divides $m$.

Moreover, the action of \(I_E\) on the fundamental group of the absolute root system \(\Phi(\bH, \bT')\), hence on its subquotient \(\Lambda\), is trivial.
Thus \(N_{I_E}(\Lambda)\) is multiplication by \(m\), which, by assumption, divides the order of every element of \(\Lambda\); so \(N_{I_E}(\Lambda)\) equals \(0\).
\end{proof}

\begin{remark}
If $I_E$ is not cyclic, then  the cardinality of
 $\X_*(\bT)_{I_E}$ and  $\X_*(\bT')_{I_E}$ need not be the same, even if \(\Lambda\) is cyclic.   For example, suppose that $k$ has characteristic zero, the residue field of $K$ has characteristic $3$, and  $E$ is a Galois extension of $K$ with Galois group $S_3$.   Let $\bT$ be the maximal $k$-torus in $\mathrm{SL}_3$
determined by the standard action of $S_3$ on
$$\X_*(\bT) = \{\sum a_i e_i \, | \, a_1, a_2, a_3 \in \Z \text{ with } a_1 + a_2 + a_3 = 0\}.$$
Here $(e_i \,|\, 1 \leq i \leq 3)$ is the standard basis of $\Z^3$.
Specifically, \(\bT\) may be taken to be \(\ker(\operatorname{Res}_{E'/k}\operatorname{GL}_1 \to \operatorname{GL}_1)\), where \(E'\) is the fixed field of any one of the order-\(2\) elements of \(I_E\).
Let $\bT'$ denote the corresponding $k$-torus in $\mathrm{PGL}_3$ under the isogeny $\rho \colon \mathrm{SL}_3 \rightarrow \mathrm{PGL}_3$.  Then $\X_*(\bT)_{I_E}$ is trivial while $\X_*(\bT')_{I_E}$ is isomorphic to $\mathbb{Z}/3\mathbb{Z}$.
\end{remark}

\section{On \texorpdfstring{$-1$}{-1} invariants and Tits groups of a split simply connected group}  \label{sec:appendixone}

\begin{center}
   Ram Ekstrom
\end{center}
We continue to use the notation of the main body of this paper. Suppose $\bG$ is a simply connected $K$-split group.  This means that $\Esplits = K$ and $\barsigma = 1$.  Since $\bG$ is simply connected, $\X_*(\Ebtorus)$ is spanned by the coroots, ${\Phi}^\vee = \Phi^\vee(\bG,\Ebtorus)$.

Suppose $-1$ belongs to $\absW$ and let $n \in \titsW$ be a representative of $-1$.     In this note we investigate when the group of $n$-fixed points of $N = N_{\Egroup}(\Ebtorus)$ is a Tits group (with respect to some pinning).  We assume $n$ is tame.  Thanks to the next lemma, this  is equivalent to assuming $p$ is not two.

\begin{lemma}
$n^4 = 1$.
\end{lemma}

\begin{proof}
Since $(-1)^2 = 1$, we know $n^2 \in \titsW \cap \EbtorusErat$.   Thus, there exists $\mu \in \Z {\Phi}^{\vee}$ such that $n^2 = \mu(-1)$.  The result follows.
\end{proof}

Let $\ell$ denote the order of $n$; so $\ell$ is $2$ or $4$.
Let $\xi$ be $-1$ or $i$ depending on whether $\ell$ is  two or four.

\subsection{The structure of \texorpdfstring{$N^{nt}$ for any $t \in \EbtorusErat_0$}{nt fixed points of N for any t in the rational split torus}}
Recall that $\bG$ is simply connected and $K$-split.
Let $N  = N_{\Egroup}(\Ebtorus)$.

\begin{lemma}  \label{lem:allnbasicallysame}
For all $t \in \EbtorusErat$, the elements $n$ and $n t$ have the same order.  In fact, we have $(n t)^2 = n^2$.
\end{lemma}

\begin{remark}
    Since $n$ is elliptic, $n$ and $nt$ are conjugate by an element of $\EbtorusErat$.  Hence $n$ and $nt$ have the same order.
\end{remark}

\begin{proof}
Note that
$$(n t)^2 =\lsup{n} t \cdot \lsup{n^2}t \cdot n^2 =  t\inv \cdot  t \cdot n^2 = n^2.$$
Since $n$ has order 2 or 4, the result follows.
\end{proof}

\begin{lemma}   \label{lem:natureofminusonefixed}
Fix $t \in \EbtorusErat_0$.
There is a pinning $(\bG, \Ebtorus, \bB, \{X^t_a\}_{a \in \Esimple})$ such that if
$n_a^t$ is the unique element of
$$N \, \bigcap \, \bU_{-a}(\Esplits) \exp(X^t_a) \bU_{-a}(\Esplits) = N \, \bigcap \,  \exp(X^t_a) \bU_{-a}(\Esplits) \exp(X^t_a)$$
for $a \in \Esimple$, then
$$N^{n t} = \langle  n^t_a \colon  a \in \Esimple \rangle.$$
That is, $N^{n t}$ is the Tits group for the pinning  $(\bG, \Ebtorus, \bB, \{X^t_a\}_{a \in \Esimple})$.
\end{lemma}

\begin{proof}
If $a \in \Esimple$, then
    \begin{align*}
        \lsup{nt}n_a = \check{a}(a(t))^{-1} \cdot \lsup{n}n_a.
    \end{align*}
Recall that $n_a$ is the unique element of $N \cap \bU_{-a}(\Esplits) \exp(X_a) \bU_{-a}(\Esplits)$. Because $\lsup{n} n_a$ and $n_a$ have the same image in $\absW$ and $n$ normalizes $\bG_a$, we have that  $n_a^{-1}\cdot\lsup{n}n_a$ lies in both $\Egroup_a \cap \EbtorusErat$ and $\titsW \cap \EbtorusErat$. So $n_a^{-1}\cdot \lsup{n}n_a$ is in the image of $\check{a}$ and has order either one or two. In the former case $n_a$ is fixed by $n$ while in the latter $\lsup{n}n_a = n_a \check{a}(-1)$.
If $\lsup{n}n_a = n_a$, then
define $y_a^t \in R_{K}$ to be a square root of $a(t)$.    If  $\lsup{n}n_a = n_a \check{a}(-1)$, then define $y_a^t$ to be a square root of $-a(t)$.
 Observe that these elements exist since the polynomial $X^2 - a(t)$ has roots in $R_{\tilde{K}}^\times = R_K^\times$.  This is because,  modulo $\wp_{\Esplits}$, $X^2 - a(t)$ has simple roots over the algebraically closed residue field, $p$ is not equal to two, and $\Esplits$ is Henselian. Put $X^t_a = (y_a^t)^{-1} \cdot X_a$ and note that $n^t_a = \lsup{\check{a}(f_a^t)}n_a = n_a \check{a}(y^t_a)$ where $f_a^t$ is a square root of $(y_a^t)^{-1}$. We then calculate that if $\lsup{n}n_a = n_a$, then
    \begin{align*}
        \lsup{nt}n^t_a = \check{a}(a(t))^{-1}\cdot n_a\cdot \check{a}(y_a^t)^{-1} = n_a \cdot (\check{a}(a(t))\cdot  \check{a}(y_a^t)^{-1}) = n_a\check{a}(y^t_a) = n_a^t.
    \end{align*}
While if $\lsup{n}n_a = n_a \check{a}(-1)$, then we similarly find
    \begin{align*}
        \lsup{nt}n_a^t = n_a \cdot ( \check{a}(a(t)) \cdot \check{a}(-1) \cdot \check{a}(y_a^t)^{-1}) = n_a\check{a}(y_a^t) = n_a^t.
    \end{align*}
Consequently the group $\titsW^t =  \langle  n^t_a \colon  a \in \Esimple \rangle$ is a Tits group for the pinning  $(\bG, \Ebtorus, \bB, \{X^t_a\}_{a \in \Esimple})$ and  is fixed by $nt$. Since $N = \titsW^t \EbtorusErat$ and $\bG$ is simply connected, we find
    \[
        N^{nt} = \titsW^t \cdot (\EbtorusErat)^{nt} = \titsW^t \cdot (\EbtorusErat)^n = \titsW^t. \qedhere \]
\end{proof}

\begin{remark}  \label{rem:useofminusonefixed}
    Note that  $N^{n} = \titsW^1$, and it can happen that $\titsW^1 \neq \titsW$.
\end{remark}

\begin{remark}
Simply connectedness is necessary in the above lemma. For example for the non-simply connected group $\bG = \PGL_2$ with $\Ebtorus$ the diagonal torus, consider the matrices $$n =\begin{bmatrix}
0 & 1 \\
1 & 0
\end{bmatrix}, \, t_0 = \begin{bmatrix} -1 & 0 \\ 0 & 1 \end{bmatrix}.$$
Then $n=n_a$ is a nontrivial representative of the Weyl group $N_{\tilde{G}}(\Ebtorus)/\EbtorusErat$ and one calculates that $(N_{\tilde{G}}(\Ebtorus))^n$ is the four element group generated by $n$ and $t_0$. On the other hand, preserving the notation of the above lemma, one checks that since $n_a$ is fixed by $n_a$ we have that $y^1_a$ can be taken to be $1$, and $n_a^1 = n_a \check{a}(y^1_a) = n_a$. Consequently $\langle  n^1_a \colon  a \in \Esimple \rangle = \{1, n\}$ is a Tits group and is a proper subgroup of $(N_{\tilde{G}}(\Ebtorus))^n$.
\end{remark}

\subsection{Results about \texorpdfstring{$n$}{n}}

\begin{lemma}  \label{lem:A.2.1}
For all $t \in \EbtorusErat$ we have that
 $n^2 = (nt)^2$ is central.
\end{lemma}

\begin{proof}
Fix $t \in \EbtorusErat$.   From Lemma~\ref{lem:allnbasicallysame}  we know that $n^2 = (nt)^2$.
It will be enough to show that for all $a \in \Esimple$ we have $\lsup{(nt)^2}X^t_a = X^t_a$.

Fix $a \in \Esimple$.
Since $\lsup{nt}n^t_a = n^t_a$ and $\Ad({(nt)^2}) \bU_{-a}(\Esplits) = \bU_{-a}(\Esplits)$, we have
$$N_{\Egroup}(\Ebtorus) \, \bigcap \,  \bU_{-a}(\Esplits) \exp(X^t_a) \bU_{-a}(\Esplits)  =
N_{\Egroup}(\Ebtorus) \, \bigcap \, \bU_{-a}(\Esplits) \exp(\lsup{(nt)^2}X^t_a) \bU_{-a}(\Esplits).$$
A calculation shows this can happen if and only if $\lsup{(nt)^2}X^t_a = X^t_a$.
\end{proof}

\begin{cor} \label{cor:A.2.2}
If $n$ has order four, then for all $t \in \EbtorusErat$ we have $n^2= (n t)^2 = \lambda_n(-1)$.
\end{cor}

\begin{proof}
Recall from Section~\ref{sec:markssec} that there is an $h \in \Egroup$ such that $h\inv n h = \lambda_n(i)$.   Thus,
$$\lambda_n(-1) = (h\inv n h)(h\inv n h) = h\inv n^2 h.$$
Since $n^2$ is central, the conclusion follows.
\end{proof}

\begin{lemma} \label{lem:Frfixed}
    If $\bG$ is $k$-quasi-split, $K$-split, then we may assume $\Fr(n) = n$.
\end{lemma}
\begin{proof}
If $\bG$ is $k$-quasi-split and $K$-split, then we may assume $\{X_a\}_{a \in \Esimple}$ is a Chevalley-Steinberg system for $\bG$ (see \cite[Section 2.9]{kaletha-prasad:bruhat-tits}).   Note that we then have $\Fr(n_a)= n_{\Fr(a)}$ for all $a \in \Esimple$.

If $w \in \absW$ has a reduced decomposition $w = s_{a_1} \cdots s_{a_r}$ where $a_i \in \Esimple$, then the lift $n_w$ of $w$ to $\titsW$ defined by $n_w = n_{a_1} \cdots n_{a_r}$ is independent of the reduced decomposition of $w$ \cite[Proposition 2.9.11]{kaletha-prasad:bruhat-tits}. It thus follows that if $-1 \in \absW$ has a reduced decomposition $-1 = s_{a_1} \cdots s_{a_m}$, then $\Fr(n) = \Fr(n_{a_1} \cdots n_{a_m}) = n_{\Fr(a_1)} \cdots n_{\Fr(a_m)} = n$ since $s_{\Fr(a_1)} \cdots s_{\Fr(a_m)}$ is another reduced decomposition of $-1$. \qedhere
\end{proof}

\begin{remark}
If $\bG$ is not quasi-split, then it may be the case that $n$ cannot be Frobenius fixed. As an example, let $D$ be a central division algebra over $k$ of index $2$ and let $\bG = \mathrm{SL}_{1,D}$. Then our $n$ fixes exactly one vertex of the Bruhat-Tits building $\BB(G)$ while $\Fr$ fixes exactly the barycenter $x$ of an alcove $C$. This latter fact follows from the fact that since $\{x\}$ is the underlying set of $\BB(\bG, k)$, the reductive quotient $\bfG_x$ is defined over $\ff$ and therefore is quasi-split by Lang's theorem. And so a Borel subgroup $\bfB \subset \bfG_x$ defined over $\ff$ is a torus since $\bG$ is anisotropic, whence $\bfG_x$ is an $\ff$-torus. This implies that $x$ is an interior point of $C$ (necessarily the barycenter). As a consequence one concludes that $n$ cannot be Frobenius fixed.
\end{remark}

\subsection{Results about \texorpdfstring{$n_F$}{nF}}
Recall that $n_\Fr$ and $n_F = \lsup{\lambda_n\inv(\pi)}n_\Fr$ are defined in Definitions~\ref{defn:nfrob} and~\ref{defn:nfrob2}.
  By construction, $n_F \in N_{\bG(K_n)_{x_0,0}}(\Ebtorus)$ normalizes $N^{n \lambda_n(\xi)}$.  In this section we assume $\bG$ is $K$-split, $k$-quasi-split and simply connected.

From Lemma~\ref{lem:Frfixed}, we may and do assume $\Fr(n) = n$.
From Lemma~\ref{cor:acomputationofFr2} there exists $z \in Z_0 = \tilde{Z}_0$  such that
$\Fr \left[ N_q (n z \lambda_n(\xi)) \right] = n_F\inv (n z \lambda_n(\xi)) n_F$.  Since $-1 \in W$,  the center of $\bG$ is an elementary abelian two group and so
 $z^2 = 1$.   Since $Z_0 = Z_{0^+}$ and $p \neq 2$, we conclude that $z = 1$.   Thus  $\Fr \left[ N_q (n  \lambda_n(\xi)) \right] = n_F\inv (n \lambda_n(\xi)) n_F$.

\begin{lemma}
We have  $n_F \in N^{n \lambda_n(-1)}$.
\end{lemma}

\begin{proof}
Since $n^2$ is central, we have $\Ad(n) \lambda_n(\xi) \cdot \Ad(n^2) \lambda_n(\xi)$ is trivial.  Thus, $N_q(n \lambda_n(\xi)) = \Ad(n^q)(\lambda_n (\xi)) \cdot n^q$.

Suppose first that $n^2 = 1$.  In this case we have $\xi = -1$, and so
$N_q(n \lambda_n(-1)) = \Ad(n^q)(\lambda_n (-1)) \cdot n^q = \Ad(n) \lambda_n(-1) \cdot n = \lambda_n(-1) n = n \Ad(n\inv) \lambda_n(-1)$.   Thus,
from Corollary~\ref{cor:acomputationofFr2} we conclude that
$$ n \lambda_n(-1)   =  n \lambda_n (-1)  = \Fr( N_q( n \lambda_n(-1) ) = (n_F)\inv \cdot (  n \lambda_n(-1) ) \cdot  n_F.$$
Consequently, $n_F \in N^{n \lambda_n(-1)}$.

Suppose now that $n$ has order four. In this case we can take $\xi = i$, and so
$N_q(n \lambda_n(i)) = \Ad(n^q)(\lambda_n (i)) \cdot n^q$.   We have two cases: $q  \equiv 1  \, (4)$ and $q  \equiv 3  \, (4)$

Suppose first that $q  \equiv 1  \, (4)$.
In this case $\Fr(i) = i$ and
$N_q(n \lambda_n(i)) = \Ad(n)(\lambda_n (i)) \cdot n = n \lambda_n (i)$.
Thus, from Corollary~\ref{cor:acomputationofFr2} we conclude that
$$ n \lambda_n(i) = n \lambda_n (i)  = \Fr[ N_q(( n  \lambda_n(i)))] = (n_F)\inv \cdot ( n \lambda_n(i) ) \cdot  n_F,$$
and so  $n_F \in N^{n \lambda_n(i)}$.

Suppose now that $q  \equiv 3  \, (4)$.  In this case $\Fr(i) = -i$ and
$N_q(n \lambda_n(i)) = \Ad(n^3)(\lambda_n (i)) \cdot n^3$.  From Lemma~\ref{lem:A.2.1} and Corollary~\ref{cor:A.2.2}, $n^2 = \lambda_n(-1)$ is central, so
$$N_q(n \lambda_n(i)) = \Ad(n) \lambda_n (i) \cdot n \lambda_n(-1) = \lambda_n(-i) n \lambda_n(-1) = n \cdot \Ad(n\inv) \lambda_n(-i) \cdot \lambda_n(-1) = n \lambda_n(-i).$$
Thus,
from Corollary~\ref{cor:acomputationofFr2} we conclude that
$$n \lambda_n(i)  = \Fr [ n \lambda_n(-i)]  = \Fr[N_q (n \lambda_n(i))]= (n_F)\inv \cdot (  n \lambda_n(i) )\cdot n_F,$$
and so,  $n_F \in N^{n \lambda_n(i)}$.
\end{proof}

The following result now follows from the above work and Corollaries~\ref{cor:scellipticsummation} and~\ref{cor:scellipticsummation2}.

\begin{cor}
 Suppose $\bG$ is a simply connected, $K$-split, $k$-quasi-split group and  $-1$ belongs to $\absW$.  Let $n \in \titsW$ be a lift of $-1$ such that $\Fr(n) = n$.   The set of $k$-stable classes in $\CTk$ is indexed by the conjugacy classes in $\absW$.
    If $\bT'$ is a $k$-torus in the conjugacy-class corresponding to $w' \in \absW$,  then, up to $G^{\Fr}$-conjugacy, the set of  $k$-embeddings of $\bT'$ into $\bG$ is indexed by $\bshA{}^{-1}_{w'  \Fr}$. Finally, the set of $\bG(k)$-classes in $\CTk$ is indexed by $N^{n \lambda_n(\xi)}_{\sim Fr}$.
 \qed
\end{cor}

\begin{remark}  This Corollary and Lemma~\ref{lem:natureofminusonefixed}  were used in Example~\ref{ex:gtwo} to compute the number of rational classes in $\CTk$ for $\bT$ corresponding to $-1$ in the Weyl group of $\Gtwo$.  We also used this Corollary to double-check the results about the tori that correspond to $-1$ for $\SL_2$ and $\Sp_4$  in Examples~\ref{ex:sltwo} and  \ref{ex:sp4}.
\end{remark}

\end{document}